\documentclass[ijoc,nonblindrev]{informs4}
\usepackage{eqndefns-left}
\usepackage{mathalpha}
\usepackage{bbm}
\usepackage{mathrsfs}
\usepackage{float}
\usepackage{amsmath}
\usepackage{algorithm}
\usepackage{mathtools}
\usepackage{algpseudocode}
\usepackage{empheq} 
\usepackage{amsmath, mathtools} 
\usepackage{empheq}
\usepackage{bm}
\usepackage{enumitem}
\usepackage{caption}
\usepackage[labelfont=sf]{subcaption}
\OneAndAHalfSpacedXI 

\usepackage{algorithm}
\usepackage{algpseudocode}
\usepackage{tikz}

\usepackage{natbib}
 \bibpunct[, ]{(}{)}{,}{a}{}{,}%
 \def\bibfont{\small}%
\EquationsNumberedThrough    

\TheoremsNumberedThrough     
\ECRepeatTheorems  %

\MANUSCRIPTNO{MOOR-0001-2024.00}

\begin{document}



\RUNAUTHOR{Chen et al.}

\RUNTITLE{Joint Planning and Operations of Wind Power under Decision-dependent Uncertainty}

\TITLE{Joint Planning and Operations of Wind Power under Decision-dependent Uncertainty}

\ARTICLEAUTHORS{%
\AUTHOR{Zhiqiang Chen,\textsuperscript{a} Caihua Chen,\textsuperscript{a,*} Jingshi Cui,\textsuperscript{a} Qian Hu,\textsuperscript{a} Wei Xu\textsuperscript{a,*}}
\AFF{
\textsuperscript{a}School of Management and Engineering, Nanjing University, Nanjing 210008, China\\
\textsuperscript{*}Corresponding authors \\
\textbf{Contact:} \EMAIL{chenzq@smail.nju.edu.cn}; \EMAIL{chchen@nju.edu.cn}; \EMAIL{jscui@nju.edu.cn}; \EMAIL{huqian@nju.edu.cn}; \EMAIL{xuwei@nju.edu.cn}
} 


} 

\ABSTRACT{%
We study a joint wind farm planning and operational scheduling problem under decision-dependent uncertainty. The objective is to determine the optimal number of wind turbines at each location to minimize total cost, including both investment and operational expenses. Due to the stochastic nature and geographical heterogeneity of wind power, fluctuations across dispersed wind farms can partially offset one another, thereby influencing the distribution of aggregated wind power generation—a phenomenon known as the smoothing effect. Effectively harnessing this effect requires strategic capacity allocation, which introduces decision-dependent uncertainty into the planning process. To address this challenge, we propose a two-stage distributionally robust optimization model with a decision-dependent Wasserstein ambiguity set, in which both the distribution and the radius are modeled as functions of the planning decisions, reflecting the statistical characteristics of wind power resources. Then, we reformulate the model as a mixed-integer second-order cone program, and the optimal objective value provides a probabilistic guarantee on the out-of-sample performance. To improve computational efficiency, we develop a constraint generation based solution framework that accelerates the solution procedure by hundreds of times. Numerical experiments using different datasets validate the effectiveness of the solution framework and demonstrate the superior performance of the proposed model.

}%





\maketitle


\vspace{-0.5cm}

\section{Introduction}\label{sec:Intro}

Wind power has become a pivotal element in the global shift toward renewable energy sources, effectively addressing the dual challenges of climate change mitigation and reducing dependence on fossil fuels (\citealp{bitaraf2017reducing, iea_renewables_2023}). In 2023, global wind turbine installations reached 117 GW, marking a 50\% increase from the previous year \citep{gewc_global_wind_report_2024}. This remarkable growth is largely attributed to advancements in technology, which now allow the investment in wind energy to generate 2.5 times more energy output compared to a decade ago \citep{iea_world_energy_investment_2024}. Consequently, both onshore and offshore wind sectors have attracted significant manufacturing investments, aimed at harnessing this abundant and clean energy source \citep{bnef_energy_transition_investment_trends_2024}. These developments underline the increasing necessity for strategic capacity planning across multiple wind farms.

Nonetheless, multi-farm capacity planning presents complex challenges, especially when integrated with subsequent operational considerations (\citealp{yin2022stochastic, yin2022coordinated}). These challenges arise from the inherent uncertainty associated with wind power generation, necessitating a balance between cost savings achieved through renewable energy production and expenses incurred in managing such uncertainty. On the one hand, exogenous uncertainty, caused by factors such as geographical location, atmospheric conditions, and seasonal variations \citep{tong2010fundamentals}, introduces substantial variability in wind power resources. These fluctuations complicate long-term operational planning and must be addressed to ensure reliable and cost-effective integration of wind energy into the power system. On the other hand, endogenous uncertainty, driven by the dependence of wind power generation on capacity planning decisions, adds further complexity to the planning process. The random fluctuations in output from geographically dispersed wind farms, influenced by local wind conditions, can partially counterbalance one another, which is a phenomenon referred to as the wind power smoothing effect (\citealp{asari2002study, yang2019investigating}). Fully leveraging this effect requires thoughtful and strategic capacity allocation to optimize the overall system performance, thereby introducing the decision-dependent uncertainty into the planning framework.

To address these uncertainties, data-driven optimization methods are crucial. However, the construction and operation of meteorological towers, approximately \$50,000 per year for onshore towers and several million dollars for offshore installations \citep{wind_resource_measurement}, pose significant barriers, particularly in newly planned or underdeveloped regions. Additionally, global climate change has altered wind energy resources \citep{martinez2024global}, rendering only recent data reliable for planning purposes. As a result, the lack of sufficient long-term observations on wind speed and power generation hinders wind energy development (\citealp{gewc_global_wind_report_2017, saeed2021wind, wang2024wind}). Under such conditions, stochastic optimization (SO) often struggles to generalize and perform effectively out-of-sample due to data scarcity. Meanwhile, robust optimization (RO) tends to be overly conservative, as it relies solely on support information without incorporating distribution information \citep{zhu2019wasserstein}. These limitations reduce the effectiveness of both approaches in addressing the complexities of wind farm planning problems. 

These limitations of SO and RO motivate the exploration of data-driven distributionally robust optimization (DRO), which considers the worst case distribution within an ambiguity set. In particular, we utilize an empirical distribution derived from historical data to construct a Wasserstein-based ambiguity set, which can quantify the bias between the empirical distribution and the true distribution while accounting for uncertainties beyond those captured by the empirical distribution (\citealp{mohajerin2018data, gao2024wasserstein}). We emphasize that our work is not a trivial application of Wasserstein DRO (WDRO). Importantly, our research extends conventional decision-independent ambiguity sets to a decision-dependent framework, where both the distribution and the radius are explicitly modeled as functions of the planning decisions. Existing research on decision-dependent ambiguity sets typically relies on discrete distributions (\citealp{noyan2022distributionally, doan2022distributionally}) or constant radii \citep{noyan2022distributionally}, which limit applicability to wind farm planning problems. Moreover, directly extending these methods introduces modeling and tractability issues, such as bilinear terms \citep{basciftci2021distributionally}. To overcome these challenges, we specify distribution functions and radius function based on the characteristics of wind power resources. We adopt a two-stage DRO framework to formulate a joint planning and operations model, thereby avoiding bilinear complexities. Furthermore, we reformulate the model as a mixed-integer second-order cone programming (MISOCP) and develop an efficient exact solution framework. 

More specifically, the main contributions of our work can be summarized as follows:

\begin{enumerate}
    \item We consider the joint planning and operations problem to capture the smoothing effect among distributed wind farms. To hedge against the wind power uncertainty, we formulate the problem as a two-stage DRO model under the Wasserstein ambiguity set. 
    \item We propose a novel decision-dependent Wasserstein ambiguity set, where the distribution and the radius are functions of the planning decisions. The distribution functions are tailored to accurately reflect the characteristics of wind farm generation. Furthermore, the radius functions are derived through statistical methods, ensuring that the ambiguity set provides a probabilistic guarantee that the true distribution lies within it at a specified confidence level.
    \item By leveraging the specific forms of ambiguity sets, we derive a MISOCP reformulation of the proposed model. Theoretically, the optimal objective value obtained from this reformulation guarantees out-of-sample performance of the planning decisions with a defined confidence level.
    \item We develop an efficient solution framework that exploits the underlying structure of the problem. This framework integrates a constraint generation approach with the L-shaped algorithm, significantly enhancing computational efficiency. Furthermore, our numerical experiments reveal substantial speed improvements, particularly as the sample size increases.
    \item Utilizing different datasets, we conduct comprehensive numerical experiments to evaluate the out-of-sample performance metrics of our method against several benchmarks. The results demonstrate notable reductions in fluctuations and operational risk, underscoring the superior performance of the proposed model across diverse scenarios.
\end{enumerate}

\vspace{-0.2cm}

\subsection{Organization}
The rest of the paper is structured as follows. In Section \ref{sec:model}, we present the two-stage distributionally robust wind farm planning model. Section \ref{sec:ddas} introduces the decision-dependent Wasserstein ambiguity set and then provides its specific formulation. This formulation is then used in Section \ref{sec:mr} to reformulate the model as a MISOCP problem. In Section \ref{sec:cgbsf}, we propose a constraint generation based solution framework. Section \ref{sec:ne} reports computational experiments to evaluate the performance of the solution framework and the proposed model. Finally, Section \ref{sec:concl} concludes the paper.

\subsection{Notation}
We let $\mathbb{R}$ denote the real numbers, and $\|\cdot\|$ denote a general norm. Boldface lowercase and uppercase letters represent vectors and matrices, respectively. The element-wise product between two vectors $\bm u, \bm v$ is denoted as $\bm u \cdot \bm v$. For an integer $n \geq 1$, let $[n]$ denotes the set $\{1,2,\ldots,n\}$. For any $\xi \in \mathbb{R}$, $\delta_{\xi}$ denotes the Dirac measure at $\xi$. We use $[t]_+$ to denote $\max\{0,t\}$. Let $\mathcal{P}(\Xi)$ be the set of all probability measure on support $\Xi \subset \mathbb{R}^n$, and $\mathbb{P} \in \mathcal{P}(\Xi)$ denotes a generic distribution.
\vspace{-0.2cm}

\section{Joint Planning and Operations Model}\label{sec:model}

This section begins with an introduction to the two-stage dispatch model used in the operational phase, including key simplifications. We then describe the joint planning and operations problem and associated uncertainties. Lastly, we present the mathematical formulation of the overall model.

\subsection{Two-stage Dispatch Model}\label{subsec:tdm}

Due to the significant uncertainties in power systems, particularly stemming from the variability of renewable generation and the fluctuation of electricity demand, a two-stage dispatch framework is widely adopted in practice: day-ahead scheduling and real-time scheduling.

In the first stage, referred to as day-ahead scheduling, the system operator determines the planned generation output and allocates reserve capacities, which are margins of flexible generation that can be called upon to respond to real-time deviations. These decisions are made under forecast information of demand and renewable output.
In the second stage, real-time scheduling, adjustments are made based on the actual realizations of demand and renewable generation. Generators respond by adjusting their outputs within their reserve capacities to maintain system balance, i.e., ensuring that total generation equals total demand at all times.

To focus on the essential modeling challenges and reduce computational complexity, we adopt several standard simplifications:

\begin{itemize}
    \item Direct current (DC) power flow approximation: We employ the DC power flow model, which linearizes the alternating current (AC) power flow equations by neglecting reactive power and voltage magnitudes.
    
    \item All thermal units are online: All thermal generators are assumed to be committed and operating throughout the scheduling horizon, eliminating the need to model unit commitment decisions explicitly.
    
    \item Reserve provision: Each thermal generator can provide both upward and downward reserve, subject to its reserve capacity limits.
    
    \item Automatic generation control (AGC): AGC is employed to adjust the output of thermal generators in real time to fully compensate for wind forecast errors, ensuring supply-demand balance under all realizations. Each generator incurs a fixed adjustment cost per unit of AGC response (adjustment output).

    \item Deterministic demand loads: Unlike the wind generation outputs, the loads are known or can be accurately predicted.
\end{itemize}

These simplifications are commonly adopted in existing power system dispatch researches involving renewable energy uncertainty (\citealp{wang2018risk, zhu2019wasserstein, esteban2023distributionally}), as they strike a balance between model fidelity and computational tractability.

\subsection{Problem Description and Uncertain Variables}\label{subsec:dpu}

Wind farm planning involves selecting suitable sites and allocating capacity in the presence of uncertain wind resources. In this context, we aim to determine the optimal capacity allocation across potential locations to minimize the total cost, which includes both investment and operational components.

Let $x_w$ denote the integer decision variable representing the number of wind turbines to be installed at location $w$. The capacity planning problem thus determines the optimal number of turbines at each site, subject to investment or wind power penetration constraints.
Given these planning decisions, a two-stage operational scheduling framework is employed to account for the uncertainty inherent in wind power generation. Importantly, wind generation may also exhibit significant variation across different climatic conditions due to exogenous factors such as seasonal changes \citep{jung2019changing}. To capture this variability, the operational process can be partitioned into multiple scenarios $s \in [S]$, each representing a distinct operational condition. For example, we could define each scenario as a season—corresponding to a three-month period within a year.
The total operational cost is computed as the aggregate cost across all scenarios. This joint planning and operations optimization framework captures the long-term value of capacity decisions by embedding their downstream impact on system operations, thereby ensuring a robust and forward-looking solution (see Figure~\ref{figure:framework}).

\begin{figure}[H]
\FIGURE
{\includegraphics[width=0.8\textwidth]{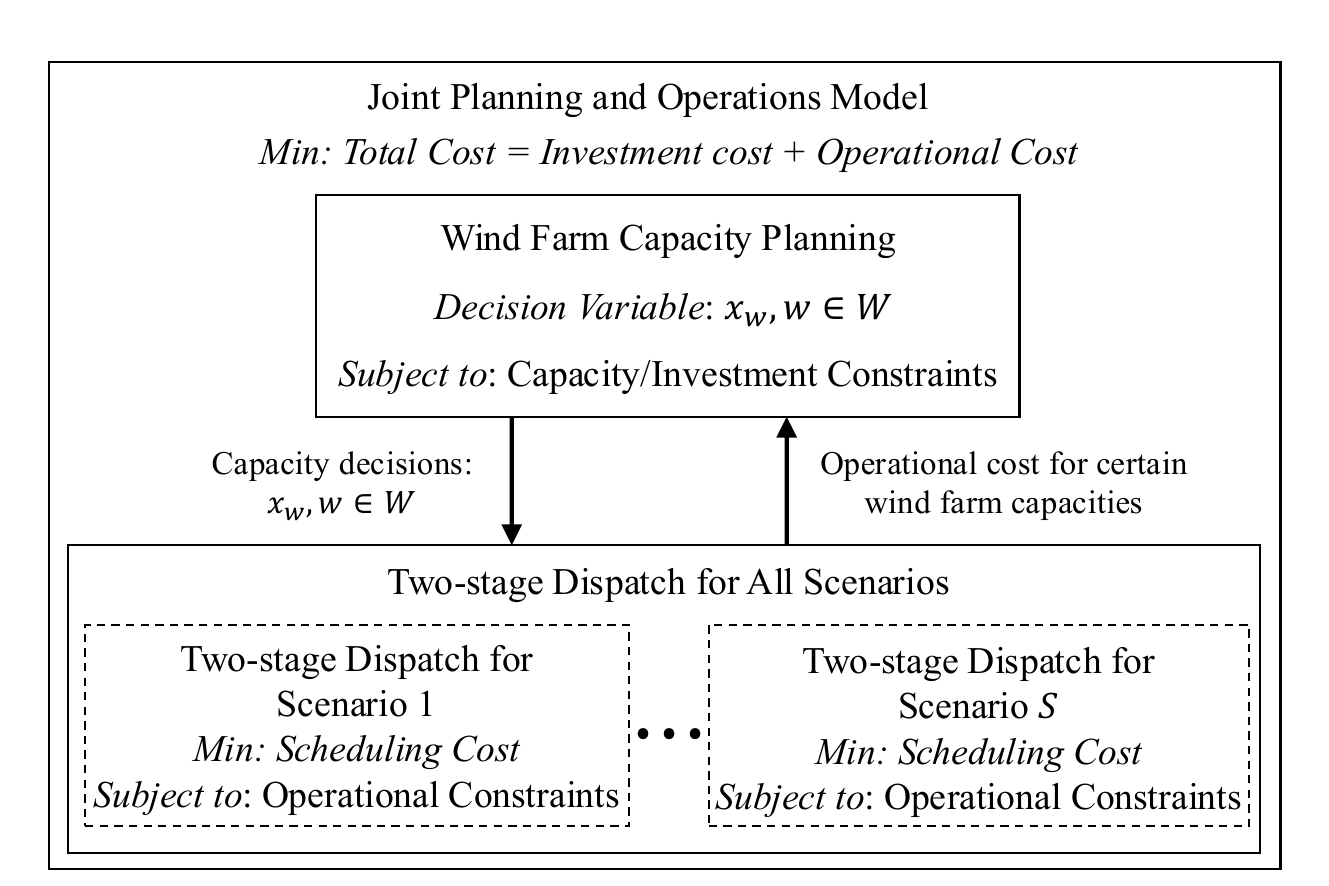}}
{Framework of the Joint Planning and Operations Model \label{figure:framework}}
{}
\end{figure}

To capture the impact of both inherent wind power resource uncertainty and planning decisions on wind farm output, we begin by defining the base unit of uncertainty.

\begin{definition}
    The output power of a single wind turbine at a specific location is defined as the wind power resource, represented by $\xi$.
\end{definition} 

Let $\bm{\xi} = (\xi_w)_{w \in [W]}$ denote the vector of wind power resources across all locations, where each $\xi_w$ is an uncertain variable corresponding to location $w$. The random wind power generation at each wind farm is denoted by $\bm{\zeta}(\bm{x}) = \big(\zeta_w(x_w)\big)_{w \in [W]}$, where $\zeta_w(x_w)=f(x_w,\xi_w)$ represents the output of a wind farm at location $w$ and its distribution depends on the decision $x_w$. This dependency also extends to any linear transformation of $\bm{\zeta}(\bm{x})$, denoted as $L(\bm{\zeta}(\bm{x}))=\sum_wa_w\zeta_w(x_w) \in \mathbb{R}$. Notably, the random variable considered in the ambiguity set, $L(\bm{\zeta}(\bm{x}))$, is one-dimensional, such as the aggregated wind power generation or transmission line power flow.

\subsection{Mathematical Formulation}\label{subsec:mf}

To describe the mathematical formulation, we first introduce the following notations. 

Let $T$, $G$, $L$ and $D$ denote the numbers of time periods, thermal generators, transmission lines and load nodes, respectively. For each scenario $s \in [S]$, period $t \in [T]$, generator $g \in [G]$ and load node $d \in [D]$, the decision variables are defined as follows: $P_{stg}$ represents the output of thermal units, $\bar r_{stg}(\underline r_{stg})$ the regulation up(down) reserve capacity, $\alpha_{stg}$ the adjustment output, $z_{stg}$ the adjustment cost, $w_{st}/l_{st}$ the amount of wind curtailment/load shedding. The parameters are defined as follows: $\Delta_s$ represents the duration of each scenario, $c_w$ the unit investment cost, $\bar \xi_{stw}$ the corresponding prediction value, $\bar \zeta_{stw}(x_w)=f(x_w, {\bar \xi}_{stw})$ the forecasted generation of farm $w$, $UR_{g}(DR_{g})$ the unit cost of upward(downward) reserve, $UA_{g}(DA_{g})$ the unit cost of upward(downward) adjustment, $RU_{g}(RD_{g})$ the ramp up(down) limit, ${\bar P}_{std}$ the demand load, ${\bar x}_w$ the capacity limits, $WC/LS$ the unit cost of wind curtailment/load shedding. For each transmission line $l \in [L]$, the parameter $\bar P_l$ represents the power flow limits, and the parameters $\pi_{gl}/\pi_{wl}/\pi_{dl}$ denote the shift distribution factors for generators, wind farms, and loads, respectively.

\begin{align}
\begin{split} \label{obj}
    P: \quad \min \quad &\sum_w{c_wx_w} +\sum_s\Big\{\sum_{t}\sum_g\left[F_g(P_{stg}) + UR_{g}\bar r_{stg} + DR_{g}\underline r_{stg}\right] +  \\[-3mm]
    &\hspace{6cm}
    \sum_{t}\sup_{{\mathbb P_{st}} \in \mathcal{P}_{st}(\bm{x})}\mathbb{E}_{\mathbb P_{st}}[h_{st}(\bm{\zeta}_{st}(\bm{x}))]\Big\}\Delta_s 
\end{split}
\end{align}
\begin{align}
\text{s.t.} \quad &\sum_gP_{stg} + \sum_w\bar \zeta_{stw}(x_w) =\sum_d{\bar P}_{std} \quad &\forall s,t \label{cos:power balance}\\
&\underline P_{g}+\underline r_{stg} \leq P_{stg} \leq \bar P_{g}-\bar r_{stg} &\forall s,t,g \label{cos:generation limit} \\
&(P_{stg} + \bar r_{stg}) - (P_{s,t-1,g} - \underline r_{s,t-1,g}) \leq RU_g &\forall s,t,g \label{cos:ramp up limit}\\
&(P_{s,t-1,g} + \bar r_{s,t-1,g}) - (P_{stg} - \underline r_{stg}) \leq RD_g &\forall s,t,g \label{cos:ramp down limit} \\[2mm]
\mathbb{P}_{stl} &\left\{ \begin{array}{l}\label{cos:tlc}
\begin{aligned}
\sum_{g}\pi_{gl}&P_{stg}+[-\pi_{gl}]_+(\bar r_{stg} + \underline r_{stg})+\pi_{gl}\underline r_{stg}   \\
&+\sum_{w}\pi_{wl}\zeta_{stw}(x_w)-\sum_{d}\pi_{dl}{\bar P}_{std} \leq \bar{P}_{l}
\end{aligned} \\
\begin{aligned}
\sum_{g}\pi_{gl}&P_{stg}-[\pi_{gl}]_+(\bar r_{stg} + \underline r_{stg})+\pi_{gl}\underline r_{stg} \\
&+\sum_{w}\pi_{wl}\zeta_{stw}(x_w)-\sum_{d}\pi_{dl}{\bar P}_{std} \geq -\bar{P}_{l}
\end{aligned} 
\end{array} \right\} \geq 1-\epsilon \ \  \forall \mathbb{P}_{stl} \in \mathcal{P}_{stl}(\bm{x})  &\forall s,t,l \\[2mm]
&P_{stg},\bar r_{stg},\underline r_{stg} \geq 0  &\forall s,t,g \label{cos:non-negative} \\
&x_w \leq {\bar x}_w & \forall w \label{cos:capacity limit}\\
&x_w \in \{0,1,2,...\} & \forall w \label{cos:integer}
\end{align}
where $F_g(P_{stg})$ is a non-decreasing quadratic function and 
\begin{align}
    SP_{st}: \quad h_{st}(\bm{\zeta}_{st}(\bm{x}))&=\min_{\alpha_{stg},z_{stg},w_{st},l_{st}}\sum_gz_{stg}+WCw_{st} + LSl_{st} \label{second_stage} \\[-1mm]
    \text{s.t.}\quad &DA_g\alpha_{stg} \leq z_{stg} &\forall g \label{subcos:obj non-negative1}\\
    &-UA_g\alpha_{stg} \leq z_{stg} &\forall g \label{subcos:obj non-negative2}\\[-1mm]
    &\sum_g\alpha_{stg} + w_{st}-l_{st} = \sum_w\zeta_{stw}(x_w) - \sum_w\bar \zeta_{stw}(x_w) \label{subcos:error balance}\\[-1mm]
    & -\bar r_{stg} \leq \alpha_{stg} \leq \underline r_{stg} &\quad \forall g \label{subcos:adjustment limit} \\[-1mm]
    & w_{st},l_{st}\geq 0  &\quad \label{subcos:non-negative}
\end{align}

The objective function (\ref{obj}) aims to minimize the total cost associated with wind power generation, encompassing investments in wind turbines, generation costs, reserve costs, and the worst-case expected adjustment costs arising from uncertainty. The constraints of the model are described as follows. Constraint (\ref{cos:power balance}) ensures power balance, maintaining that the total generated power matches the demand at all times. Constraints (\ref{cos:generation limit}), (\ref{cos:ramp up limit}), and (\ref{cos:ramp down limit}) enforce limits on generation output and ramp rates, ensuring that the system operates within the physical capabilities of and thermal generators. Constraint (\ref{cos:tlc}) imposes power flow limits across transmission lines. These are modeled as chance constraints, which allow a tolerance $\epsilon$ for violation probabilities under distributions within decision-dependent ambiguity sets. The DRO framework is applied to consider the worst-case distribution within the ambiguity set. Constraint (\ref{cos:non-negative}) ensures that the outputs from thermal generators and reserve capacities remain non-negative. Constraint (\ref{cos:capacity limit}) and (\ref{cos:integer}) impose the upper bound and integer requirement on the capacity planning decisions. Constraints (\ref{subcos:obj non-negative1}) and (\ref{subcos:obj non-negative2}) enforce non-negative adjustment costs for each thermal generator. Constraint (\ref{subcos:error balance}) manages aggregated wind power prediction errors by ensuring adjustments compensate for deviations, where slack variables $w_{st}$ and $l_{st}$ represent wind curtailment and load shedding, respectively. Constraint (\ref{subcos:adjustment limit}) limits adjustment outputs to remain within the established reserve capacities. Finally, constraint (\ref{subcos:non-negative}) enforces non-negativity for both wind curtailment and load shedding.

\begin{remark}

    The transmission line capacity constraints within chance constraints (\ref{cos:tlc}) can be explicitly derived as follows:
    \begin{equation*}
    \left\{\begin{aligned}
        \max_{-\bar r_{stg}\leq \alpha_{stg}\leq\underline r_{stg}}\sum_{g}\pi_{gl}(P_{stg}+\alpha_{stg})  
        +\sum_{w}\pi_{wl}\zeta_{stw}(x_w)-\sum_{d}\pi_{dl}{\bar P}_{std}\leq\bar{P}_{l} \\
        \min_{-\bar r_{stg}\leq \alpha_{stg}\leq\underline r_{stg}}\sum_{g}\pi_{gl}(P_{stg}+\alpha_{stg})  
        +\sum_{w}\pi_{wl}\zeta_{stw}(x_w)-\sum_{d}\pi_{dl}{\bar P}_{std}\geq-\bar{P}_{l}
    \end{aligned} \right. \text{.}
    \end{equation*}
    These constraints are separated from the inner problem $SP_{st}$ and are formulated based on the classical RO framework, which is inherently more robust compared to traditional formulation \citep{wang2018risk}. The adoption of this formulation, coupled with the constraint generation approach in Section \ref{sec:cgbsf}, significantly accelerates the solution process. Without this approach, computational time would become prohibitive, even for the smallest problem instances. This limitation makes the traditional approach impractical for real-world networks of a relatively large scale.
    
\end{remark}

\section{Decision-dependent Wasserstein Ambiguity Set}\label{sec:ddas}

In this section, we first explore the distributional information inherent in the real data through empirical analysis, and then present a specific formulation tailored to the characteristics of wind power resources.

\subsection{Empirical Analysis of Distributional Information Inherent in $\bm{\xi}$}\label{subsec:eadi}

This section focuses on the empirical analysis of real wind power data, examining the distributional information, namely log-concavity, to capture the uncertainty more accurately.

Wind power generation is significantly influenced by wind speed, which is a key determinant of the energy output from wind turbines. Wind speed is commonly modeled using log-concave distributions, such as Lognormal and Weibull, which effectively represent its variability and statistical characteristics at specific locations \citep{garcia1998fitting}.  This insight provides a solid foundation for exploring the statistical property of wind power resources. To validate this property, we conduct thorough tests to evaluate the log-concavity of the underlying distributions, ensuring that our models accurately reflect the physical dynamics of wind power generation.

The primary dataset \citep{kaggle_2024} used in our research includes detailed meteorological observations and normalized wind power generation data from four operational locations over a span of 4 years. For better organization and statistical analysis, the data is initially grouped by seasons, with each season covering a 90-day interval. Within each season, normalized wind power generation data recorded at the same time of day is further categorized to capture temporal patterns effectively. We test the log-concavity of wind power resource distributions using data spanning 4 years, resulting in a sample size of 360. The null hypothesis is set as 
\begin{equation*}
    H_0: \text{The distribution of } \bm{\xi} \text{ is log-concave}. 
\end{equation*}

Testing for log-concavity has been relatively underdeveloped \citep{gangrade2023sequential}. Recent advancements by \citet{dunn2024universal} have introduced such a test based on the Universal Inference strategy \citep{wasserman2020universal}. This approach is provably valid in finite samples under only the assumption that the data sample is independent and identically distributed. We adopt this methodology to evaluate the log-concavity of the underlying distribution.  
The testing results are compared against a critical value $1/\alpha$ corresponding to the chosen significance level $\alpha$. The null hypothesis $H_0$ is rejected if the test statistic exceeds $1/\alpha$. Across 200 repeated simulations, the test statistic is almost always less than 0.5. For intuitive visualization, we present the proportion of instances where the test statistic exceeds 0.7, a value well below any practical $\alpha \in (0,1)$. The partial results are displayed in Figure \ref{fig:logconcave test} (see Appendix \ref{sec:test_all} for comprehensive results). Even under these stringent conditions, at most one violation occurs across 200 simulations, strongly validating the log-concavity of the distribution. These findings are consistent with the prior research (\citealp{li2018distributionally}), which underscores the appropriateness of log-concave distributions for accurately modeling wind power generation.

\begin{figure}[H]
    \FIGURE
    {\includegraphics[width=0.7\textwidth]{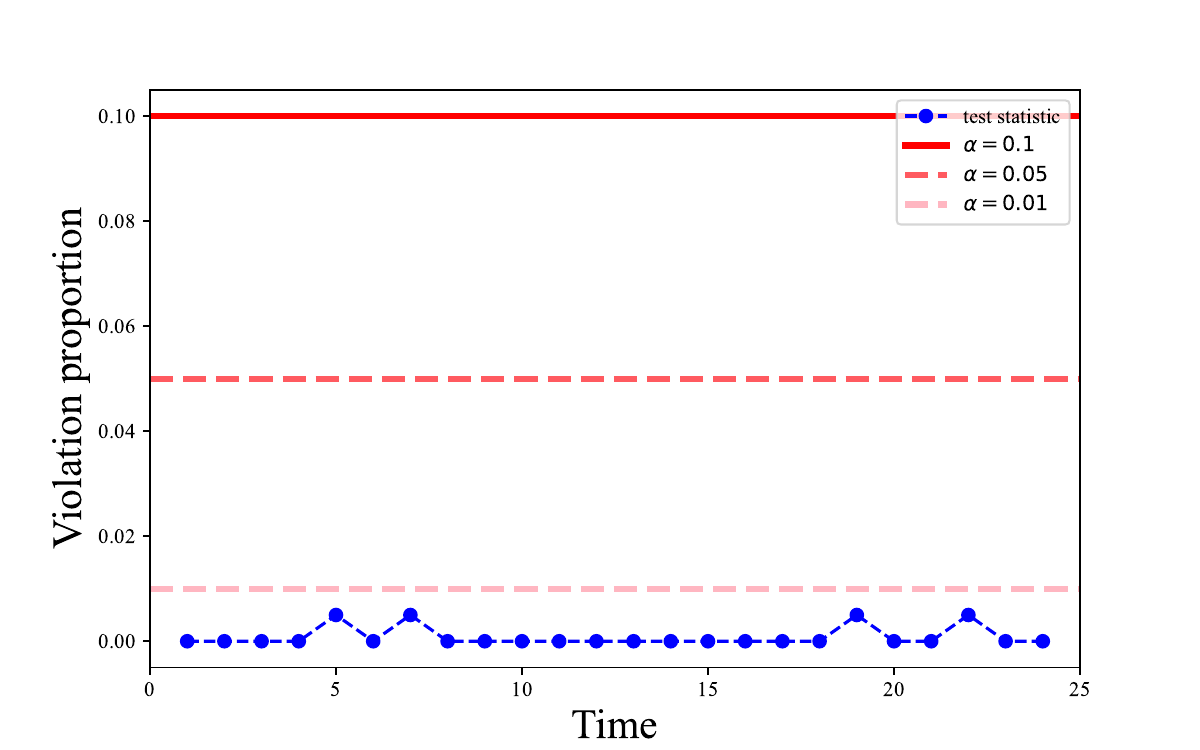}}
    {Tests Results for Log-Concavity \label{fig:logconcave test}}
    {}
\end{figure}

Given these results, we make the following assumption.
\begin{assumption} \label{assumption1}
    The distribution of $\bm{\xi}$ is log-concave.
\end{assumption}

Assumption \ref{assumption1} emphasizes the log-concavity of the distribution, ensuring that their affine transformations also retain log-concavity. This characteristic is critical for specifying the radius function of the Wasserstein ambiguity set.

\subsection{Characterization of the Ambiguity Set for $L(\bm{\zeta}(\bm{x}))$}\label{subsec:cas}

In practice, the true distribution of wind power resources is often unknown, and decision-makers typically have only historical wind power generation data at disposal. Suppose that the historical data contains $N$ observations from $W$ wind farms, these observations form an empirical distribution as follows:

\begin{equation*}
    \hat {\mathbb{P}}=\frac 1N\sum_{i\in [N]} \delta_{L(\hat{\bm{\zeta}}_i(\bm{x}))},
\end{equation*}

However, the true distribution can not be described accurately by the empirical distribution, even combined with log-concavity information. To hedge against uncertainties that are not adequately captured in the empirical distribution, we adopt a WDRO model, which is data-driven and has a strong theoretical foundation. Formally, the $p$-Wasserstein distance between two distributions $\mathbb{P}$ and $\mathbb{Q}$ on $\mathbb{R}$ is defined as
\begin{equation*}
W^p_p\left(\mathbb{P},\mathbb{Q}\right):=\inf_{\Pi\in \mathcal{P}(\mathbb{R}\times\mathbb{R})}\mathbb{E}_{\Pi}[|\tilde\xi-\xi|^p]\text{,}
\end{equation*}
where $\Pi$ is a joint distribution of $\tilde \xi$ and $\xi$ with marginals $\mathbb{P}$ and $\mathbb{Q}$, respectively. Then, the ambiguity set $\mathcal{P}$ is defined as 
\begin{equation}\label{as}
\mathcal{P}{:}=\left\{\mathbb P\in\mathcal{P}(\mathbb{R}):W_p(\hat{\mathbb P}, {\mathbb P})\leq \varepsilon\right\}\text{,}
\end{equation}
which can be interpreted as a Wasserstein ball of radius $ \varepsilon$ centered at the empirical distribution $\hat{\mathbb P}$.

In the traditional ambiguity set, the probability distribution remains fixed, assuming that the uncertainty of the random parameters is exogenous and independent of the decision-making process. This assumption restricts the applicability of such ambiguity sets in scenarios where uncertainty is influenced by decisions. To overcome this limitation and address the decision-dependent uncertainty in the planning problem, we propose the following decision-dependent variant:
\begin{equation}\label{ddas}
     \mathcal{P}(\bm{x}){:}=\left\{\mathbb P(\bm{x})\in\mathcal{P}(\mathbb{R}):W_p(\hat {\mathbb P}(\bm{x}),\mathbb P(\bm{x}))\leq \varepsilon(\bm{x})\right\} \text{,}
\end{equation}
where $\mathbb{P}(\bm{x})$ represents the distribution of $L(\bm{\zeta}(\bm{x}))$. The ambiguity set (\ref{ddas}) accounts for the influence of the decision $\bm{x}$ on the probability measure of the random variable. The functional forms of $\mathbb{P}(\bm{x})$ and $\hat{\mathbb{P}}(\bm{x})$, which capture the impact of decisions on distributions, depend on $\zeta_{w}(x_w)$. To express $\zeta_{w}(x_w)$ in a tractable manner without loss of generality, we make the following assumptions:

\begin{assumption} \label{assumption2}
The output of a wind farm is proportional to the number of wind turbines installed. 
\end{assumption}

Under Assumption \ref{assumption2}, the output of a single wind farm at location $w$ can be expressed as the product of the number of wind turbines and the wind power resource, i.e., $\zeta_{w}(x_w)=f(x_w, \xi_w)=x_w\xi_w$, a formulation commonly used in practice \citep{elia_wind_power}.

\begin{remark}
    Assumption \ref{assumption2} offers a general framework to capture variations in the output of a single wind farm. In cases where internal variability arises due to local conditions or the wake effect, our model can be naturally extended to accommodate such complexities without altering its analytical structure (see Appendix \ref{sec:em}).
\end{remark}

We next focus on the radius function of the decision-dependent Wasserstein ambiguity set based on the proposed distribution function. A significant advantage of the Wasserstein-based ambiguity set is its ability to provide a high confidence level that the true distribution lies within it. For decision-dependent variants, it is essential to ensure that this probabilistic guarantee holds for any given decision $x$, maintaining the robustness and reliability of the WDRO framework under decision-dependent uncertainty. Therefore, we begin by introducing a measure concentration inequality that establishes a connection between the Wasserstein radius and the variance of the distribution. 

\begin{lemma}\label{lemma:mc}
Assume that probability distribution $\mathbb{P}$ on $\mathbb{R}$ is log-concave. Then, for any $1\leq p<2$, $\varepsilon>\left(\frac{C_2}{2-p}\right)^{1/p}\left(\frac\sigma{\sqrt{N}}\right)$, and any $N\geq1$,
\begin{equation}\label{mc}
   \mathbb{P}\{W_p({\hat{\mathbb{P}}},\mathbb{P})\geq \varepsilon\}\leq C_1\mathrm{exp}\left\{-\sqrt\frac{N}{3} \sigma^{-1}\left[\varepsilon-\left(\frac{C_2}{2-p}\right)^{1/p}\left(\frac\sigma{\sqrt{N}}\right)\right]\right\} 
\end{equation}
where $\sigma$ is the standard deviation and $C_1, C_2$ are positive constants.
\end{lemma}

Lemma \ref{lemma:mc} offers an a priori estimate of the probability that the unknown data-generating distribution $\mathcal{P}$ lies beyond the boundary of the Wasserstein ball. In this sense, the Wasserstein-based ambiguity set serves as a confidence region for the data-generating distribution. Moreover, the theorem reveals a crucial relationship between the Wasserstein radius and the variance, where the variance is explicitly affected by decision-making. This provides a theoretical basis for specifying and justifying the radius function. Thus far, based on Lemma \ref{lemma:mc}, we can estimate the radius of the smallest Wasserstein ball that contains $\mathcal{P}$ with a confidence level of $1-\beta$ for a given $\beta \in (0, 1)$. 

\begin{theorem}\label{thm:radius}
    Suppose Assumptions \ref{assumption1} and \ref{assumption2} hold, for $L(\bm{\zeta}(\bm{x}))=\sum_wa_wx_w\xi_{stw}$ with any decision $\bm{x}$, any significance level $\beta \in (0, 1)$ and a sufficiently large sample size $N$, there exists a radius 
    \begin{equation}\label{radius:true variance}
        \varepsilon_{st}(\bm{x}) = \sqrt{\frac{3(\bm{a}\cdot\bm{x})^\top{{\bm{\Sigma}}}_{st}(\bm{a}\cdot\bm{x})}{{N}}}  \log \left( \frac{C_1}{\beta} \right) + \left(\frac{C_2}{2-p}\right)^{1/p}\sqrt{\frac{(\bm{a}\cdot\bm{x})^\top{{\bm{\Sigma}}}_{st}(\bm{a}\cdot\bm{x})}{N}}
    \end{equation}
    such that 
    \begin{equation*}
    \mathbb{P}\{W_p(\hat {\mathbb P}_{st}(\bm{x}),{\mathbb P}_{st}(\bm{x}))\geq \varepsilon_{st}(\bm{x})\}\leq \beta
    \end{equation*}
    where ${\bm{\Sigma}}_{st}$ denotes the covariance matrix, $C_1, C_2$ are positive constants, and $\hat {\mathbb P}_{st}(\bm{x})=\frac{1}{N}\sum_{i=1}^N\delta_{\sum_wa_wx_w\hat\xi_{stwi}}$ represents the decision-dependent empirical distribution.
\end{theorem}

Theorem~\ref{thm:radius} provides a closed-form expression for the radius function of the proposed ambiguity set. A key insight is that the radius depends explicitly on the true covariance matrix $\bm{\Sigma}_{st}$ and the decision variable $\bm{x}$. This establishes a formal connection between the size of the uncertainty set and the statistical structure of wind power resources, thereby capturing the fundamental mechanism behind the smoothing effect.

However, in the radius formulation (\ref{radius:true variance}), $\bm{\Sigma}_{st}$ denotes the covariance matrix of the true distribution, which is typically unobservable in practice. Decision-makers must instead rely on finite samples to approximate it. To justify this approximation, we introduce the following mild assumption, which ensures that the sample covariance matrix can serve as a estimator of the true covariance matrix.

\begin{assumption} \label{assumption3}
The random vector $\bm{\xi}$ is bounded.
\end{assumption}

\begin{theorem}\label{thm:sample variance}
    Suppose Assumptions \ref{assumption1}-\ref{assumption3} hold, for $L(\bm{\zeta}(\bm{x}))=\sum_wa_wx_w\xi_{stw}$ with any decision $\bm{x}$, any significance level $\beta \in (0, 1)$ and a sufficiently large sample size $N$, there exists a radius 
    \begin{equation}\label{radius:sample variance}
        \varepsilon_{st}'(\bm{x}) = \sqrt{\frac{3(\bm{a}\cdot\bm{x})^\top\widehat{{\bm{\Sigma}}}_{st}(\bm{a}\cdot\bm{x})}{{N}}}  \log \left( \frac{C_1}{\beta} \right) + \left(\frac{C_2}{2-p}\right)^{1/p}\sqrt{\frac{(\bm{a}\cdot\bm{x})^\top\widehat{{\bm{\Sigma}}}_{st}(\bm{a}\cdot\bm{x})}{N}}
    \end{equation}
    such that 
    \begin{equation*}
        \mathbb{P}\{W_p(\hat {\mathbb P}_{st}(\bm{x}),{\mathbb P}_{st}(\bm{x}))\geq \varepsilon_{st}'(\bm{x})\}
        \leq \beta'
    \end{equation*}
    where $\hat {\bm{\Sigma}}_{st}$ denotes the sample covariance matrix,
    \begin{equation}\label{beta:sample variance}
        \begin{aligned}
            \beta'=&\inf_{k\in(0,1)}\left\{C_1\mathrm{exp}\Big[- k \log \Big( \frac{C_1}{\beta} \Big) +\frac{1-k}{\sqrt3}\Big(\frac{C_2}{2-p}\Big)^{1/p}\Big] \right.\\
            &+2\exp\big[-N\min\{C_3(1-k^2)^2((\bm{a}\cdot\bm{x})^\top{{\bm{\Sigma}}}_{st}(\bm{a}\cdot\bm{x}))^2, C_4(1-k^2)(\bm{a}\cdot\bm{x})^\top{{\bm{\Sigma}}}_{st}(\bm{a}\cdot\bm{x})\}\big]\Big\}\text{,}
        \end{aligned}
    \end{equation}
    and $C_1, C_2, C_3, C_4$ are positive constants.
\end{theorem}

Theorem \ref{thm:sample variance} establishes a probabilistic guarantee for the use of the sample covariance matrix in defining the radius of the Wasserstein ambiguity set. This result is built on the mild Assumption \ref{assumption3}, which is a reasonable premise due to the limited real-world wind power resources.

While the radius function (\ref{radius:sample variance}) is complex, with a pre-specified $p$ and a fixed sample size $N$, the radius $\varepsilon_{st}'(\bm{x})$ simplifies to
\begin{equation}\label{radius2}
\varepsilon_{st}'(\bm{x}) = \kappa_{st}\sqrt{(\bm{a}\cdot\bm{x})^\top\widehat{{\bm{\Sigma}}}_{st}(\bm{a}\cdot\bm{x})}\text{,}
\end{equation}
where $\kappa_{st}$ is a constant. Although there exists inevitable deviation between $\varepsilon_{st}'(\bm{x})$ and $\varepsilon_{st}(\bm{x})$, in practical applications, the optimal value of $\kappa_{st}$ can be effectively tuned through cross-validation.

\section{Model Reformulation}\label{sec:mr}

We now derive a tractable reformulation of the two-stage WDRO model, facilitating its direct implementation in standard optimization solvers. 

The reformulation challenge arises from two aspects: the inner problem involving uncertain aggregated wind power generation and the chance constraints addressing uncertain transmission line power flow. Both sources of uncertainty originate from the decision-dependent random output of wind farms. These two components, denoted as $L_1(\bm{\zeta}_{st}(\bm{x}))$ and $L_2(\bm{\zeta}_{st}(\bm{x}))$, respectively, are linear combinations of the wind farm output and inherit its decision-dependent characteristics. To address the uncertainty in the inner problem $SP_{st}$ and the chance constraints (\ref{cos:tlc}), we construct ambiguity sets specifically tailored to each component. Specifically, $\forall s \in [S], t \in [T], l\in [L]$, we define $L_1(\bm{\zeta}_{st}(\bm{x}))=\sum_wx_w\xi_{stw}$ for the inner problem $SP_{st}$, and $L_2(\bm{\zeta}_{st}(\bm{x}))=\sum_{w}\pi_{wl}x_w\xi_{stw}$ for the chance constraints (\ref{cos:tlc}). These formulations enable precise management of the decision-dependent uncertainty in both components.

\subsection{Inner Problem Reformulation}

The inner worst-case expectation problem, $\max_{{\mathbb P_{st}} \in \mathcal{P}_{st}(\bm{x})}\mathbb{E}_{\mathbb P_{st}}[h_{st}(\bm{\zeta}_{st}(\bm{x}))]$, poses a significant challenge as it involves infinite-dimensional optimization over probability distributions. Using the duality theory, we derive the tractable reformulations as a MISOCP for the inner problem $SP_{st}$.

\begin{theorem}\label{thm:mr}
    Given a decision $\bm{x}$, along with the ambiguity set in (\ref{ddas}) defined over the random variable $L_1(\bm{\zeta}_{st}(\bm{x}))$ and the radius specified in (\ref{radius2}), for every $s\in [S],t\in [T]$, 
    \begingroup
    \setlength{\abovedisplayskip}{0pt}
    \setlength{\belowdisplayskip}{0pt}
    \begin{align}
        \sup_{{\mathbb P_{st}} \in \mathcal{P}_{st}(\bm{x})}&\mathbb{E}_{\mathbb P_{st}}[h_{st}(\bm{\zeta}_{st}(\bm{x}))] = \min_{\vartheta_{st}}\ \mathbb{E}_{\hat {\mathbb P}_{st}(\bm{x})}[h_{st}(\bm{\zeta}_{st}(\bm{x}))] + \phi_{st}\kappa_{st}\vartheta_{st}\label{mr}
    \end{align}
    \begin{equation}\label{reformulation:radius}
        \hspace{2.7cm}\text{s.t.} \quad \sqrt{\bm{x}^\top\widehat{{\bm{\Sigma}}}_{st}\bm{x}} \leq \vartheta_{st}
    \end{equation}
    \endgroup
    where $\hat {\mathbb P}_{st}(x)=\frac{1}{N}\sum_{i=1}^N\delta_{\sum_w x_w\hat\xi_{stwi}}$, $\phi_{st}:=\max_{\gamma_{st} \in \mathcal{D}_{st}}|\gamma_{st}|$, and $\mathcal{D}_{st}$ is the dual feasible region.
\end{theorem}

Based on Theorem \ref{thm:mr}, the inner worst-case expectation problem is reformulated as a regularization problem comprising a sample average approximation term and a regularization term. This transformation simplifies the structure of the optimization problem, enabling the reduction of the original two-stage DRO model $P$ into a computationally tractable single-stage problem with chance constraints.

Another implication of the MISOCP reformulation lies in the statistical property of the WDRO. The objective value of the inner problem provides a probabilistic guarantee on the out-of-sample performance of planning decisions at a high confidence level. This property ensures the robustness of the optimal decision and enhances the reliability of the proposed model, as formalized in the following result.

\begin{theorem}\label{thm:pg}
    Suppose Assumptions \ref{assumption1}-\ref{assumption3} hold, then for any decision $\bm{x}$ and a sufficiently large sample size $N$, there exists a radius parameter $\kappa_{st}$ such that 
    \begingroup
    \setlength{\abovedisplayskip}{2pt}
    \setlength{\belowdisplayskip}{2pt}
    \begin{equation*}
    \begin{aligned}
    \mathbb{P}\left\{\mathbb{E}_{\mathbb P_{st}}[h_{st}(\bm{\zeta}_{st}(\bm{x}))]\leq{J}_{st}\right\}\geq 1-\beta'
    \end{aligned}
    \end{equation*}
    \endgroup
    where $J_{st}$ represent the optimal value of problem (\ref{mr}), and $\beta'$ is defined as in (\ref{beta:sample variance}).
\end{theorem}

\subsection{Chance Constraints Approximation}
The distributionally robust chance-constrained model ensures that the probability of violating transmission line power flow constraints remains below a specified risk threshold for the worst-case distribution within the ambiguity set. Mathematically, the distributionally robust joint chance constraints for the transmission line power flow (\ref{cos:tlc}) can be written equivalently as follows: 
\begin{equation}\label{cvar1}
    \inf_{\mathbb{P}_{stl}(\bm{x})\in \mathcal{P}_{stl}(\bm{x})} \mathbb{P}_{stl}(\bm{x})\left[\max_{k\in \{1,2\}}l_{stlk}(\bm{\zeta}_{st}(\bm{x}))+l^0_{stlk}\leq  0 \right]\geq 1- \epsilon \quad\forall s,t,l
\end{equation}
where $l_{stlk}(\bm{\zeta}_{st}(\bm{x}))=\pm L_2(\bm{\zeta}_{st}(\bm{x}))$ for $k=1,2$, representing the directional power flow components, and $l^0_{stlk}$ denotes additional terms within each constraint that do not involve uncertain parameters.

Solving problems with exact chance constraints \citep{chen2024data} is computationally challenging, especially when extended to decision-dependent scenarios, as it involves bilinear terms. Additionally, in practice, violations of constraints (\ref{cos:tlc}) can incur costly risks, which must be factored into the decision-making process. To balance tractability and risk management, we adopt a risk-averse method, Conditional Value at Risk (CVaR), to quantify the magnitude of violations and provide a tractable approximation. The CVaR approximation of the chance constraints is expressed as:

\begin{equation}\label{cvar2}
    \sup_{\mathbb{P}_{stl}(\bm{x})\in \mathcal{P}_{stl}(\bm{x})} \text{CVaR}_{1-\epsilon}\left[\max_{k\in\{1,2\}} l_{stlk}(\bm{\zeta}_{st}(\bm{x}))+l^0_{stlk}\right]\leq 0 \quad\forall s,t,l
\end{equation}

As shown in \citet{chen2023approximations}, CVaR provides a tight convex approximation of chance constraints under the WDRO framework, making it a widely used approach in risk management. Furthermore, as a tool in optimization modeling, CVaR can be expressed as a convex minimization problem \citep{rockafellar2000optimization}:

\begin{equation*}
    \mathrm{CVaR}_{1-\epsilon}(\zeta)=\min\left\{\eta+\frac1{\epsilon}\mathbb{E}\left([\zeta-\eta]_+\right) : \eta\in\mathbb{R}\right\}\text{.} \label{cvar_min} 
\end{equation*}
This expression leads to a tractable reformulation of the distributionally robust CVaR constraints.

\begin{theorem}\label{thm:tlc}

Given a decision $\bm{x}$ and ambiguity set (\ref{ddas}) with the random variable $L_2(\bm{\zeta}_{st}(\bm{x}))$, the constraints (\ref{cvar2}) can be equivalently reformulated as

\begin{equation}\label{cvar3}
\left\{\begin{array}{l}
\eta_{stl} + \frac{1}{\epsilon}\left\{l^{0}_{stl,1}-\eta_{stl} + \kappa_{stl}\vartheta_{stl} + \frac1N\sum_i\sum_w\pi_{wl}x_w\hat\xi_{stwi}\right\} \leq 0 \\[2mm]
\eta_{stl} + \frac{1}{\epsilon}\left\{l^0_{stl,2}-\eta_{stl} + \kappa_{stl}\vartheta_{stl}-\frac1N\sum_i\sum_w\pi_{wl}x_w\hat\xi_{stwi}\right\} \leq 0 \\[2mm]
\eta_{stl} + \frac{1}{\epsilon}\kappa_{stl}\vartheta_{stl} \leq 0 \\[2mm]
\sqrt{(\bm{\pi}_{\cdot l}\cdot\bm{x})^\top\widehat{{\bm{\Sigma}}}_{st}(\bm{\pi}_{\cdot l}\cdot\bm{x})} \leq \vartheta_{stl}
\end{array} \right. \quad \quad \forall s,t,l
\end{equation}
where $\bm{\pi}_{\cdot l}:=(\pi_{wl})_{w \in [W]}$ and 
\begin{equation}
\left\{\begin{array}{l}
l^{0}_{stl,1} = \sum_g\pi_{gl}P_{stg}+[-\pi_{gl}]_+(\bar r_{stg} + \underline r_{stg})+\pi_{gl}\underline r_{stg}-\sum_{d}\pi_{dl}{\bar P}_{std} - \bar P_l \\[2mm]
l^{0}_{stl,2}= -\sum_{g}\pi_{gl}P_{stg}+[\pi_{gl}]_+(\bar r_{stg} + \underline r_{stg})-\pi_{gl}\underline r_{stg}+\sum_{d}\pi_{dl}{\bar P}_{std} - \bar P_l 
\end{array} \right. \quad \forall s,t,l
\end{equation}
\end{theorem}

\subsection{Model Simplification without Correlation} \label{subsec:uncorrelated}

In the previous formulation, the ambiguity set accounts for correlations among the wind power outputs of different wind farms through the sample covariance matrix $\widehat{{\bm{\Sigma}}}_{st}$. While incorporating correlation information can theoretically improve decision quality, estimating $\widehat{{\bm{\Sigma}}}_{st}$ accurately requires a large amount of high-quality data. In practice, however, such data may be unavailable or unreliable, especially in early-stage planning or in regions with limited historical records. Furthermore, the inclusion of correlation increases computational burden. To address these issues, we consider a simplified model that ignores correlation and only utilizes the marginal variance information.

Let $\hat{\sigma}^2_{stw}$ denote the sample variance of the wind power resource at $w$ under scenario $s$ and time $t$. By assuming independence across wind farms, the covariance matrix $\hat{{\bm{\Sigma}}}_{st}$ reduces to a diagonal matrix. Under this simplification, the radius constraints (\ref{reformulation:radius}) reduces to
\begin{equation}
\sqrt{\sum_{w} x_w^2 \hat{\sigma}^2_{stw}} \leq \vartheta_{st}\text{,}\label{radius_diag}
\end{equation}
and the radius constraint in (\ref{cvar3}) can be transformed analogously.

These reformulated constraints maintain the model’s tractability and robustness while significantly reducing the data requirement and computational burden. Despite ignoring correlation, the variance-only approach retains much of the critical uncertainty information. This simplification also avoids the potential risk of overfitting to noisy correlation estimates, which may lead to suboptimal or unstable planning decisions. Numerical results in Section~\ref{subsec:vvc} demonstrate that simplified model achieves a strong smoothing effect and stable performance, making it a practical and efficient alternative when correlation information is unavailable or unreliable.

\section{A Constraint Generation Based Solution Framework}\label{sec:cgbsf}

We start by providing the MISOCP reformulation of problem $P$,
\begin{align}
\begin{split} \label{reformulation:obj}
    RP:\quad \min \quad &\sum_w{c_wx_w} +\sum_s\Big\{\sum_{t}\sum_g\left[F_g(P_{stg}) + UR_{g}\bar r_{stg} + DR_{g}\underline r_{stg}\right] +  \\[-2mm]
    &\hspace{2cm}
    \sum_{t}\big[\phi_{st}\kappa_{st}\vartheta_{st} +\frac1N\sum_i(\sum_gz_{stgi}+WCw_{sti} + LSl_{sti})\big]\Big\}\Delta_s 
\end{split} \\
\text{s.t.} \quad &(\ref{cos:power balance})-(\ref{cos:ramp down limit}), (\ref{cos:non-negative})-(\ref{cos:integer}), (\ref{reformulation:radius}), (\ref{cvar3}) \\
&\left\{\begin{array}{l}\label{reformulation:empirical expectation}
DA_g\alpha_{stgi} \leq z_{stgi}  \quad \forall g,i\\[2mm]
-UA_g\alpha_{stgi} \leq z_{stgi} \quad \forall g,i \\[2mm]
\sum_g\alpha_{stgi} + w_{sti}-l_{sti} = \sum_wx_w(\hat\xi_{stwi}-\bar \xi_{stwi}) \quad \forall i  \\[2mm]
-\bar r_{stg} \leq \alpha_{stgi} \leq \underline r_{stg} \quad\quad\quad \forall g,i 
\end{array}\right. &\forall s,t
\end{align}

While the MISOCP reformulation can be solved using off-the-shelf solvers like Gurobi, computational time often poses a significant bottleneck in practical applications. Preliminary experiments reveal that directly solving the MISOCP reformulation with Gurobi fails to yield optimal solutions within an hour, even for small instances involving the planning of just five wind farms. This highlights the urgent need to accelerate the solution procedure. To tackle this challenge, we propose an efficient solution framework that leverages the structural properties of the MISOCP to improve computational performance.

Within the MISOCP model, the most computationally challenging components are the distributionally robust CVaR constraints related to transmission line capacity (\ref{cos:tlc}). These constraints involve numerous mixed-integer second-order cone conditions, substantially increasing computational burden, particularly in large-scale systems with many transmission lines. However, since most transmission lines typically have sufficient power flow capacity, many of these constraints are redundant. Therefore, we propose a constraint generation approach that dynamically adds these constraints during the solution process, which can greatly reduce the model size and decrease the computational time.

For simplicity, we define a set $\mathcal{L} = \{(s,t,l, k):s\in [S], t\in [T],l \in [L], k \in \{1,2\}\}$, where each tuple $(s,t,l,k)$ represents a transmission line $l \in [L]$ with its power flow direction $k$ in scenario $s$ and time period $t$. The constraint generation approach iteratively solves a relaxation of Problem $RP$, where the full set of transmission lines and directions $\mathcal{L}$ is replaced with a subset $\mathcal{L}' \subseteq \mathcal{L}$. Due to the relative stability of power flow directions, the subset $\mathcal{L}'$ includes only specific transmission lines with specific directions that have ever violated the distributionally robust CVaR constraints during the iterative process. These constraints can be equivalently evaluated as follows:
\begingroup
    \setlength{\abovedisplayskip}{5pt}
    \setlength{\belowdisplayskip}{10pt}
\begin{equation*}
    \sup_{\mathbb{P}_{stl}(\bm{x})\in \mathcal{P}_{stl}(\bm{x})} \text{CVaR}_\epsilon\left[l_{stlk}(\bm{\zeta}_{st}(\bm{x}))+l^0_{stlk}\right] = \text{CVaR}_\epsilon\left[l_{stlk}(\hat {\bm{\zeta}}_{st}(\bm{x}))+l^0_{stlk} + \frac{\varepsilon_{stl}'(\bm{x})}{\epsilon}\right]\text{.}
\end{equation*}
\endgroup
If the solution of the relaxed problem is feasible for all distributionally robust CVaR constraints, the solution procedure is terminated. The details are summarized in Algorithm \ref{Constraint Generation}. 
\begin{algorithm}[H]
\caption{Constraint Generation}\label{Constraint Generation}
\textbf{Input:} $\mathcal{L}' \gets \emptyset$
\begin{algorithmic}[1]
\State \textbf{Solve relaxed problem}
    \Statex \quad 1.1. Relax Problem $RP$ by replacing $\mathcal{L}$ with $\mathcal{L}'$.
    \Statex \quad 1.2. Solve the relaxation to obtain a solution $(\bm{x}', \dots)$.
\State \textbf{Check feasibility and add constraints}
    \Statex \quad 2.1. Given $\bm{x}'$, let $k\in\{1,2\}$ denote the different power flow directions.
    \Statex \quad 2.2. Find $\Phi=\left\{(s,t,l,k) : \text{CVaR}_\epsilon\left[l_{stlk}(\hat{\bm{\zeta}}_{st}(\bm{x}'))+l^0_{stlk} + \frac{\varepsilon_{stl}'(\bm{x}')}{\epsilon}\right]>0,s\in [S], t \in [T], l \in [L], k \in \{1, 2\}\right\}$.
    \Statex \quad 2.3. If $\Phi \neq \emptyset$, update $\mathcal{L}' \gets \mathcal{L}' \cup \Phi$ and return to step 1; otherwise, terminate.
\end{algorithmic}
\textbf{Output:} $(\bm{x}', \dots)$ as the optimal solution.
\end{algorithm}

\vspace{-0.5cm}

Empirically, after a single iteration of adding back a small number of constraints, the solution satisfies all CVaR constraints. This confirms that most constraints are indeed redundant and demonstrates the efficiency and practicality of the constraint generation approach. Detailed information on the improvement in solution time is provided in Section \ref{sec:ne}.

Within the relaxed problem, the empirical expectation in Problem $RP$ becomes computationally intensive as the sample size increases, as both the number of variables and constraints grows proportionally with the sample size. To further accelerate the solution process, we employ the L-shaped algorithm (See Appendix \ref{sec:la} for details), which leverages decomposition techniques to improve computational efficiency.

\section{Numerical Experiments}\label{sec:ne}

In this section, we evaluate the performance of the solution framework and the proposed decision-dependent distributionally robust optimization (DDRO) method on different datasets. We first introduce the benchmark methods as follows.

\begin{enumerate}[label=\Roman*.]
    \item Normalized distributionally robust optimization (NDRO) method: NDRO method models wind power resources at different locations as a decision-independent multivariate random variable. We consider the case where $p = 1$ and adopt the 2-norm in the Wasserstein distance. For consistency, we denote the Wasserstein radius as $\kappa'_{st}$. Based on Corollary 5.1 in \citet{mohajerin2018data}, the inner problem $SP_{st}$ can then be reformulated as:
    \begin{equation}\label{NDRO:mr}
    \max_{\mathbb{P}_{st}(\bm{x}) \in \mathcal{P}_{st}(\bm{x})} \mathbb{E}_{\mathbb{P}_{st}}\left[h_{st}(\bm{\zeta}_{st}(\bm{x}))\right] = \mathbb{E}_{\hat{\mathbb{P}}_{st}(\bm{x})} [h_{st}(\bm{\zeta}_{st}(\bm{x}))] + \phi_{st}\kappa^{\prime}_{st}\|\bm{x}\|_2 \text{.}
    \end{equation}
    \item Empirical optimization (EO) method: Different from DRO methods, EO treats the empirical distribution as the true distribution by setting $\varepsilon_{st}(\bm{x}) = 0, \forall s,t$. Thus, the worst-case expectation $\max_{\mathbb{P}_{st}(\bm{x}) \in \mathcal{P}_{st}(\bm{x})}\mathbb{E}_{\mathbb{P}_{st}(\bm{x})} h_{st}(\bm{\zeta}_{st}(\bm{x}))$ is replaced by the empirical expectation $\mathbb{E}_{\hat{\mathbb{P}}_{st}(\bm{x})} [h_{st}(\bm{\zeta}_{st}(\bm{x}))]$.
\end{enumerate}

For the hyperparameters of each method, namely $\kappa_{st}$ and $\kappa'_{st}$, we employed cross-validation to determine their optimal values. All experiments are conducted on an Intel Xeon E5-2620 v4 Processor with 128 GB memory. The experiments are coded using Python 3.11.3, with Gurobi 11.0.3 serving as the MISOCP solver.

\subsection{Data with Independent Marginals} \label{subsec:sd}

We conducted experiments using the classical IEEE 118-bus test system, a well-established benchmark in power system researches (\citealp{yin2022stochastic, yin2022coordinated}). This system consists of 118 buses, 186 transmission lines, 54 thermal generators, and 99 load buses. In our experiments, we assume a deterministic net load and make capacity decisions for wind farms located at 9 candidate nodes \{1, 13, 25, 37, 49, 51, 63, 75, 87\}. The unit cost of wind curtailment $WC$ and load shedding $LS$ are set to 100 and 200, respectively. The tolerance violation probability of chance constraints is specified as 0.1.

In this section, we focus on a single scenario and a single time period ($S=T=1$) for clarity of exposition. Nevertheless, the proposed model can be readily extended to multiple scenarios and time periods, which requires only additional scenario and temporal data as well as increased computational effort. We generate the dataset using the Weibull distribution, a widely used log-concave distribution in the literature for modeling thewind power resources (\citealp{bowden1983weibull, islam2011assessment}). We randomly pre-specify the mean of each wind power resource within the range $[0.96, 1.44]$ and the variance within $[0.0576, 0.1210]$. Then we calculate the corresponding scale and shape parameters of Weibull distributions for sampling. To examine the correlations among the sampled data, we compute the correlation coefficients across all wind farms and find that they are consistently close to zero. As a result, the experiments in this section utilize only variance information.

We first report the average computational time for three solution methods in Table \ref{tab:systhetic data 1} and Table \ref{tab:systhetic data 2}: direct computation (DC), the constraint generation alone (CG), and the constraint generation combined with the L-shaped algorithm (CG-L). The computational time is averaged over 10 experiments, varying the number of wind farms, sample sizes, and transmission line capacities.

\begin{table}[H]
\TABLE
{Average Computational Time (in Seconds) \label{tab:systhetic data 1}}
{\begin{tabular}{ccccccccccccc}
\hline
\textbf{} &     & \multicolumn{3}{c}{420}               &  & \multicolumn{3}{c}{380}               &  & \multicolumn{3}{c}{350}               \\ \cline{3-5} \cline{7-9} \cline{11-13} 
W        & N   & DC     & CG           & CG-L          &  & DC     & CG           & CG-L          &  & DC     & CG           & CG-L          \\ \cline{1-9} \cline{11-13} 
3         & 30  & 290.4  & \textbf{2.4} & 3.7           &  & 253.9  & \textbf{3.4} & 4             &  & 275.3  & \textbf{3.4} & 4.2           \\
          & 60  & 661    & 5.2          & \textbf{5}    &  & 601    & 9.3          & \textbf{5.9}  &  & 620    & 10.3         & \textbf{6}    \\
          & 90  & 1056.9 & 15.9         & \textbf{7}    &  & 1006.2 & 26.8         & \textbf{8.4}  &  & 1202.6 & 27.7         & \textbf{8.8}  \\
          & 120 & 1349.1 & 19.3         & \textbf{7.6}  &  & 1590.2 & 45.2         & \textbf{9}    &  & 1412.5 & 39.5         & \textbf{9.1}  \\
          & 150 & 1956.8 & 31           & \textbf{8.5}  &  & 1754.9 & 59.9         & \textbf{10.6} &  & 1853.3 & 59.2         & \textbf{12.3} \\
          & 180 & 2690.7 & 42.6         & \textbf{10.6} &  & 2391.2 & 91.6         & \textbf{12.8} &  & 2534.3 & 87.4         & \textbf{13}   \\
          & 240 & 3615.3 & 77.9         & \textbf{13.6} &  & 3837.2 & 161.7        & \textbf{17.1} &  & 3730   & 164.4        & \textbf{16.6} \\ \hline
\end{tabular}}
{}
\end{table}

\begin{table}[H]
    \TABLE
    {Average Computational Time (in Seconds)\label{tab:systhetic data 2}}
    {
    \centering
    \begin{tabular}{cccccccccc}
    \hline
    \multicolumn{1}{l}{\textbf{}} & \textbf{} & \multicolumn{2}{c}{420}                           & \textbf{} & \multicolumn{2}{c}{380}                           & \textbf{} & \multicolumn{2}{c}{350}                           \\ \cline{3-4} \cline{6-7} \cline{9-10} 
    W                             & N         & \multicolumn{1}{c}{CG} & \multicolumn{1}{c}{CG-L} & \textbf{} & \multicolumn{1}{c}{CG} & \multicolumn{1}{c}{CG-L} & \textbf{} & \multicolumn{1}{c}{CG} & \multicolumn{1}{c}{CG-L} \\ \hline
    5                             & 30        & \textbf{3.8}           & 5.3                      & \textbf{} & \textbf{4.5}           & 5.6                      &           & \textbf{5.7}           & 6                        \\
                                  & 60        & \textbf{6.7}           & 7.7                      &           & 10.4                   & \textbf{9.4}             &           & 14.4                   & \textbf{8}               \\
                                  & 90        & 16.9                   & \textbf{8.9}             &           & 28.1                   & \textbf{9.1}             &           & 33.4                   & \textbf{10.4}            \\
                                  & 120       & 37.1                   & \textbf{10.9}            &           & 61.6                   & \textbf{12.5}            &           & 64.6                   & \textbf{11.9}            \\
                                  & 150       & 55.8                   & \textbf{12}              &           & 87.9                   & \textbf{13.7}            &           & 277.4                  & \textbf{14}              \\
                                  & 180       & 74.6                   & \textbf{15.1}            &           & 133.1                  & \textbf{16.8}            &           & 142.7                  & \textbf{15.2}            \\
                                  & 240       & 180.1                  & \textbf{19.1}            &           & 337.4                  & \textbf{19.2}            &           & 382.2                  & \textbf{18.5}            \\ \hline
    7                             & 30        & \textbf{3.1}           & 7.6                      &           & \textbf{3.9}           & 8.2                      &           & \textbf{5.6}           & 8.4                      \\
                                  & 60        & 18.3                   & \textbf{9.9}             &           & 30.6                   & \textbf{10.3}            &           & 31.8                   & \textbf{11.7}            \\
                                  & 90        & 20.7                   & \textbf{14.1}            &           & 34.4                   & \textbf{14.6}            &           & 38.7                   & \textbf{15.3}            \\
                                  & 120       & 44.7                   & \textbf{17.6}            &           & 74.6                   & \textbf{18.9}            &           & 82.3                   & \textbf{17.7}            \\
                                  & 150       & 91                     & \textbf{22}              &           & 140.1                  & \textbf{21.1}            &           & 154.7                  & \textbf{21.1}            \\
                                  & 180       & 140.6                  & \textbf{21.4}            &           & 220.5                  & \textbf{23}              &           & 254.9                  & \textbf{23.4}            \\
                                  & 240       & 258.6                  & \textbf{27.7}            &           & 466.1                  & \textbf{28.5}            &           & 551.3                  & \textbf{28.2}            \\ \hline
    9                             & 30        & \textbf{2.9}           & 16                       &           & \textbf{3.3}           & 18.2                     &           & \textbf{5.6}           & 18.8                     \\
                                  & 60        & \textbf{9.1}           & 16.4                     &           & \textbf{9.3}           & 17.8                     &           & 24.5                   & \textbf{18}              \\
                                  & 90        & \textbf{24}            & 29.8                     &           & \textbf{24.4}          & 29.6                     &           & 43.2                   & \textbf{33}              \\
                                  & 120       & 45.4                   & \textbf{34.2}            &           & 43.9                   & \textbf{33.6}            &           & 97.5                   & \textbf{34.6}            \\
                                  & 150       & 70.6                   & \textbf{32.4}            &           & 74.3                   & \textbf{32.8}            &           & 146.6                  & \textbf{33}              \\
                                  & 180       & 147.4                  & \textbf{37.6}            &           & 148                    & \textbf{37.3}            &           & 245.8                  & \textbf{41.4}            \\
                                  & 240       & 244.6                  & \textbf{45.2}            &           & 244                    & \textbf{44.8}            &           & 568.1                  & \textbf{47.5}            \\ \hline
    \end{tabular}}
    {}
\end{table}
\vspace{-0.3cm}

Due to the complexity of MISOCP, the computational time for DC increases rapidly as the number of wind farms and sample size grow. Notably, DC fails to deliver an optimal solution within one hour when the number of wind farms increases to 5, even with only 30 training samples. In contrast, both CG and CG-L consistently and substantially improve computational performance. Therefore, for larger problem instances, we limit the computational performance tests to these two methods. For smaller sample sizes, CG achieves the optimal decision in less time, as the iterative nature of the L-shaped algorithm (CG-L) offers limited benefits in such cases. However, as the sample size increases, the L-shaped algorithm (CG-L) significantly improves computational efficiency and leads to a much slower growth rate of training time compared to CG. This enhancement is particularly important for large-scale power system analyses, where computational demands can become substantial. Additionally, in most instances, a reduction in transmission line capacity results in increased computational time. This is expected, as lower capacity raises the likelihood of transmission line violations concerning the distributionally robust CVaR constraints. When violations occur, additional constraints must be reintroduced into the relaxed model to ensure reliability and compliance with operational standards. Nonetheless, empirical observations indicate that only a small number of lines need to be added back, often in a single iteration, to achieve the optimal decision. This finding highlights the efficiency and practicality of the constraint generation approach in power systems.

In the following, we evaluate the model’s performance. All optimization models are solved based on a sample size 60. The out-of-sample performance of the planning decisions obtained from each method is assessed using a test set of size 3000. For each instance, we solve the DDRO/NDRO model across a sequence of parameters $\kappa/\kappa'$, with EO included when $\kappa=\kappa^\prime=0$. 

\begin{remark}
    We emphasize that DDRO outperforms benchmark methods in terms of total cost, but the differences are marginal—typically less than 1\%, which is consistent with prior studies in multi-farm planning(\citealp{yin2022stochastic,yin2022coordinated}). This is primarily because investment and generation costs account for the vast majority of total cost, whereas the risk management cost (including reserve, adjustment, curtailment, and load shedding) constitutes only a small portion. Therefore, differences in the models’ ability to hedge against operational risk are largely masked in the total cost. To more clearly highlight the advantages of our model in risk management and ensure a fair comparison across different models, we standardize the unit investment cost across all wind farms and adjust parameters in the day-ahead scheduling stage. Under these settings, if the total installed capacity $\sum_w x_w = X$ is fixed, then any differences in risk management cost directly correspond to differences in total cost. As a result, We shift focus to the risk management, including reducing operational risks and mitigating fluctuations in the power system.  
    Further implementation details are presented in Appendix \ref{sec:nr}.
\end{remark}

\begin{figure}[H]
    \FIGURE
    {
    \subcaptionbox{$X=400$}
    {\includegraphics[width=0.3\textwidth]{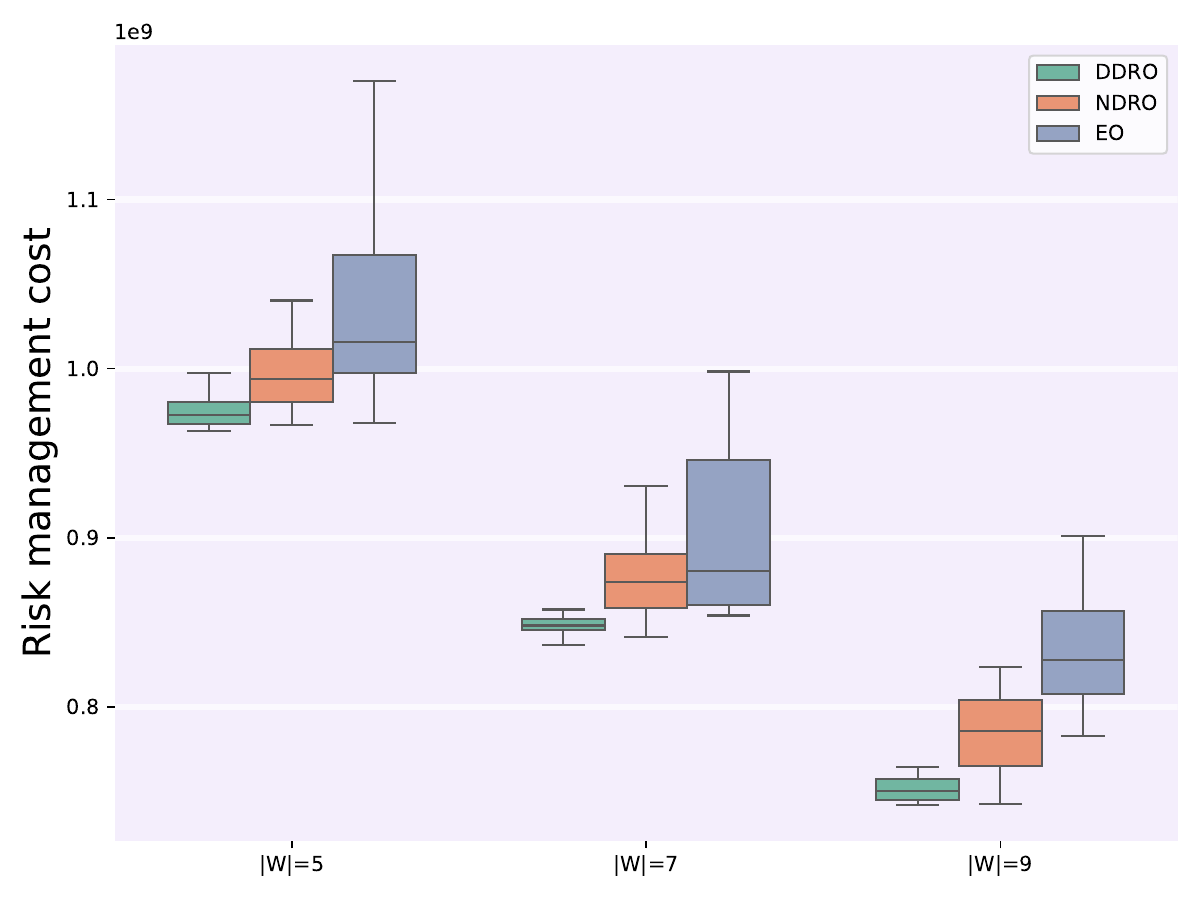}}
    \subcaptionbox{$X=600$}
    {\includegraphics[width=0.3\textwidth]{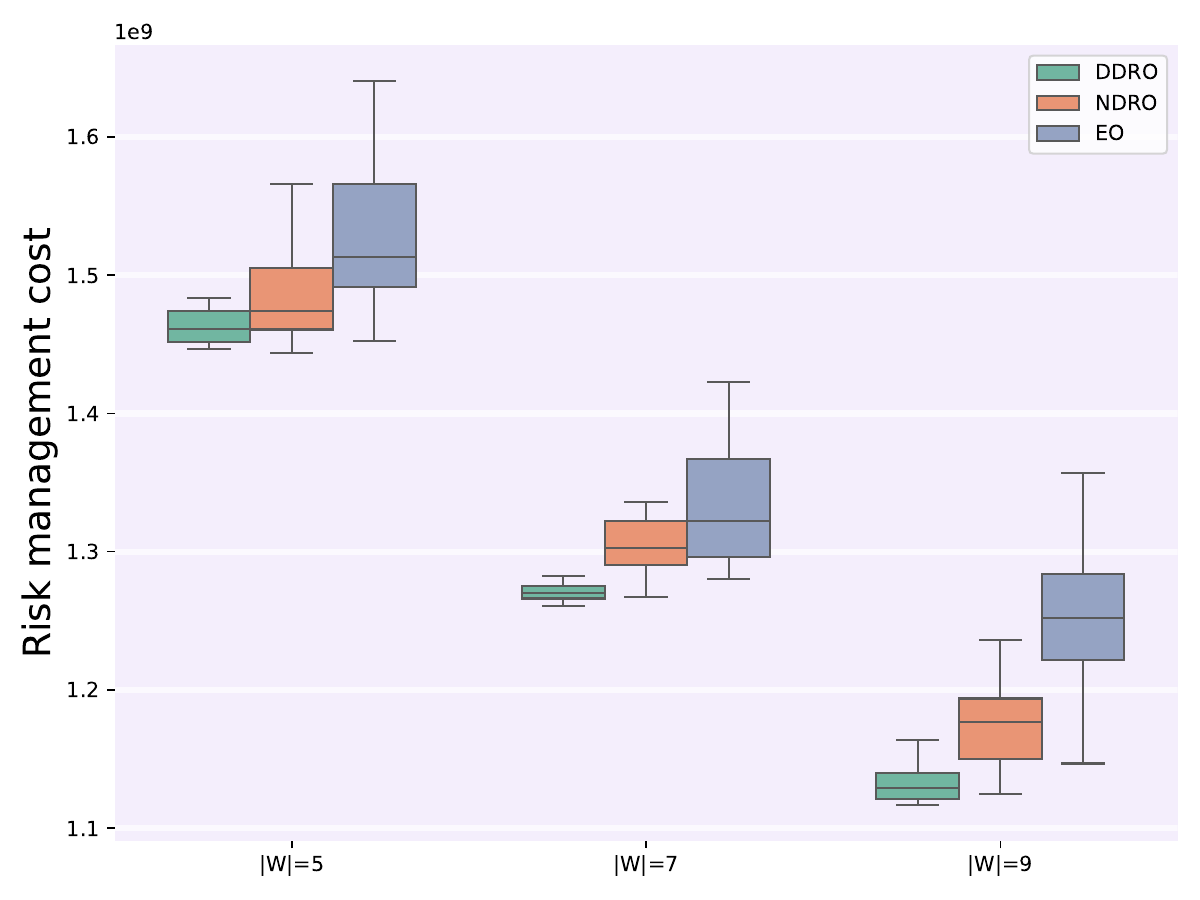}}
    \subcaptionbox{$X=900$}
    {\includegraphics[width=0.3\textwidth]{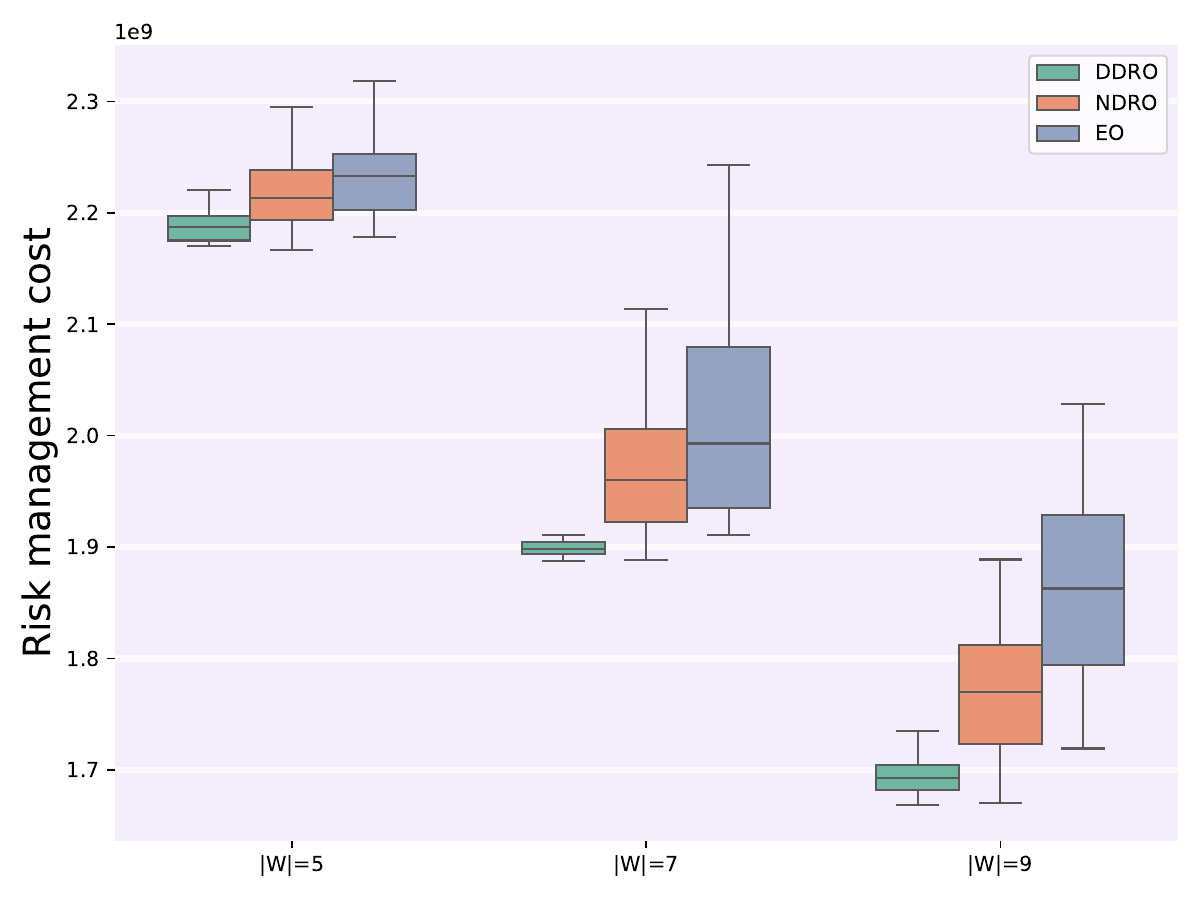}}
    }
    {
    Box Plots of Out-of-sample Risk Management Cost Using Different Methods  \label{fig:synthetic_data_boxplot}}
    {}
\end{figure}

\begin{table}[H]
    \TABLE
    {Average Risk Management Cost ($10^8$) \label{tab:systhetic data mean}}
    {\begin{tabular}{cccccccc}
    \hline
    $X$         &  W             &  & DDRO           &  & NDRO          &  & EO    \\ \hline
    400     & 5       &  & \textbf{9.75}  &  & 10.04          &  & 10.32  \\
            & 7       &  & \textbf{8.48}  &  & 8.81           &  & 9.13  \\
            & 9       &  & \textbf{7.51}  &  & 7.98           &  & 8.35  \\
    600     & 5       &  & \textbf{14.63} &  & 15.02          &  & 15.40 \\
            & 7       &  & \textbf{12.69} &  & 13.12          &  & 13.56 \\
            & 9       &  & \textbf{11.34} &  & 11.82          &  & 12.58 \\
    900     & 5       &  & \textbf{21.89} &  & 22.41          &  & 22.43 \\
            & 7       &  & \textbf{18.97} &  & 19.82          &  & 20.51 \\
            & 9       &  & \textbf{16.98} &  & 17.91          &  & 18.74 \\ \hline
    \end{tabular}}
    {}
\end{table}

To examine the impact of target installed capacity, we vary $X \in \{400, 600, 900\}$. Additionally, to investigate the effect of the number of wind farm, we vary $W \in \{5, 7, 9\}$. For consistency, each method uses the same training and testing datasets, and 20 repeated experiments are conducted. The out-of-sample risk management costs are visualized as box plots in Figure \ref{fig:synthetic_data_boxplot}, while the average out-of-sample costs are summarized in Table \ref{tab:systhetic data mean}. Based on these results, we have the following observations:

\begin{enumerate}[label=\Roman*.] 
    \item In general, DDRO demonstrates superior performance, particularly in regimes with less wind farms ($\text{W} = 5$) and larger target capacity ($X = 600, 900$), compared to NDRO. This superiority arises from DDRO’s ability to incorporate variance information, enabling it to account for fluctuations across different locations and allocate capacity more precisely. Moreover, DDRO aligns with the development trends of wind energy systems, where larger scales benefit more significantly from its robust and variance-informed optimization framework. 
    \item Under the same target capacity, experiments involving more wind farms exhibit lower risk management costs. This finding highlights the promising potential of distributed renewable energy systems, as increasing the number of wind farms under a fixed capacity target enhances the smoothing effect, thereby enabling more efficient and resilient risk management.
    \item Compared to DDRO and NDRO, the EO method performs the worst and exhibits the highest variability, as indicated by its larger interquartile range in the box plots. These observations highlight the superior risk-aversion capability of DRO-based models.
\end{enumerate}

\begin{figure}[htbp]
    \FIGURE
    {
    \subcaptionbox{$X=400$}
    {\includegraphics[width=0.3\textwidth]{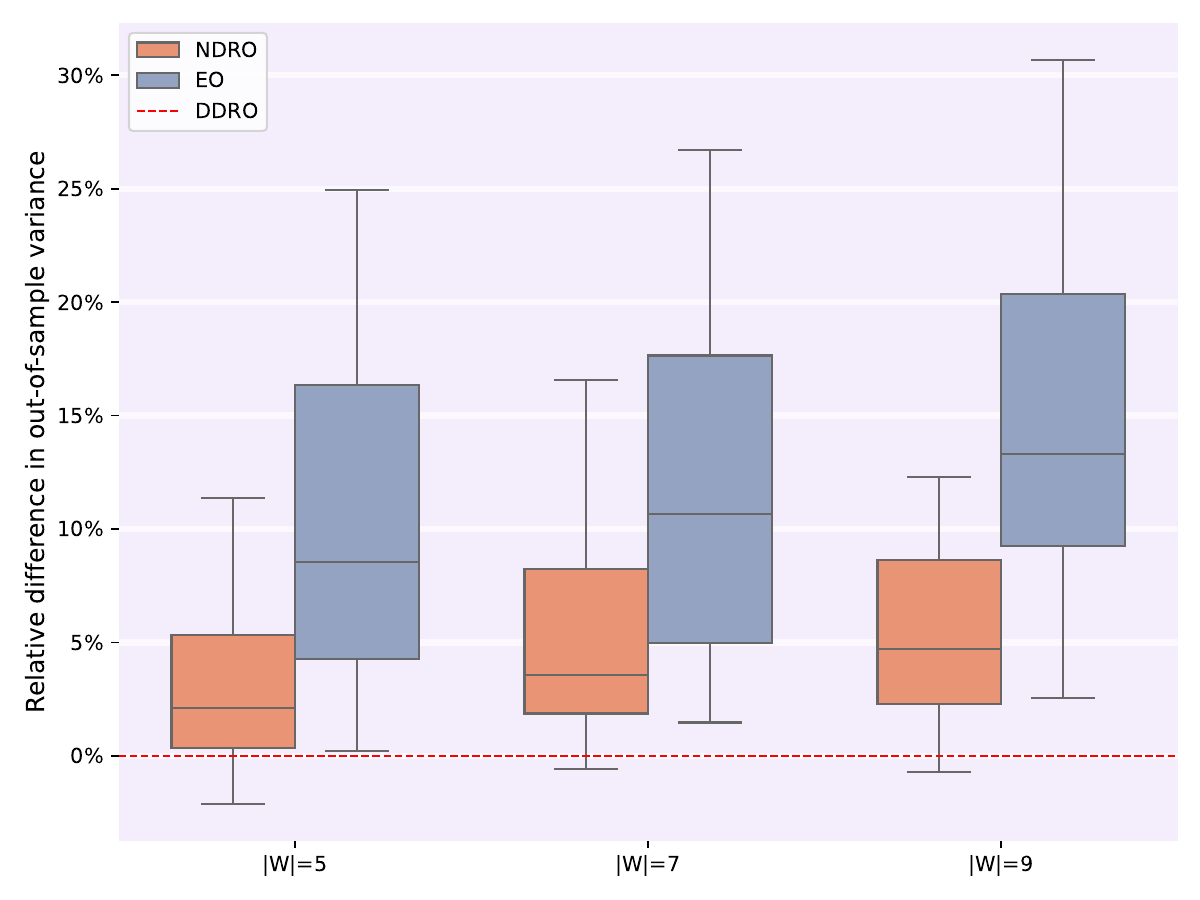}}
    \subcaptionbox{$X=600$}
    {\includegraphics[width=0.3\textwidth]{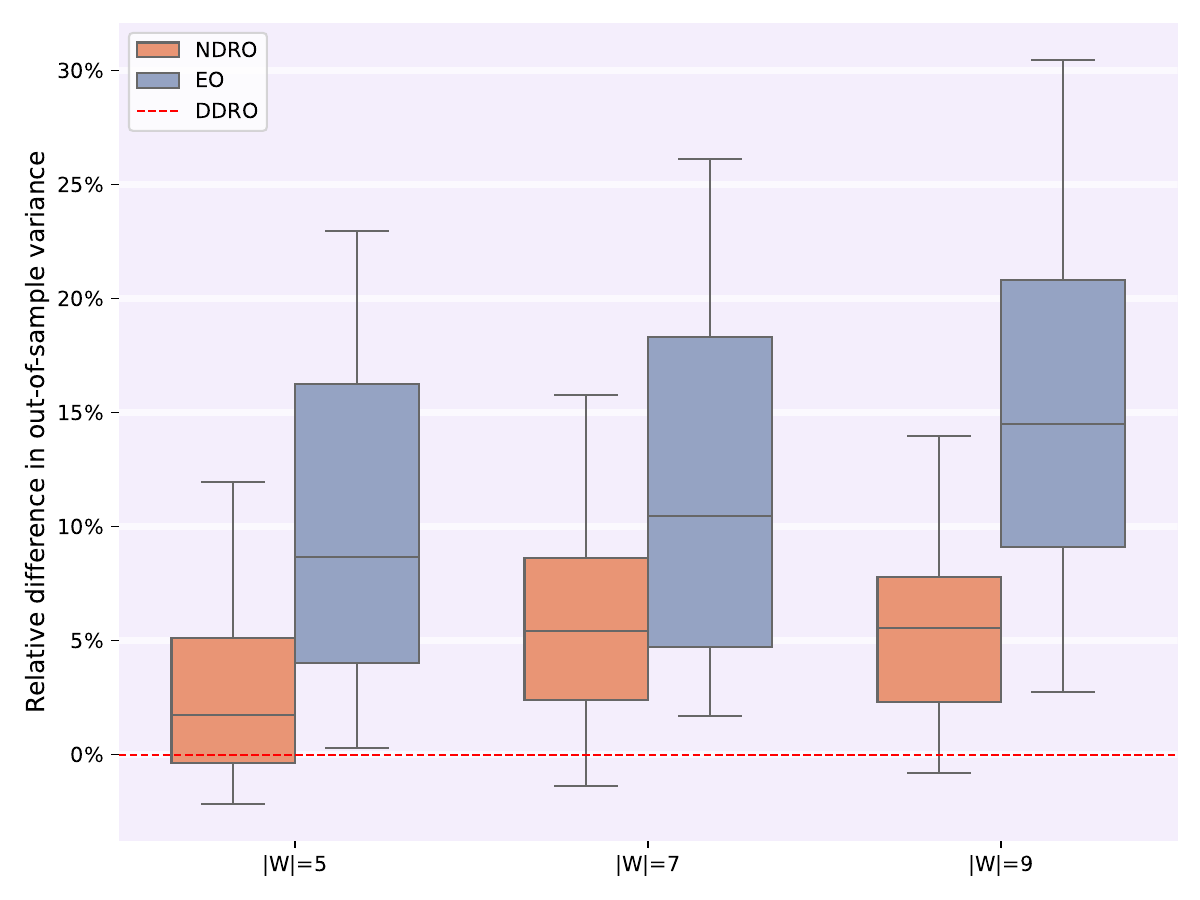}}
    \subcaptionbox{$X=900$}
    {\includegraphics[width=0.3\textwidth]{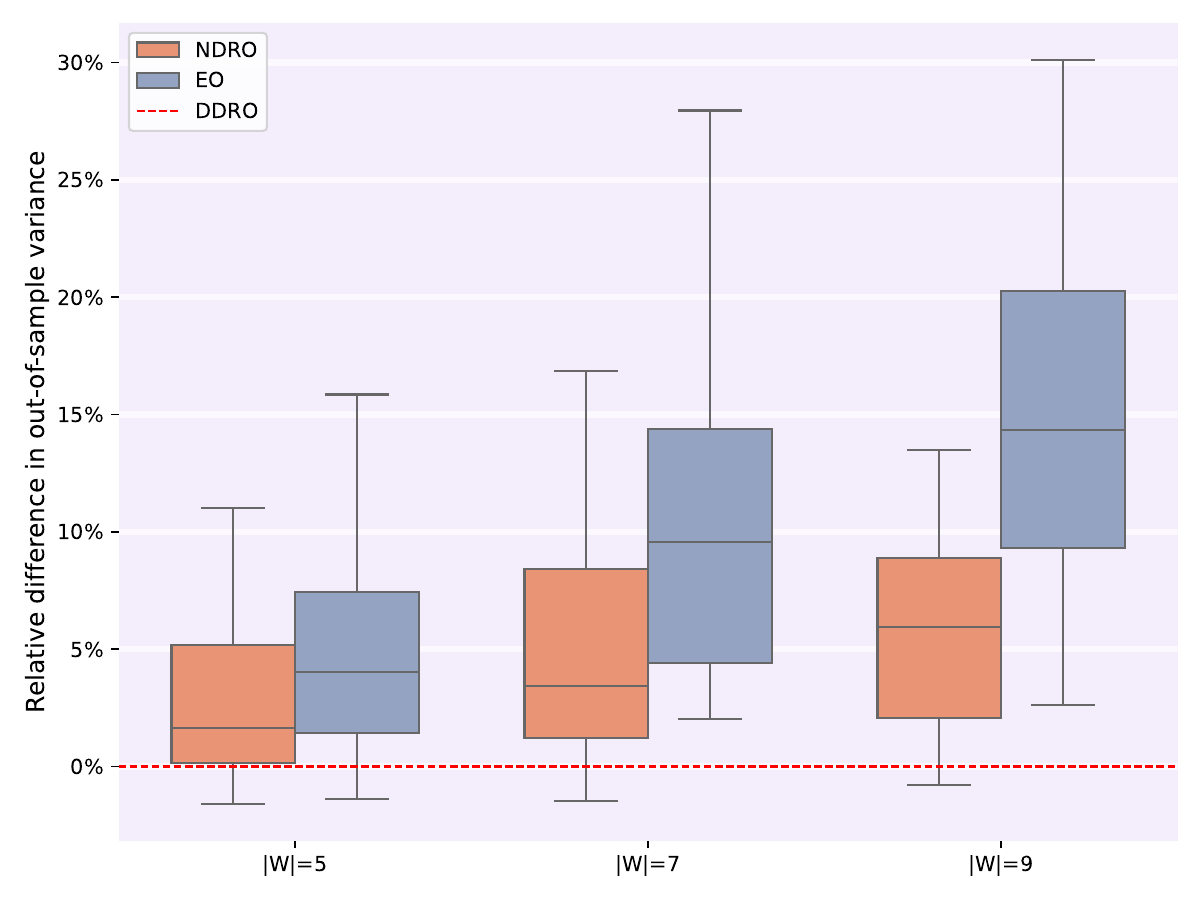}}
    }
    {
    Box Plots of Out-of-sample Smoothing Performance Using Different Methods  \label{fig:synthetic_data_var_boxplot}}
    {}
\end{figure}

Additionally, we use the out-of-sample variance of the aggregated wind power output to quantify the smoothing performance. The relative differences in out-of-sample smoothing performance are illustrated in Figure~\ref{fig:synthetic_data_var_boxplot}, where positive values indicate that DDRO achieves a certain percentage reduction in variance compared to benchmark methods. The results highlight that DDRO’s capacity allocation decisions consistently result in smaller variances and more stable aggregated wind power generation. This provides a clear and intuitive demonstration of DDRO’s advantages, showcasing its ability to effectively leverage the wind power smoothing effect to mitigate fluctuations more efficiently than benchmark methods. Moreover, the improvement in smoothing performance becomes more pronounced as the number of wind farms increases, demonstrating the model’s strong potential for application in the context of distributed wind energy development.

\subsection{Data with Correlated Structures}\label{subsec:rd}

In this section, we evaluate the performance of the proposed methods using real wind power resource data, as mentioned in Section \ref{subsec:eadi}. We make capacity decisions for 4 wind farms located at nodes \{37, 49, 51, 63\}, with the data upscaled to match the IEEE-118 bus system. All other experimental settings remain consistent with those used in the synthetic data experiments.

We first assess the average computational performance of the three solution methods (DC, CG, and CG-L). The results, as presented in Table \ref{tab:real data}, are generally consistent with those from the synthetic experiments. The application of CG and CG-L significantly accelerates the solving process. However, the real data experiments reveal that CG alone is less effective, particularly as the problem complexity increases. In contrast, the combined approach of constraint generation and the L-shaped algorithm (CG-L) demonstrates greater stability and reliability, consistently delivering solutions within reasonable timeframes. This highlights CG-L as the preferred method for handling large-scale and complex problem instances in real-world applications.

\vspace{-0.5cm}
\begin{table}[H]
    \TABLE
    {Average Computational Time (in Seconds)\label{tab:real data}}
    {
    \centering
    \begin{tabular}{ccccccccccccc}
    \hline
               &           & \multicolumn{3}{c}{420}                &           & \multicolumn{3}{c}{380}                &           & \multicolumn{3}{c}{350}               \\ \cline{3-5} \cline{7-9} \cline{11-13} 
    W & N         & DC       & CG            & CG-L          &           & DC      & CG            & CG-L           &           & DC        & CG            & CG-L           \\ \hline
    4   & 30      & 384.64   & \textbf{4.37} & 5.75          & \textbf{} & 352.77  & \textbf{4.91} & 7.52           & \textbf{} & 328.30    & \textbf{4.52} & 6.87           \\
        & 50      & 969.04   & 30.51         & \textbf{8.98} & \textbf{} & 903.05  & 48.19         & \textbf{12.42} & \textbf{} & 952.76    & 42.88         & \textbf{12.26} \\
        & 80      & 1739.12  & 86.58         & \textbf{12.10} & \textbf{} & 1796.13 & 123.27        & \textbf{15.13} & \textbf{} & 1665.09   & 114.05        & \textbf{13.82}  \\ \hline
    \end{tabular}}
    {}
\end{table}
\vspace{-0.5cm}

\vspace{-0.5cm}
\begin{figure}[H]
\FIGURE
{
\subcaptionbox{Risk Management Cost \label{fig:real_data_boxplot}}
{\includegraphics[width=0.32\textwidth]{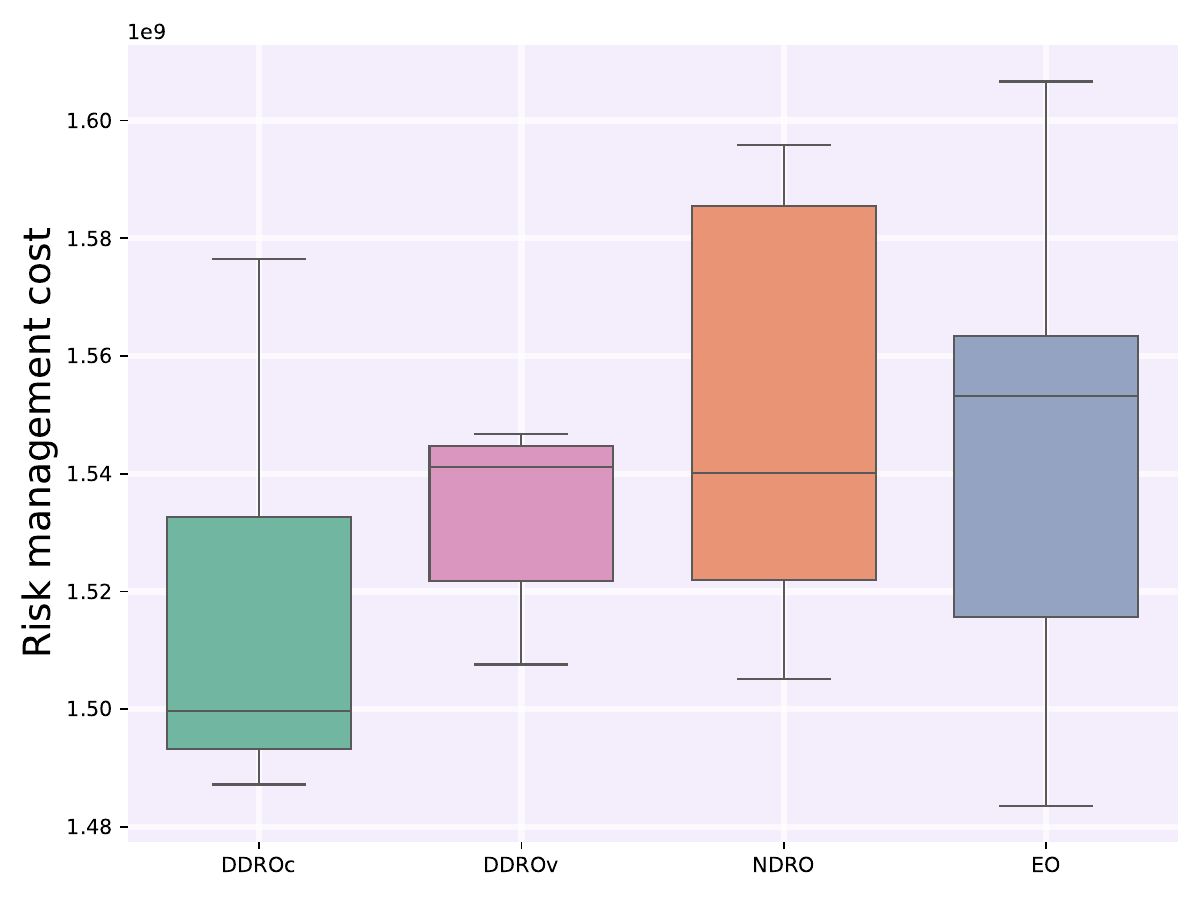}}
\subcaptionbox{Smoothing Performance \label{fig:real_data_var_boxplot}}
{\includegraphics[width=0.32\textwidth]{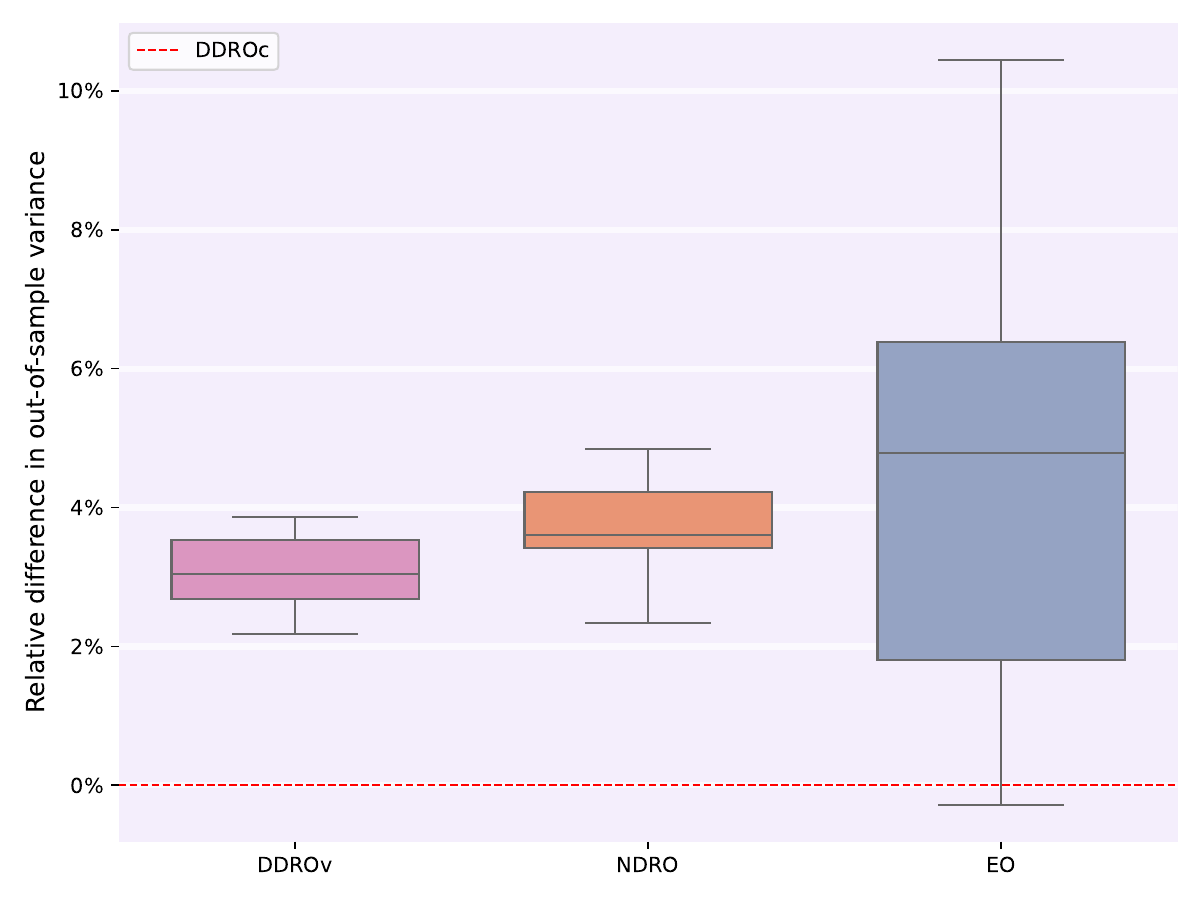}}
}
{
Box Plots of Out-of-sample Metrics Using Real Data\label{fig:real_data_plot}}
{}
\end{figure}
\vspace{-0.5cm}

Due to the limited availability of real data, we generate datasets that match the marginal distributions and correlation structure of the real data (see Appendix \ref{sec:dg}), with 3000 samples for training and another 3000 samples for testing. We fix the total installed capacity at $X=500$. 
For the DDRO method, we consider two variants: one that incorporates only variance information (denoted as DDROv), and another that includes both variance and correlation information (denoted as DDROc).
For each run, we randomly select 60 samples from the training dataset to train the model. 
The out-of-sample metric is defined as the average risk management cost over the entire testing dataset. We conduct 20 repeated experiments and visualize the out-of-sample risk management costs as box plots in Figure \ref{fig:real_data_plot}.

Figure~\ref{fig:real_data_boxplot} presents the results for risk management cost, highlighting DDRO’s robust performance under real-world conditions. Specifically, both DDROc and DDROv outperform the benchmark methods NDRO and EO in terms of median cost, reflecting the superior risk-hedging capability of decision-dependent ambiguity sets. Notably, while DDROc exhibits greater variability, it achieves the lowest median cost among all methods, indicating that incorporating correlation information can enhance risk management. However, the wider spread also reveals a trade-off: when correlation estimates are imprecise, their inclusion may introduce volatility into decision outcomes. This underscores the practical consideration that model sophistication should be matched with data quality.

To further evaluate model performance from a system-level stability perspective, Figure~\ref{fig:real_data_var_boxplot} shows the relative differences in out-of-sample variance of the aggregated wind power generation. The results are consistent with those observed in the synthetic experiments and demonstrate that DDRO-based models consistently yield lower variance. In particular, both DDROc and DDROv outperform NDRO and EO, with DDROc achieving the best smoothing performance overall. 

Interestingly, while risk management cost and smoothing performance are generally aligned, the variability across methods reveals important distinctions. In particular, DDROc exhibits relatively large fluctuations in risk management cost but maintains consistently low variance in smoothing performance. This suggests that, although incorporating correlation information may lead to instability in cost outcomes—especially under estimation errors—it remains effective in ensuring stable aggregated wind output. Conversely, EO shows high variability in both metrics, indicating that its performance is more susceptible to uncertainty, both in terms of cost and output stability.
This divergence highlights a key managerial insight: minimizing expected cost is not sufficient for ensuring robust and reliable system performance. Instead, smoothing performance—reflected by output variance—is more directly tied to capacity allocation decisions and long-term operational resilience. As such, we place greater emphasis on smoothing performance when evaluating model effectiveness. The consistently strong variance reduction achieved by DDRO models, particularly under real data, underscores their value in managing uncertainty and promoting stable integration of renewable energy in large-scale systems.

\subsection{Value of Variance and Correlation in the Smoothing effect}\label{subsec:vvc}

The regularization terms in DDRO (\ref{radius2}) and the benchmark NDRO (\ref{NDRO:mr}) inherently sensitize capacity allocation decisions to the statistical properties of wind power resources. Specifically, DDROv becomes equivalent to NDRO when the variances of wind power resources are uniform across sites, underscoring the unique contribution of DDRO in leveraging site-specific variance information. Similarly, DDROc converges to DDROv when inter-farm correlations are nullified, highlighting the distinct role of correlation. 

To systematically investigate the marginal value of these two crucial distributional features-heterogeneity in site-specific variances and inter-farm correlation-we design experiments that allow us to modulate their influence. We quantify the heterogeneity in variances by considering the variance of the site-specific variances themselves; a higher variance of variances indicates greater diversity in the risk profiles of individual sites. Inter-farm correlation is directly represented by the off-diagonal elements of the covariance matrix. We then introduce a scaling parameter, $\omega$, which simultaneously multiplies both the deviations of individual site variances from their mean (thus amplifying or dampening the heterogeneity of variances) and the off-diagonal covariance terms (thus strengthening or weakening the inter-farm correlations). We use synthetic datasets based on a multivariate normal distribution, and out-of-sample variance of aggregated wind power output is used as a proxy for smoothing performance, with experimental settings aligned with Section~\ref{subsec:rd}.

\begin{figure}[H]
    \FIGURE
    {
    \subcaptionbox{Sample Variance \label{fig:dla_var_sample}}
    {\includegraphics[width=0.4\textwidth]{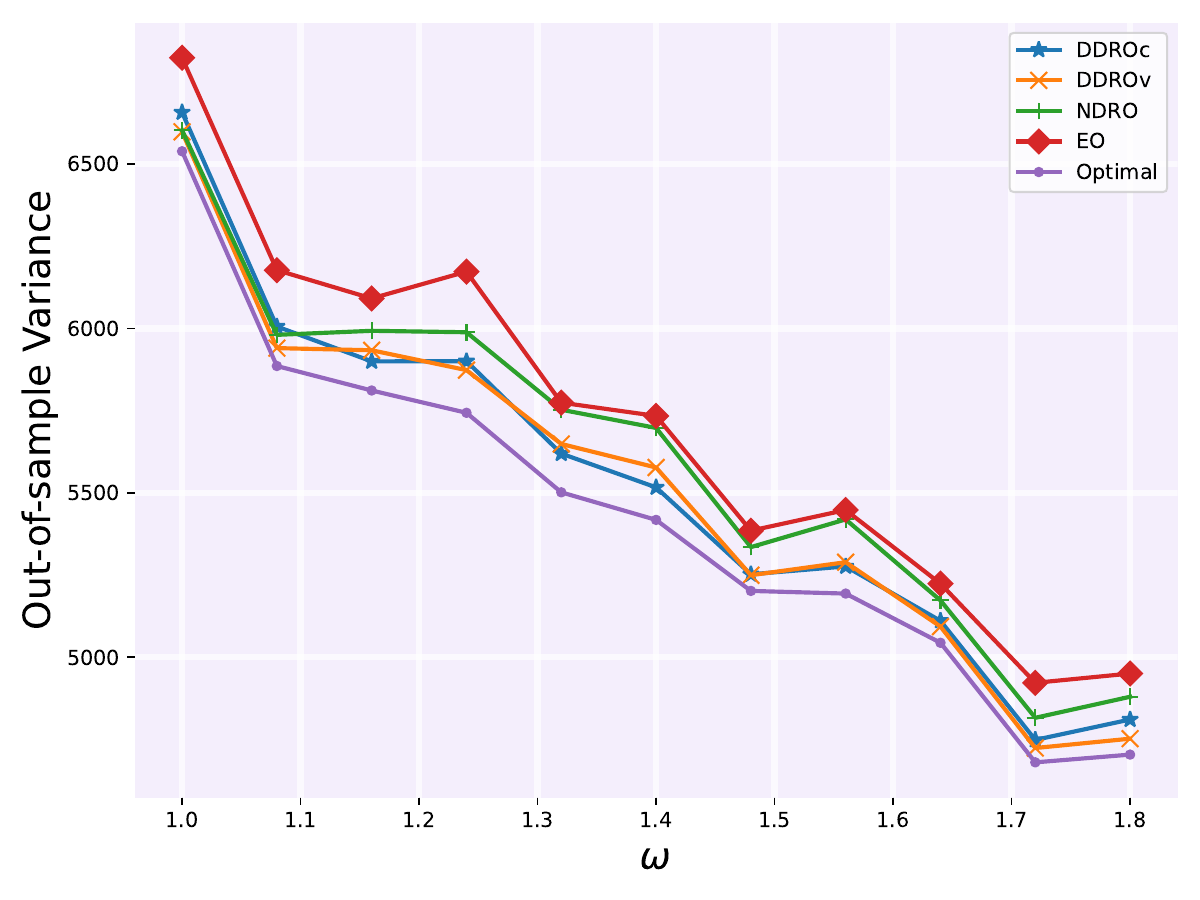}}
    \subcaptionbox{True Variance \label{fig:dla_var_true}}
    {\includegraphics[width=0.4\textwidth]{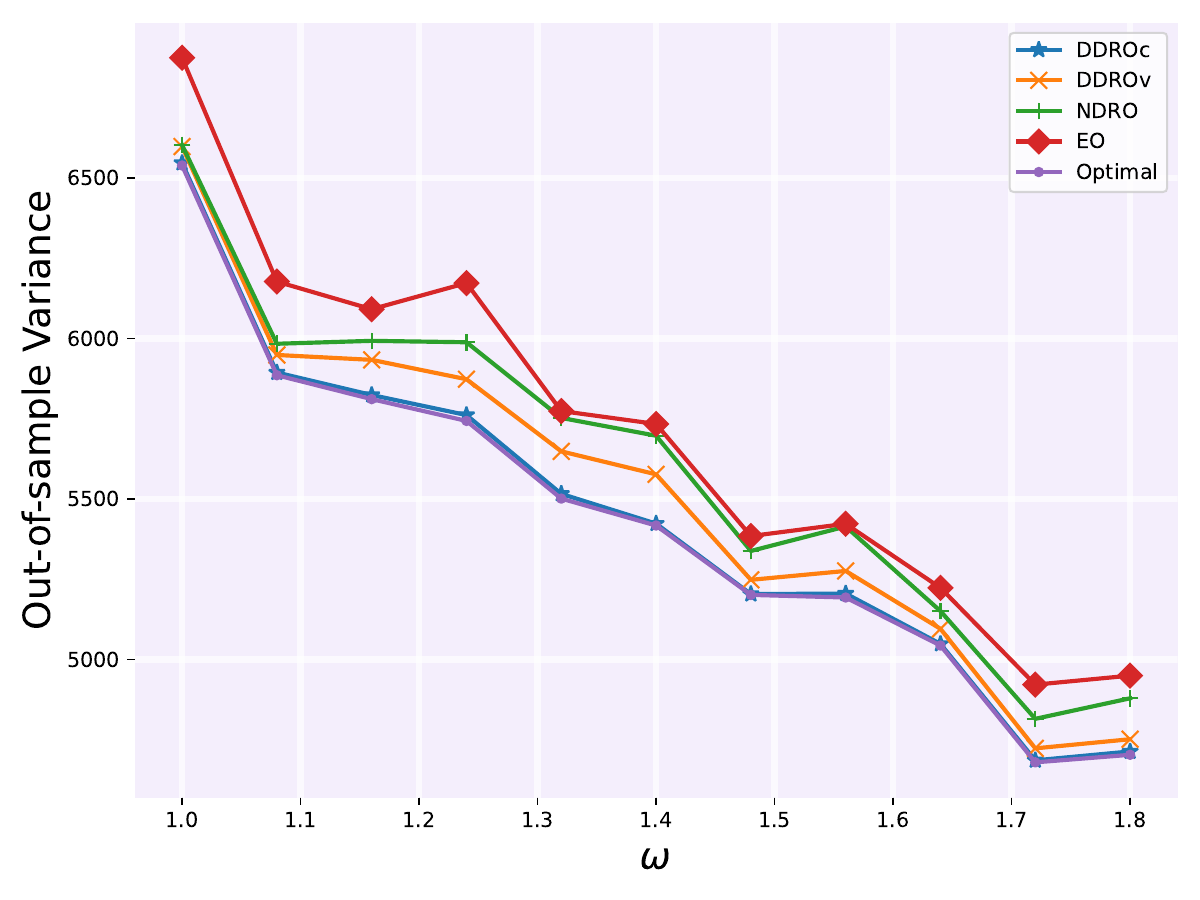}}
    }
    {
    Out-of-sample Smoothing Performance with Different Variance}
    {}
\end{figure}

\begin{figure}[H]
    \FIGURE
    {
    \subcaptionbox{Sample Correlation Matrix  \label{fig:dla_cov_sample}}
    {\includegraphics[width=0.4\textwidth]{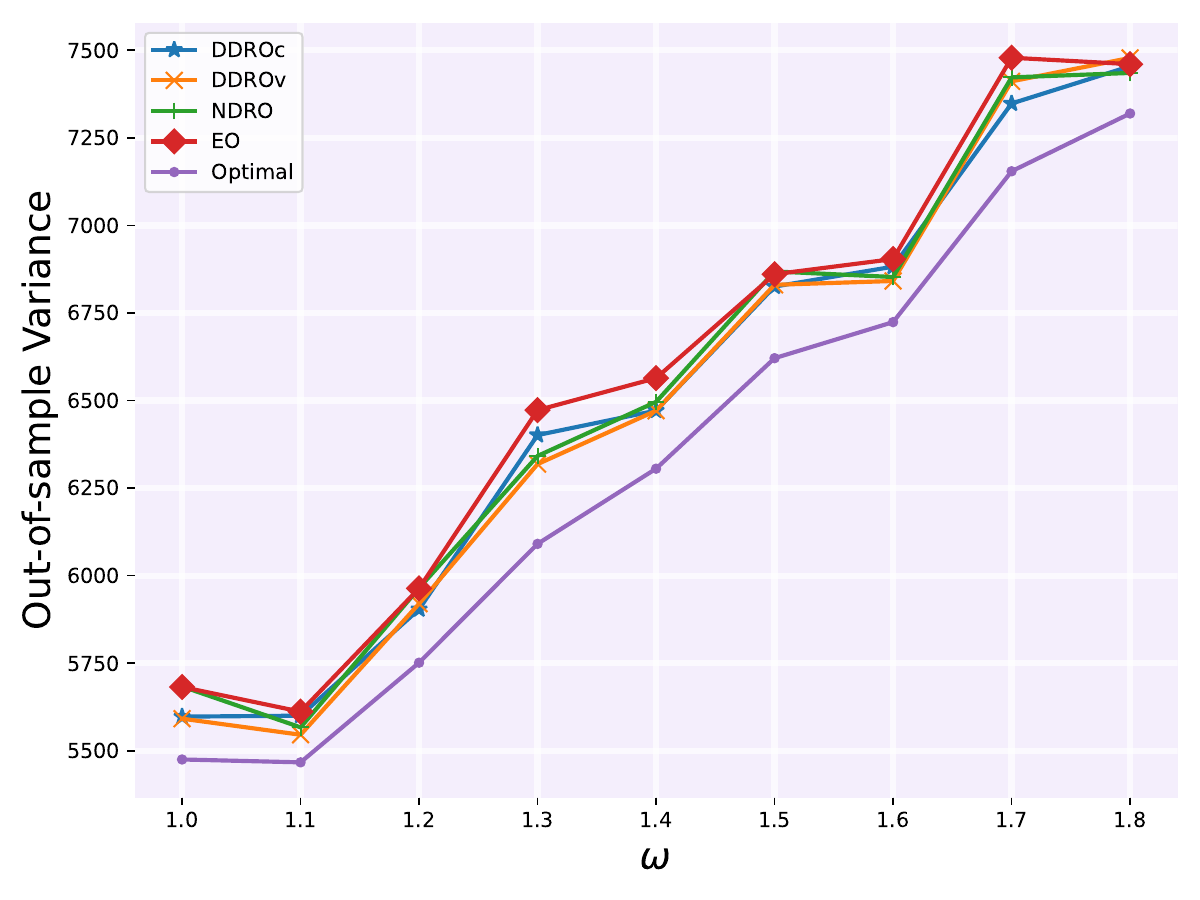}}
    \subcaptionbox{True Correlation Matrix \label{fig:dla_cov_true}}
    {\includegraphics[width=0.4\textwidth]{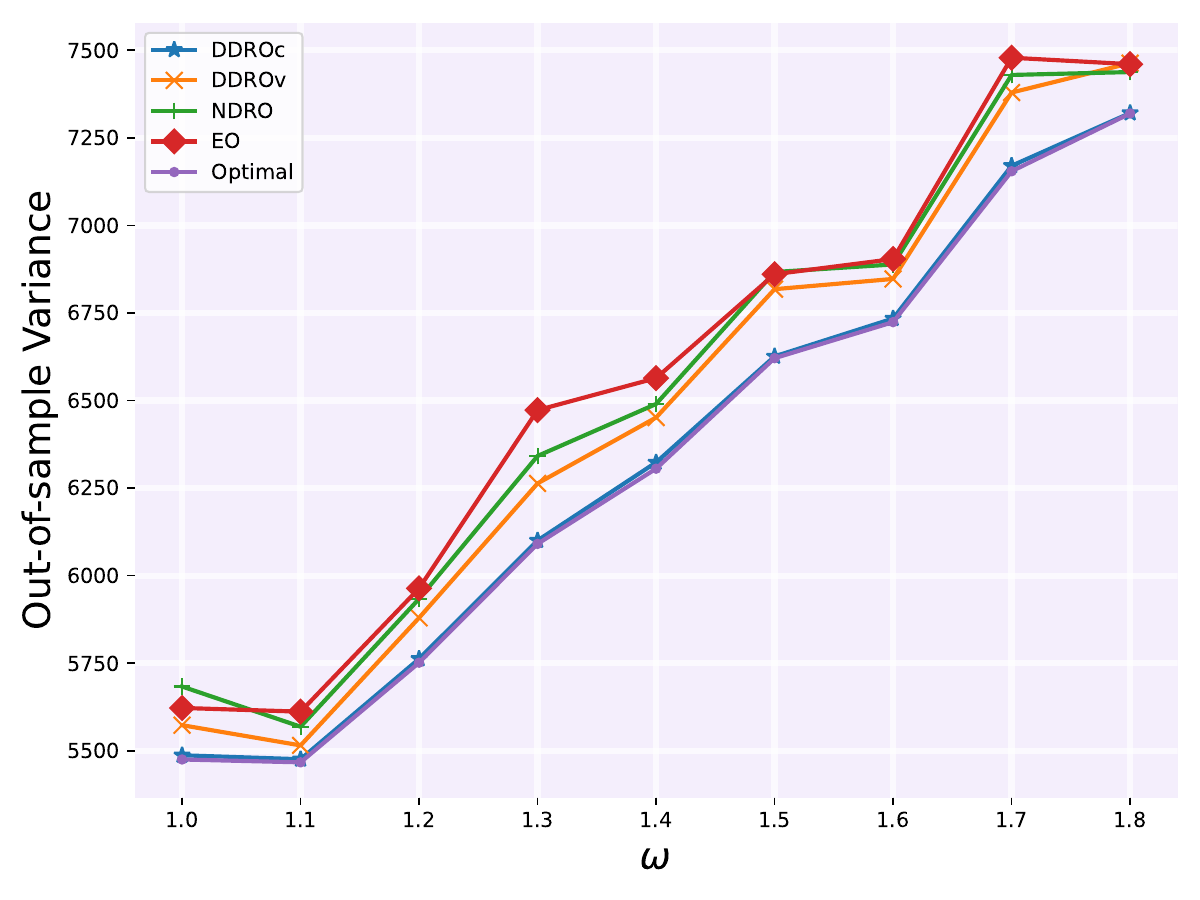}}
    }
    {
    Out-of-sample Smoothing Performance with Different Covariance}
    {}
\end{figure}

Our experiments first consider scenarios where decision-makers rely on sample-based estimates of variance and covariance, a common reality in practical planning horizons where perfect foresight is unavailable (Figures \ref{fig:dla_var_sample} and \ref{fig:dla_cov_sample}). We have the following observations:
\begin{enumerate}[label=\Roman*.] 
    \item In general, DDRO achieves superior smoothing effect compared to NDRO and EO. Even without considering interdependencies, acknowledging differential variance across sites leads to more resilient capacity allocation.
    \item DDROc, despite its theoretical sophistication in capturing interdependencies, occasionally performs worse than the simpler DDROv ($\omega=2$ in Figure \ref{fig:dla_cov_sample}). This suggests that inaccurate correlation estimates can mislead the optimization, leading to capacity allocations that chase spurious patterns, ultimately degrading the smoothing effect.
    \item The relatively strong performance of DDROv, even when DDROc falters due to poor correlation estimates, suggests that for practical smoothing purposes, variance information alone carries the major share of the benefit.
\end{enumerate}

We then analyze performance using the true variance and covariance, as shown in Figures~\ref{fig:dla_var_true} and \ref{fig:dla_cov_true}.
The results here paint a clearer picture of theoretical advantages. DDROv continues to demonstrate near-optimal smoothing, reinforcing its efficacy. However, when armed with perfect information about inter-farm correlations, DDROc consistently and significantly outperforms DDROv. This confirms the intrinsic value of understanding and leveraging these interdependencies for maximizing the geographic smoothing effect.

Synthesizing across these scenarios, several key insights for wind farm planning and renewable integration emerge. First, given that DDROv captures a substantial portion of the smoothing effect while remaining computationally efficient (by avoiding the full covariance structure), it serves as a foundational and pragmatic strategy for decision-making, especially in data-limited settings. Second, while true correlations are valuable, estimated correlations from limited or noisy datasets should be treated with caution. Blindly incorporating them into complex models (like DDROc) without assessing data quality can be detrimental. If reliable methods for forecasting or estimating inter-farm correlations can be developed, the incremental benefits in system smoothing and cost reduction can be substantial. Our findings offer a practical illustration of the value of information: begin with DDROv, and only graduate to DDROc as confidence in correlation data improves.

\section{Conclusion}\label{sec:concl}

In this paper, we propose a novel framework for wind farm planning using a DRO approach with decision-dependent Wasserstein ambiguity sets. The problem is formulated as a two-stage DRO model that integrates joint planning and operations under decision-dependent uncertainty. Leveraging the intrinsic characteristics of wind power resources, we define the distribution function and derive the radius function of the ambiguity set. Utilizing duality theory, we reformulate the model into a MISOCP problem and establish an out-of-sample performance guarantee. On the algorithmic side, we develop an efficient solution framework to tackle the computational challenges. Numerical experiments validate the effectiveness of DDRO in managing risk and mitigating uncertainty in wind power generation. Our findings demonstrate the potential of decision-dependent Wasserstein ambiguity sets for practical applications, particularly in scenarios where the distributions exhibit log-concavity. Looking ahead, a promising direction for future research involves extending this framework to accommodate multi-dimensional random variables within Wasserstein ambiguity sets, further enhancing its applicability and robustness in real-world scenarios.

\bibliographystyle{informs2014}
\bibliography{DRO_WFP}

\begin{thebibliography}{43}
\providecommand{\natexlab}[1]{#1}
\providecommand{\url}[1]{\texttt{#1}}
\providecommand{\urlprefix}{URL }

\bibitem[{Asari et~al.(2002)Asari, Nanahara, Maejima, Yamaguchi, \protect\BIBand{} Sato}]{asari2002study}
Asari M, Nanahara T, Maejima T, Yamaguchi K, Sato T (2002) A study on smoothing effect on output fluctuation of distributed wind power generation. \emph{IEEE/PES Transmission and distribution Conference and Exhibition}, volume~2, 938--943 (IEEE).

\bibitem[{Basciftci et~al.(2021)Basciftci, Ahmed, \protect\BIBand{} Shen}]{basciftci2021distributionally}
Basciftci B, Ahmed S, Shen S (2021) Distributionally robust facility location problem under decision-dependent stochastic demand. \emph{European Journal of Operational Research} 292(2):548--561.

\bibitem[{Bertsekas(2009)}]{bertsekas2009convex}
Bertsekas D (2009) \emph{Convex optimization theory}, volume~1 (Athena Scientific).

\bibitem[{Bitaraf \protect\BIBand{} Rahman(2017)}]{bitaraf2017reducing}
Bitaraf H, Rahman S (2017) Reducing curtailed wind energy through energy storage and demand response. \emph{IEEE Transactions on Sustainable Energy} 9(1):228--236.

\bibitem[{BloombergNEF(2024)}]{bnef_energy_transition_investment_trends_2024}
BloombergNEF (2024) Energy transition investment trends 2024. \url{https://assets.bbhub.io/professional/sites/24/Energy-Transition-Investment-Trends-2024.pdf}, accessed: July 1, 2024.

\bibitem[{Bobkov \protect\BIBand{} Ledoux(2019)}]{bobkov2019one}
Bobkov S, Ledoux M (2019) \emph{One-dimensional empirical measures, order statistics, and Kantorovich transport distances}, volume 261 (American Mathematical Society).

\bibitem[{Bobkov(1999)}]{bobkov1999isoperimetric}
Bobkov SG (1999) Isoperimetric and analytic inequalities for log-concave probability measures. \emph{The Annals of Probability} 27(4):1903--1921.

\bibitem[{Bowden et~al.(1983)Bowden, Barker, Shestopal, \protect\BIBand{} Twidell}]{bowden1983weibull}
Bowden G, Barker P, Shestopal V, Twidell J (1983) The weibull distribution function and wind power statistics. \emph{Wind Engineering} 85--98.

\bibitem[{Chen et~al.(2023)Chen, Kuhn, \protect\BIBand{} Wiesemann}]{chen2023approximations}
Chen Z, Kuhn D, Wiesemann W (2023) On approximations of data-driven chance constrained programs over wasserstein balls. \emph{Operations Research Letters} 51(3):226--233.

\bibitem[{Chen et~al.(2024)Chen, Kuhn, \protect\BIBand{} Wiesemann}]{chen2024data}
Chen Z, Kuhn D, Wiesemann W (2024) Data-driven chance constrained programs over wasserstein balls. \emph{Operations Research} 72(1):410--424.

\bibitem[{Doan(2022)}]{doan2022distributionally}
Doan XV (2022) Distributionally robust optimization under endogenous uncertainty with an application in retrofitting planning. \emph{European Journal of Operational Research} 300(1):73--84.

\bibitem[{Dunn et~al.(2024)Dunn, Gangrade, Wasserman, \protect\BIBand{} Ramdas}]{dunn2024universal}
Dunn R, Gangrade A, Wasserman L, Ramdas A (2024) Universal inference meets random projections: a scalable test for log-concavity. \emph{Journal of Computational and Graphical Statistics} 1--13.

\bibitem[{Elia(2024)}]{elia_wind_power}
Elia (2024) Wind power generation. \url{https://www.elia.be/en/grid-data/generation-data/wind-power-generation}, accessed: July 1, 2024.

\bibitem[{{Energy Education}(2023)}]{wind_resource_measurement}
{Energy Education} (2023) Wind resource measurement. \url{https://energyeducation.ca/encyclopedia/Wind_resource_measurement}, accessed: July 1, 2024.

\bibitem[{Esteban-P{\'e}rez \protect\BIBand{} Morales(2023)}]{esteban2023distributionally}
Esteban-P{\'e}rez A, Morales JM (2023) Distributionally robust optimal power flow with contextual information. \emph{European Journal of Operational Research} 306(3):1047--1058.

\bibitem[{Gangrade et~al.(2023)Gangrade, Rinaldo, \protect\BIBand{} Ramdas}]{gangrade2023sequential}
Gangrade A, Rinaldo A, Ramdas A (2023) A sequential test for log-concavity. \emph{arXiv preprint arXiv:2301.03542} .

\bibitem[{Gao et~al.(2024)Gao, Chen, \protect\BIBand{} Kleywegt}]{gao2024wasserstein}
Gao R, Chen X, Kleywegt AJ (2024) Wasserstein distributionally robust optimization and variation regularization. \emph{Operations Research} 72(3):1177--1191.

\bibitem[{Garcia et~al.(1998)Garcia, Torres, Prieto, \protect\BIBand{} De~Francisco}]{garcia1998fitting}
Garcia A, Torres J, Prieto E, De~Francisco A (1998) Fitting wind speed distributions: a case study. \emph{Solar energy} 62(2):139--144.

\bibitem[{GEWC(2017)}]{gewc_global_wind_report_2017}
GEWC (2017) Global wind report 2017. \url{https://gwec.net/global-wind-report-2017}, accessed: July 1, 2024.

\bibitem[{GEWC(2024)}]{gewc_global_wind_report_2024}
GEWC (2024) Global wind report 2024. \url{https://gwec.net/global-wind-report-2024}, accessed: July 1, 2024.

\bibitem[{Gonz{\'a}lez et~al.(2014)Gonz{\'a}lez, Pay{\'a}n, Santos, \protect\BIBand{} Gonz{\'a}lez-Longatt}]{gonzalez2014review}
Gonz{\'a}lez JS, Pay{\'a}n MB, Santos JMR, Gonz{\'a}lez-Longatt F (2014) A review and recent developments in the optimal wind-turbine micro-siting problem. \emph{Renewable and Sustainable Energy Reviews} 30:133--144.

\bibitem[{IEA(2023)}]{iea_renewables_2023}
IEA (2023) Renewables 2023. \url{https://www.iea.org/reports/renewables-2023}, accessed: July 1, 2024.

\bibitem[{IEA(2024)}]{iea_world_energy_investment_2024}
IEA (2024) World energy investment 2024. \url{https://www.iea.org/reports/world-energy-investment-2024#overview}, accessed: July 1, 2024.

\bibitem[{Islam et~al.(2011)Islam, Saidur, \protect\BIBand{} Rahim}]{islam2011assessment}
Islam M, Saidur R, Rahim N (2011) Assessment of wind energy potentiality at kudat and labuan, malaysia using weibull distribution function. \emph{Energy} 36(2):985--992.

\bibitem[{Jung \protect\BIBand{} Schindler(2019)}]{jung2019changing}
Jung C, Schindler D (2019) Changing wind speed distributions under future global climate. \emph{Energy Conversion and Management} 198:111841.

\bibitem[{Kaggle(2024)}]{kaggle_2024}
Kaggle (2024) Wind power generation data. \url{https://www.kaggle.com/datasets/mubashirrahim/wind-power-generation-data-forecasting}, accessed: 2024-7-1.

\bibitem[{Li et~al.(2018)Li, Mathieu, \protect\BIBand{} Jiang}]{li2018distributionally}
Li B, Mathieu JL, Jiang R (2018) Distributionally robust chance constrained optimal power flow assuming log-concave distributions. \emph{2018 Power Systems Computation Conference (PSCC)}, 1--7 (IEEE).

\bibitem[{Martinez \protect\BIBand{} Iglesias(2024)}]{martinez2024global}
Martinez A, Iglesias G (2024) Global wind energy resources decline under climate change. \emph{Energy} 288:129765.

\bibitem[{Mohajerin~Esfahani \protect\BIBand{} Kuhn(2018)}]{mohajerin2018data}
Mohajerin~Esfahani P, Kuhn D (2018) Data-driven distributionally robust optimization using the wasserstein metric: Performance guarantees and tractable reformulations. \emph{Mathematical Programming} 171(1):115--166.

\bibitem[{Niu et~al.(2022)Niu, Yang, Zhang, Sun, \protect\BIBand{} Wang}]{niu2022missing}
Niu F, Yang C, Zhang J, Sun Y, Wang R (2022) Missing wind speed imputation research for wind farm considering wake effect. \emph{2022 2nd International Conference on Computer, Control and Robotics (ICCCR)}, 120--125 (IEEE).

\bibitem[{Noyan et~al.(2022)Noyan, Rudolf, \protect\BIBand{} Lejeune}]{noyan2022distributionally}
Noyan N, Rudolf G, Lejeune M (2022) Distributionally robust optimization under a decision-dependent ambiguity set with applications to machine scheduling and humanitarian logistics. \emph{INFORMS Journal on Computing} 34(2):729--751.

\bibitem[{Rigollet \protect\BIBand{} H{\"u}tter(2023)}]{rigollet2023high}
Rigollet P, H{\"u}tter JC (2023) High-dimensional statistics. \emph{arXiv preprint arXiv:2310.19244} .

\bibitem[{Rockafellar et~al.(2000)Rockafellar, Uryasev et~al.}]{rockafellar2000optimization}
Rockafellar RT, Uryasev S, et~al. (2000) Optimization of conditional value-at-risk. \emph{Journal of risk} 2:21--42.

\bibitem[{Saeed et~al.(2021)Saeed, Ahmed, Hussain, \protect\BIBand{} Zhang}]{saeed2021wind}
Saeed MA, Ahmed Z, Hussain S, Zhang W (2021) Wind resource assessment and economic analysis for wind energy development in pakistan. \emph{Sustainable Energy Technologies and Assessments} 44:101068.

\bibitem[{Shapiro(2001)}]{shapiro2001duality}
Shapiro A (2001) On duality theory of conic linear problems. \emph{Nonconvex Optimization and its Applications} 57:135--155.

\bibitem[{Tong(2010)}]{tong2010fundamentals}
Tong W (2010) \emph{Fundamentals of wind energy}, volume~44 (WIT press Southampton, UK).

\bibitem[{Wang et~al.(2018)Wang, Gao, Qiu, Wang, \protect\BIBand{} Xin}]{wang2018risk}
Wang C, Gao R, Qiu F, Wang J, Xin L (2018) Risk-based distributionally robust optimal power flow with dynamic line rating. \emph{IEEE Transactions on Power Systems} 33(6):6074--6086.

\bibitem[{Wang et~al.(2024)Wang, Zhang, Sun, Trivedi, Chung, \protect\BIBand{} Srinivasan}]{wang2024wind}
Wang S, Zhang W, Sun Y, Trivedi A, Chung C, Srinivasan D (2024) Wind power forecasting in the presence of data scarcity: A very short-term conditional probabilistic modeling framework. \emph{Energy} 291:130305.

\bibitem[{Wasserman et~al.(2020)Wasserman, Ramdas, \protect\BIBand{} Balakrishnan}]{wasserman2020universal}
Wasserman L, Ramdas A, Balakrishnan S (2020) Universal inference. \emph{Proceedings of the National Academy of Sciences} 117(29):16880--16890.

\bibitem[{Yang et~al.(2019)Yang, Zhang, Cui, Zhou, Chen, \protect\BIBand{} Yan}]{yang2019investigating}
Yang M, Zhang L, Cui Y, Zhou Y, Chen Y, Yan G (2019) Investigating the wind power smoothing effect using set pair analysis. \emph{IEEE Transactions on Sustainable Energy} 11(3):1161--1172.

\bibitem[{Yin et~al.(2022{\natexlab{a}})Yin, Feng, \protect\BIBand{} Hou}]{yin2022stochastic}
Yin W, Feng S, Hou Y (2022{\natexlab{a}}) Stochastic wind farm expansion planning with decision-dependent uncertainty under spatial smoothing effect. \emph{IEEE Transactions on Power Systems} 38(3):2845--2857.

\bibitem[{Yin et~al.(2022{\natexlab{b}})Yin, Li, Hou, Miao, \protect\BIBand{} Hou}]{yin2022coordinated}
Yin W, Li Y, Hou J, Miao M, Hou Y (2022{\natexlab{b}}) Coordinated planning of wind power generation and energy storage with decision-dependent uncertainty induced by spatial correlation. \emph{IEEE Systems Journal} 17(2):2247--2258.

\bibitem[{Zhu et~al.(2019)Zhu, Wei, \protect\BIBand{} Bai}]{zhu2019wasserstein}
Zhu R, Wei H, Bai X (2019) Wasserstein metric based distributionally robust approximate framework for unit commitment. \emph{IEEE Transactions on Power Systems} 34(4):2991--3001.

\end{thebibliography}

\clearpage

\begin{APPENDICES}

\section{The Result of Tests} \label{sec:test_all}

As is shown in Figure \ref{fig:logconcave_all}, all simulation results (across 24 time periods and 4 seasons) fail to reject the null hypothesis, even when the critical statistic value is set to 0.7, a value well below any practical $\alpha \in (0,1)$. These outcomes strongly validate the hypotheses of log-concavity for the underlying distributions.

\begin{figure}[H]
\FIGURE
{\includegraphics[width=1\textwidth]{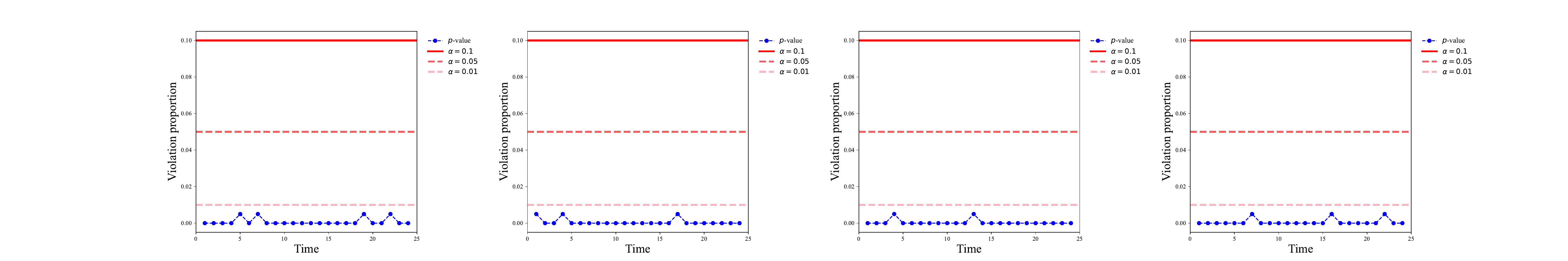}}
{Test Results for Log-Concavity\label{fig:logconcave_all}}
{}
\end{figure}

\section{Extensions of the Model} \label{sec:em}
We provide model extensions to address two important sources of intra-farm output variability: (i) fluctuations due to variable local conditions, and (ii) output heterogeneity arising from the wake effect.

\begin{itemize}
    \item Variable local conditions: To account for variability caused by local conditions, we can decompose the output of each wind turbine into the sum of a common component $\xi_w$ (representing the dominant wind condition within a farm) and a noise term $\epsilon$, i.e., $\xi_w + \epsilon$, where $\epsilon \sim N(0, \sigma^2)$ is an independent zero-mean Gaussian noise capturing local fluctuations. This decomposition preserves the analytical structure of our model, and all theoretical results remain valid under this refined formulation.

    \item Wake effect: When internal variability is driven by the wake effect—i.e., variations in wind speed due to turbine interactions within the same wind farm \citep{niu2022missing}—we extend the formulation by decomposing $\xi_w$ into multiple components, such as $\xi_w^1$ and $\xi_w^2$, corresponding to regions with different wind speeds $V_w^1$ and $V_w^2$. Following wake modeling literature \citet{gonzalez2014review}, these wind speeds are linearly correlated, and by the wind power generation formula, the components $\xi_w^1$ and $\xi_w^2$ are also linearly related. This extension also retains the analytical tractability of our original model.
\end{itemize}

\section{Proof in Section \ref{sec:ddas}}

\begin{lemma}{\citep{bobkov2019one}}\label{lemma2}
    Assume that $\mu$ satisfies $Poincar\acute{e}\text{-}type$ inequality on the real line with constant $\lambda >0$. Then, for any $p \geq 1$ and $\varepsilon > 0$, and any $N \geq 1$, 
     \begin{equation} 
     \mathbb{P}\Big\{ \left|W_p(\xi_N,\xi)-\mathbb{E}\big(W_p(\xi_N,\xi)\big)\big|\geq \varepsilon\right\} \leq C_1\exp\big\{ -2N^{1/\max(p,2)}\sqrt{\lambda} \varepsilon\big\}\text{,} 
     \end{equation}
    where $C_1 > 0$ is an absolute constant.
\end{lemma}

\begin{proof}{Proof of Lemma \ref{lemma:mc}\label{prf:mc}}
    \quad
    
    We first consider the right-hand side of inequality in Lemma \ref{lemma2}. Based on Corollary 4.3 \citep{bobkov1999isoperimetric}, we have
    \begingroup
    \setlength{\abovedisplayskip}{0pt}
    \setlength{\belowdisplayskip}{0pt}
    \begin{align}
        \mathbb{P}\Big\{ \left|W_p({\hat{\mathbb{P}}},\mathbb{P})-\mathbb{E}\big(W_p({\hat{\mathbb{P}}},\mathbb{P})\big)\big|\geq \varepsilon\right\} &\leq C_1\exp\big\{ -2N^{1/2}\sqrt{\lambda} \varepsilon\big\} \\
        &\leq C_1\exp\big\{ -N^{1/2}3^{-1/2}\sigma^{-1}\varepsilon\big\} \text{.}
    \end{align}
    \endgroup
    
    Then, consider the left-hand side of the inequality and we have
    \begin{equation}\label{proof:lh}
         \mathbb{P}\Big\{ \left|W_p({\hat{\mathbb{P}}},\mathbb{P})-\mathbb{E}\big(W_p({\hat{\mathbb{P}}},\mathbb{P})\big)\big|\geq \varepsilon\right\} \geq \mathbb{P}\left\{W_p({\hat{\mathbb{P}}},\mathbb{P})-\mathbb{E}\big(W_p({\hat{\mathbb{P}}},\mathbb{P})\big)\geq \varepsilon\right\} \text{.}
    \end{equation}
    
    Due to the convexity of $f(x)=x^p(x\geq 0, p\geq1)$ and Jensen's inequality, we have
    \begin{equation}
        \mathbb{E}\big(W_p({\hat{\mathbb{P}}},\mathbb{P})\big) \leq\left(\mathbb{E}\big(W_p^p({\hat{\mathbb{P}}},\mathbb{P})\big)\right)^{1/p} \leq \left(\frac {C_2}{2-p}\right)^{1/p} \left(\frac\sigma{\sqrt{N}}\right) \text{,}
    \end{equation}
    where the second inequality follows from Corollary 6.10 \citep{bobkov2019one}. Then, we have: 
    \begin{equation}\label{proof:rh}
        \mathbb{P}\left\{W_p({\hat{\mathbb{P}}},\mathbb{P})-\mathbb{E}\big(W_p({\hat{\mathbb{P}}},\mathbb{P})\big)\geq \varepsilon\right\} \geq \mathbb{P}\left\{W_p({\hat{\mathbb{P}}},\mathbb{P})-\left(\frac {C_2}{2-p}\right)^{1/p} \left(\frac\sigma{\sqrt{N}}\right) \geq \varepsilon\right\} \text{.}
    \end{equation}
    Combining (\ref{proof:lh}) and (\ref{proof:rh}), we have

    \begin{equation}
        \mathbb{P}\left\{W_p({\hat{\mathbb{P}}},\mathbb{P})-\left(\frac {C_2}{2-p} \right)^{1/p}\left(\frac\sigma{\sqrt{N}}\right) \geq \varepsilon\right\} \leq C_1\exp\big\{ -N^{1/2}3^{-1/2}\sigma^{-1}\varepsilon\big\}  \text{.}
    \end{equation}
    
    By setting $\varepsilon \rightarrow \varepsilon-(\frac {C_2}{2-p})^{1/p} (\frac\sigma{\sqrt{N}})$, we have
    \begin{equation}\label{ci}
       \mathbb{P}\{W_p({\hat{\mathbb{P}}},\mathbb{P})\geq \varepsilon\}\leq C_1\mathrm{exp}\left\{-\sqrt\frac{N}{3} \sigma^{-1}\left[\varepsilon-\left(\frac{C_2}{2-p}\right)^{1/p}\left(\frac\sigma{\sqrt{N}}\right)\right]\right\} 
    \end{equation}
    This concludes the proof. 
    \hfill \Halmos  
\end{proof}

\begin{proof}{Proof of Theorem \ref{thm:radius}}
    By setting the right-hand side of \eqref{mc} equal to $\beta$ and solving for $ \varepsilon$, we have
    \begin{equation}\label{radius1}
    \varepsilon = \sqrt{3} \frac{\sigma}{\sqrt{N}} \log \left( \frac{C_1}{\beta} \right) + \left(\frac{C_2}{2-p}\right)^{1/p}\left(\frac\sigma{\sqrt{N}}\right)\text{.}
    \end{equation}
    Considering the random variable $L(\bm{\zeta}(\bm{x}))=\sum_wa_wx_w\xi_{stw}$ in the inner problem $SP_{st}$, we substitute the $\sigma$ in (\ref{radius1}) with $\sqrt{(\bm{a}\cdot\bm{x})^\top{{\bm{\Sigma}}}_{st}(\bm{a}\cdot\bm{x})}$, which concludes the proof.
    \hfill \Halmos
\end{proof}

\begin{lemma}{(Tail bound)}\label{lemma:Hoeffding bound}

    Centered random variable $X$ is sub-exponential with parameters $(\nu^2, \alpha),\nu, \alpha > 0$ if:
    \begin{equation}
        \mathbb{E}e^{\lambda X} \leq e^{\frac{\lambda^2 \nu^2}{2}}, \quad \forall \lambda: |\lambda| < \frac{1}{\alpha}
    \end{equation}
    Then,
    \begin{equation}
        \mathbb{P}(|X - \mu| \geq t) \leq e^{-\frac{1}{2} \min\left\{\frac{t^2}{\nu^2}, \frac{t}{\alpha}\right\}}
    \end{equation}
\end{lemma}

\begin{proof}{Proof of Theorem \ref{thm:sample variance}}
    \quad  

    Let us treat the concentration inequality. Setting $k \in (0,1)$, we have 
    \begingroup
    \setlength{\abovedisplayskip}{5pt}
    \setlength{\belowdisplayskip}{5pt}
    \begin{align}
        &\mathbb{P}\{W_p(\hat {\mathbb P}_{st}(\bm{x}),{\mathbb P}_{st}(\bm{x}))\geq \varepsilon_{st}'(\bm{x})\}\\[2mm]
        =&\mathbb{P}\{W_p(\hat {\mathbb P}_{st}(\bm{x}),{\mathbb P}_{st}(\bm{x}))\geq \varepsilon_{st}'(\bm{x}),\sqrt{(\bm{a}\cdot\bm{x})^\top\widehat{{\bm{\Sigma}}}_{st}(\bm{a}\cdot\bm{x})}\geq k \sqrt{(\bm{a}\cdot\bm{x})^\top{{\bm{\Sigma}}}_{st}(\bm{a}\cdot\bm{x})}\} \notag \\[2mm]
        &+ \mathbb{P}\{W_p(\hat {\mathbb P}_{st}(\bm{x}),{\mathbb P}_{st}(\bm{x}))\geq \varepsilon_{st}'(\bm{x}),\sqrt{(\bm{a}\cdot\bm{x})^\top\widehat{{\bm{\Sigma}}}_{st}(\bm{a}\cdot\bm{x})}\leq k \sqrt{(\bm{a}\cdot\bm{x})^\top{{\bm{\Sigma}}}_{st}(\bm{a}\cdot\bm{x})}\} \\[2mm]
        \leq&\mathbb{P}\{W_p(\hat {\mathbb P}_{st}(\bm{x}),{\mathbb P}_{st}(\bm{x}))\geq \varepsilon_{st}'(\bm{x})\Big|\sqrt{(\bm{a}\cdot\bm{x})^\top\widehat{{\bm{\Sigma}}}_{st}(\bm{a}\cdot\bm{x})}\geq k \sqrt{(\bm{a}\cdot\bm{x})^\top{{\bm{\Sigma}}}_{st}(\bm{a}\cdot\bm{x})}\} \notag \\[2mm]
        &+ \mathbb{P}\{\sqrt{(\bm{a}\cdot\bm{x})^\top\widehat{{\bm{\Sigma}}}_{st}(\bm{a}\cdot\bm{x})}\leq k \sqrt{(\bm{a}\cdot\bm{x})^\top{{\bm{\Sigma}}}_{st}(\bm{a}\cdot\bm{x})}\} \label{proof:inequality}
    \end{align}

    The inequality follows from the conditional probability. Focusing on the first term in (\ref{proof:inequality}), we have 
    \begin{align}
        &\mathbb{P}\{W_p(\hat {\mathbb P}_{st}(\bm{x}),{\mathbb P}_{st}(\bm{x}))\}\geq \varepsilon_{st}'(\bm{x})\Big|\sqrt{(\bm{a}\cdot\bm{x})^\top\widehat{{\bm{\Sigma}}}_{st}(\bm{a}\cdot\bm{x})}\geq k \sqrt{(\bm{a}\cdot\bm{x})^\top{{\bm{\Sigma}}}_{st}(\bm{a}\cdot\bm{x})}\} \\[2mm]
        =&C_1\exp\left\{ \begin{array}{l}
        \begin{aligned}
        -k\sqrt{\frac{N}{3(\bm{a}\cdot\bm{x})^\top\widehat{{\bm{\Sigma}}}_{st}(\bm{a}\cdot\bm{x})}}
                &\Bigg[\sqrt{\frac{3(\bm{a}\cdot\bm{x})^\top\widehat{{\bm{\Sigma}}}_{st}(\bm{a}\cdot\bm{x})}{{N}}}  \log \left( \frac{C_1}{\beta} \right)    \\[2mm]
        + \left(\frac{C_2}{2-p}\right)^{1/p}&\sqrt{\frac{(\bm{a}\cdot\bm{x})^\top\widehat{{\bm{\Sigma}}}_{st}(\bm{a}\cdot\bm{x})}{N}}-\left(\frac{C_2}{2-p}\right)^{1/p}\sqrt{\frac{(\bm{a}\cdot\bm{x})^\top\widehat{{\bm{\Sigma}}}_{st}(\bm{a}\cdot\bm{x})}{k^2N}}\Bigg]
        \end{aligned}
        \end{array} \right\} \\[2mm]
        =&C_1\mathrm{exp}\left\{- \left[ k \log \left( \frac{C_1}{\beta} \right)  +3^{-1/2}(k-1)\left(\frac{C_2}{2-p}\right)^{1/p}\right]\right\} \label{proof:1}
    \end{align}
    \endgroup
    The first equality follows from the monotonicity of the right-hand side of (\ref{ci}) with respect to the variance, and we can relax it by substituting $\sqrt{(\bm{a}\cdot\bm{x})^\top{{\bm{\Sigma}}}_{st}(\bm{a}\cdot\bm{x})}$ with $\frac1k\sqrt{(\bm{a}\cdot\bm{x})^\top\widehat{{\bm{\Sigma}}}_{st}(\bm{a}\cdot\bm{x})}$.

    Given that $\bm{\xi}=(\xi_{w})_{w \in [W]}$ is bounded and sub-Gaussian, any projection of $\bm{\xi}$, namely $\bm{v}^\top\bm{\xi},\|\bm{v}\|=1$, is sub-Gaussian.
    Let $\bm{\xi}_1, \ldots, \bm{\xi}_N \in \mathbb{R}^{W}$ be i.i.d sub-Gaussian random vectors with mean $\bm{\mu} \in \mathbb{R}^{W}$ and covariance matrix ${\bm{\Sigma}}$. The sample covariance matrix is defined as
    \[
    \hat{{\bm{\Sigma}}} = \frac{1}{N-1} \sum_{i=1}^N (\bm{\xi}_i - \bar{\bm{\xi}})(\bm{\xi}_i - \bar{\bm{\xi}})^\top, \quad \text{where } \bar{\bm{\xi}} = \frac{1}{N} \sum_{i=1}^N \bm{\xi}_i.
    \]
    For any fixed vector $\bm{v} \in \mathbb{R}^{W}$, define the scalar-valued random variable $Y_i = \bm{v}^\top \bm{\xi}_i$, which has mean $\mu_v = \bm{v}^\top \bm{\mu}$ and variance $\sigma_v^2 = \bm{v}^\top {\bm{\Sigma}} \bm{v}$. The corresponding sample mean is denoted by $\bar{Y} = \frac{1}{N} \sum_{i=1}^N Y_i$.
    
    We define $\bm{v}=\frac{\bm{a}\cdot\bm{x}}{\|\bm{a}\cdot\bm{x}\|}$ and omit the index $s,t$ for simplicity, then the second term in (\ref{proof:inequality}) can be reformulated as, 
    \begingroup
    \setlength{\abovedisplayskip}{2pt}
    \setlength{\belowdisplayskip}{1pt}
    \begin{align*}
        &\mathbb P\{ \sqrt{(\bm{a}\cdot\bm{x})^\top\widehat{{\bm{\Sigma}}}_{st}(\bm{a}\cdot\bm{x})}\leq k \sqrt{(\bm{a}\cdot\bm{x})^\top{{\bm{\Sigma}}}_{st}(\bm{a}\cdot\bm{x})} \} \\
        =&\mathbb P\{ (\bm{a}\cdot\bm{x})^\top\widehat{{\bm{\Sigma}}}_{st}(\bm{a}\cdot\bm{x})\leq k^2 (\bm{a}\cdot\bm{x})^\top{{\bm{\Sigma}}}_{st}(\bm{a}\cdot\bm{x}) \} \\[2mm]
        =&\mathbb P\{ \hat{\sigma}^2_v \leq k^2\sigma^2_v \}
    \end{align*}
    \endgroup

    Due to $\frac{N-1}{N}\hat \sigma_v^2=\frac{1}{N}\sum_{i=1}^{N}(Y_i-\mu_v)^2-(\mu_v-\bar Y)^2$, we have 
    \begingroup
    \setlength{\abovedisplayskip}{2pt}
    \setlength{\belowdisplayskip}{1pt}
    \begin{align}
        &\mathbb P\{ \hat{\sigma}^2_v \leq k^2\sigma^2_v \} \\
        =&\mathbb{P}\{\frac{1}{N}\sum_{i=1}^{N}(Y_i-\mu_v)^2-(\mu_v-\bar Y)^2\leq \frac{N-1}{N}k^2\sigma_v^2 \} \\
        \leq&\mathbb{P}\{\frac{1}{N}\sum_{i=1}^{N}(Y_i-\mu_v)^2-\sigma_v^2-(\mu_v-\bar Y)^2\leq (k^2-1)\sigma_v^2\} \\
        \leq&\mathbb{P}\{\frac{1}{N}\sum_{i=1}^{N}(Y_i-\mu_v)^2-\sigma_v^2\leq\frac{k^2-1}{2}\sigma_v^2\}+\mathbb{P}\{(\mu_v-\bar Y)^2\geq\frac{1-k^2}{2}\sigma_v^2\} \label{proof:sub}
    \end{align}
    \endgroup

    Given that $Y_i$ is sub-Gaussian, we denote its sub-Gaussian parameter by $\zeta$, which is bounded due to the boundness of $\bm{\xi}$. Then $\frac{1}{N}\sum_{i=1}^{N}(Y_i-\mu_v)^2$ and $(\bar Y-\mu_v)^2$ are sub-exponential with parameters $(\frac{256\zeta^4}{N}, \frac{16\zeta^2}{N})$ and $(\frac{256\zeta^4}{N^2}, \frac{16\zeta^2}{N})$, respectively (Lemma 1.12, \citealp{rigollet2023high}). Based on Lemma \ref{lemma:Hoeffding bound}, we have
    \begingroup
    \setlength{\abovedisplayskip}{1pt}
    \setlength{\belowdisplayskip}{1pt}
    \begin{align*}
        \mathbb{P}\{\frac{1}{N}\sum_{i=1}^{N}(Y_i-\mu_v)^2-\sigma_v^2\leq\frac{k^2-1}{2}\sigma_v^2\} \leq \exp\Big\{-N\min\{c_1(1-k^2)^2\sigma_v^4, c_2(1-k^2)\sigma_v^2\}\Big\} \\
        \mathbb{P}\{(\mu_v-\bar Y)^2\geq\frac{1-k^2}{2}\sigma_v^2\} \leq \exp\Big\{-N\min\{c_3N(1-k^2)^2\sigma_v^4, c_4(1-k^2)\sigma_v^2\}\Big\}
    \end{align*}
    \endgroup
    where $c_1, c_2, c_3, c_4$ are positive constants.
    
    For the second term in (\ref{proof:inequality}), we have
    \begingroup
    \setlength{\abovedisplayskip}{1pt}
    \setlength{\belowdisplayskip}{1pt} 
    \begin{align}
        &\mathbb{P}\{\sqrt{(\bm{a}\cdot\bm{x})^\top\widehat{{\bm{\Sigma}}}_{st}(\bm{a}\cdot\bm{x})}\leq k \sqrt{(\bm{a}\cdot\bm{x})^\top{{\bm{\Sigma}}}_{st}(\bm{a}\cdot\bm{x})}\} \\
        \leq &\exp\Big\{-N\min\{c_1(1-k^2)^2((\bm{a}\cdot\bm{x})^\top{{\bm{\Sigma}}}_{st}(\bm{a}\cdot\bm{x}))^2, c_2(1-k^2)(\bm{a}\cdot\bm{x})^\top{{\bm{\Sigma}}}_{st}(\bm{a}\cdot\bm{x})\}\Big\} \\
        &+ \exp\Big\{-N\min\{c_3N(1-k^2)^2((\bm{a}\cdot\bm{x})^\top{{\bm{\Sigma}}}_{st}(\bm{a}\cdot\bm{x}))^2, c_4(1-k^2)(\bm{a}\cdot\bm{x})^\top{{\bm{\Sigma}}}_{st}(\bm{a}\cdot\bm{x})\}\Big\} \notag\\
        \leq &2\exp\Big\{-N\min\{C_3(1-k^2)^2((\bm{a}\cdot\bm{x})^\top{{\bm{\Sigma}}}_{st}(\bm{a}\cdot\bm{x}))^2, C_4(1-k^2)(\bm{a}\cdot\bm{x})^\top{{\bm{\Sigma}}}_{st}(\bm{a}\cdot\bm{x})\}\Big\} \label{proof:2}
    \end{align}
    \endgroup
    Where $C_3, C_4$ are positive constants.

    Combining (\ref{proof:1}) and (\ref{proof:2}), we have
    \begingroup
    \setlength{\abovedisplayskip}{1pt}
    \setlength{\belowdisplayskip}{1pt}
    \begin{align*}
        &\mathbb{P}\{W_p(\hat {\mathbb P}_{st}(\bm{x}),{\mathbb P}_{st}(\bm{x}))\geq \varepsilon_{st}'(\bm{x})\} \\[2mm]
        \leq &C_1\mathrm{exp}\left\{- \left[ k \log \left( \frac{C_1}{\beta} \right)  +3^{-1/2}(k-1)\left(\frac{C_2}{2-p}\right)^{1/p}\right]\right\} \\[2mm]
        & + 2\exp\Big\{-N\min\{C_3(1-k^2)^2((\bm{a}\cdot\bm{x})^\top{{\bm{\Sigma}}}_{st}(\bm{a}\cdot\bm{x}))^2, C_4(1-k^2)(\bm{a}\cdot\bm{x})^\top{{\bm{\Sigma}}}_{st}(\bm{a}\cdot\bm{x})\}\Big\}
    \end{align*}
    \endgroup

    Finally, we choose a $k$ that minimizes the violation probability and define 
    \begingroup
    \setlength{\abovedisplayskip}{1pt}
    \setlength{\belowdisplayskip}{1pt}
    \begin{align}
        \beta'=&\inf_{k\in(0,1)}\left\{C_1\mathrm{exp}\Big[- k \log \Big( \frac{C_1}{\beta} \Big) +\frac{1-k}{\sqrt3}\Big(\frac{C_2}{2-p}\Big)^{1/p}\Big] \right.\\
        &+2\exp\big[-N\min\{C_3(1-k^2)^2((\bm{a}\cdot\bm{x})^\top{{\bm{\Sigma}}}_{st}(\bm{a}\cdot\bm{x}))^2, C_4(1-k^2)(\bm{a}\cdot\bm{x})^\top{{\bm{\Sigma}}}_{st}(\bm{a}\cdot\bm{x})\}\big]\Big\}\text{.}
    \end{align}
    \endgroup
    This concludes the proof.

    \hfill \Halmos
\end{proof}

\section{Proof in Section \ref{sec:mr}}
\begin{proof}{Proof of Theorem \ref{thm:mr}}

    Let $\tau^1_{stg}$, $\tau^2_{stg}$ and $\gamma_{st}$ be the dual variables associated with constraints (\ref{subcos:obj non-negative1}), (\ref{subcos:obj non-negative2}) and (\ref{subcos:error balance}); $\mu^1_{stg}$ and $\mu^2_{stg}$ the dual variables associated with constraint (\ref{subcos:adjustment limit}). Then we formulate the dual of problem $h_{st}(\bm{\zeta}_{st}(\bm{x}))$ as

    \begingroup
    \setlength{\abovedisplayskip}{0pt}
    \setlength{\belowdisplayskip}{0pt}
    \begin{align}
        \max \quad -&\left[\sum_w\zeta_{stw}(x_w)-\bar \zeta_{stw}(x_w)\right]\gamma_{st}  - \sum_g{\left({\bar r}_{stg}\mu^1_{stg}+{\underline r}_{stg}\mu^2_{stg}\right)}  \\
        \text{s.t.}\quad & 1-\tau^1_{stg}-\tau^2_{stg}=0  & \forall g \label{cos:power balance0}\\
        & \gamma_{st}+DA_g\tau^1_{stg}-UA_g\tau^2_{stg}-\mu^1_{stg}+\mu^2_{stg}=0 &\forall g \\
        & WC +\gamma_{st} \geq 0 &\quad\\
        & LS - \gamma_{st} \geq 0  &\quad\\
        & \tau^1_{stg}, \tau^2_{stg} \geq 0 &\forall g \\
        & \mu^1_{stg}, \mu^2_{stg} \geq 0  &\forall g \label{cos:power balance1}
    \end{align}
    \endgroup
    We focus on the $\sum_w\zeta_{stw}(x_w)$ and $\gamma_{st}$, and denote the above reformulation as $\max_{\gamma_{st} \in \mathcal{D}_{st}}-\gamma_{st}\sum_w\zeta_{stw}(x_w)+\varrho$, where $\mathcal{D}_{st}$ represents the dual feasbile region defined by constraints (\ref{cos:power balance0})-(\ref{cos:power balance1}).

    By duality theory, the worst-case expectation problem admits the strong dual robust optimization problem
    \begin{align}
        \inf_{\phi_{st}>0}\  &\phi_{st}\varepsilon'_{st}(\bm{x})+\frac{1}{N}\sum_{i=1}^{N}s_i \\
        \text{s.t.} \sup_{\sum_w\zeta_{stw}(x_w) \in \mathbb{R}} h_{st}(\bm{\zeta}_{st}(\bm{x}))-&\phi_{st}|\sum_w\zeta_{stw}(x_w) - \sum_w\hat\zeta_{stwi}(x_w)| \leq s_i \quad \forall i
    \end{align}

    Substituting the $h_{st}(\bm{\zeta}_{st}(\bm{x}))$ with its dual reformulation, we obtain
    \begin{equation*}
        \sup_{\sum_w\zeta_{stw}(x_w) \in \mathbb{R}} \max_{\gamma_{st} \in \mathcal{D}_{st}}-\gamma_{st}\sum_w\zeta_{stw}(x_w)+\varrho-\phi_{st}|\sum_w\zeta_{stw}(x_w) - \sum_w\hat\zeta_{stwi}(x_w)| \leq s_i \quad \forall i
    \end{equation*}
    Invoking the definition of the dual norm and interchanging the order of the supremum and the maximum, we have
    \begingroup
    \setlength{\abovedisplayskip}{1pt}
    \setlength{\belowdisplayskip}{1pt}
    \begin{align*}
        &\max_{\gamma_{st} \in \mathcal{D}_{st}} \sup_{\sum_w\zeta_{stw}(x_w) \in \mathbb{R}} -\gamma_{st}\sum_w\zeta_{stw}(x_w)+\varrho-\phi_{st}|\sum_w\zeta_{stw}(x_w) - \sum_w\hat\zeta_{stwi}(x_w)| \leq s_i \quad \forall i \\
        =&\max_{\gamma_{st} \in \mathcal{D}_{st}} \sup_{\sum_w\zeta_{stw}(x_w) \in \mathbb{R}}\inf_{|\iota_i| \leq \phi_{st}} -\gamma_{st}\sum_w\zeta_{stw}(x_w)+\varrho-\iota_i(\sum_w\zeta_{stw}(x_w) - \sum_w\hat\zeta_{stwi}(x_w)) \leq s_i \quad \forall i \\
    \end{align*}
    \endgroup

    Next, we apply the classical minimax theorem (Proposition 5.5.4, \citealp{bertsekas2009convex}) to interchange the order of the innermost supremum and the infimum, which is valid due to the compactness of the domain of $\iota_i$. We then proceed to interchange the supremum and the outer maximization. This yields
    \begingroup
    \setlength{\abovedisplayskip}{1pt}
    \setlength{\belowdisplayskip}{1pt}
    \begin{align*}
        &\left\{\begin{aligned}
            &\inf_{\phi_{st}>0}\ \phi_{st}\varepsilon'_{st}(\bm{x})+\frac{1}{N}\sum_{i=1}^{N}s_i \\
            \text{s.t.}\quad &\inf_{|\iota_i| \leq \phi_{st}}\max_{\gamma_{st} \in \mathcal{D}_{st}} \sup_{\sum_w\zeta_{stw}(x_w) \in \mathbb{R}} -\gamma_{st}\sum_w\zeta_{stw}(x_w)+\varrho-\iota_i(\sum_w\zeta_{stw}(x_w) - \sum_w\hat\zeta_{stwi}(x_w)) \leq s_i \quad \forall i
        \end{aligned} \right. \\
        =&\left\{\begin{aligned}
            &\inf_{\phi_{st}>0}\ \phi_{st}\varepsilon'_{st}(\bm{x})+\frac{1}{N}\sum_{i=1}^{N}s_i \\
            \text{s.t.}\quad &\max_{\gamma_{st} \in \mathcal{D}_{st}}\sup_{\sum_w\zeta_{stw}(x_w) \in \mathbb{R}} -\gamma_{st}\sum_w\zeta_{stw}(x_w)+\varrho-\iota_i(\sum_w\zeta_{stw}(x_w) - \sum_w\hat\zeta_{stwi}(x_w)) \leq s_i \quad \forall i\\
            &|\iota_i| \leq \phi_{st} \quad \forall i
        \end{aligned} \right. \\
        =&\inf_{\phi_{st}>0}\ \phi_{st}\varepsilon'_{st}(\bm{x})+\frac{1}{N}\sum_{i=1}^{N}\max_{\gamma_{st} \in \mathcal{D}_{st}}-\gamma_{st}\sum_w\hat\zeta_{stwi}(x_w)+\varrho  + \chi_{|\gamma_{st}| \leq \phi_{st}}(\gamma_{st})
    \end{align*}
    \endgroup
    where $\chi_{|\gamma_{st}| \leq \phi_{st}}(\gamma_{st})$ is defined as $\chi_{|\gamma_{st}| \leq \phi_{st}}(\gamma_{st})=0$ if $|\gamma_{st}| \leq \phi_{st}$ and as $=\infty$ otherwise.

    Defining $\phi_{st}^\star = \sup\{\gamma_{st}: \gamma_{st} \in \mathcal{D}_{st}\}$, we study two cases with respect to the value of $\phi_{st}$: 
    \begin{enumerate}
        \item $\phi_{st} < \phi_{st}^\star$: In this case, the supremum over $\gamma_{st}$ would be unbounded.
        \item $\phi_{st} > \phi_{st}^\star$: In this case, any value of $\phi_{st}$ exceeding the optimal threshold would lead to an unnecessarily high objective value, due to the penalization by $\varepsilon’_{st}(\bm{x})$.
    \end{enumerate}
    Combined with the above two cases, the unique optimal solution is $\phi_{st}^\star$. We thus obtain
    \begingroup
    \setlength{\abovedisplayskip}{1pt}
    \setlength{\belowdisplayskip}{1pt}
    \begin{align*}
        &\left\{\begin{aligned}
            &\inf_{\phi_{st}>0}\ \phi_{st}\varepsilon'_{st}(\bm{x})+\frac{1}{N}\sum_{i=1}^{N}\max_{\gamma_{st} \in \mathcal{D}_{st}}-\gamma_{st}(\sum_w\hat\zeta_{stwi}(x_w))+\varrho \\
            \text{s.t.}\quad &|\gamma_{st}| \leq \phi_{st} \quad \forall \gamma_{st} \in \mathcal{D}_{st}
        \end{aligned} \right.
    \end{align*}
    \endgroup

    For intuitive representation, we transform the dual problem back into the primal problem. This yields
    \begin{equation*}
        \max_{{\mathbb P_{st}(\bm{x})} \in \mathcal{P}_{st}(\bm{x})}\mathbb{E}_{\mathbb P_{st}(\bm{x})}[h_{st}(\bm{\zeta}_{st}(\bm{x}))] = \mathbb{E}_{\hat {\mathbb P} _{st}(\bm{x})}[h_{st}(\bm{\zeta}_{st}(\bm{x}))] + \varepsilon_{st}'(\bm{x})*\phi_{st}
    \end{equation*}
    where $\hat {\mathbb P}_{st}(x)=\frac{1}{N}\sum_{i=1}^N\delta_{\sum_w x_w\hat\xi_{stwi}}$, $\phi_{st}:=\max_{\gamma_{st} \in \mathcal{D}_{st}}|\gamma_{st}|$, and $\mathcal{D}_{st}$ is the dual feasible region.

    Since $\gamma_{st}$ is bounded by the linear constraints of the dual feasible region, the value of $\phi_{st}$ can be determined as follows:
    \begingroup
    \setlength{\abovedisplayskip}{1pt}
    \setlength{\belowdisplayskip}{1pt}
    \begin{align*}
        &\max \quad |\gamma_{st}|  \\
        \text{s.t.} \quad &constraint (\ref{cos:power balance0}) - constraint (\ref{cos:power balance1})\text{,}
    \end{align*}
    \endgroup
    which can be further reformulated as a linear programming using epigraph reformulation. This concludes the proof.
    
    \hfill \Halmos
\end{proof}

\begin{proof}{Proof of Theorem \ref{thm:pg}}
    The claim follows directly from Theorem \ref{thm:sample variance}, which ensures, via the definition of $\varepsilon_{st}'(\bm{x})$, that $\mathbb{P}\left\{\mathbb{P}_{st}(\bm{x}) \in \mathcal P_{st}(\bm{x})  \right\}\geq 1-\beta'$. This implies that 
    $\mathbb{E}_{\mathbb P_{st}(\bm{x})}[h_{st}(\bm{\zeta}_{st}(\bm{x}))]\leq \sup_{\mathbb P_{st}(\bm{x}) \in \mathcal{P}_{st}(\bm{x})}\mathbb{E}_{\mathbb P_{st}(\bm{x})}[h_{st}(\bm{\zeta}_{st}(\bm{x}))]={J}_{st}$ with probability $1-\beta'$.
    
    \hfill \Halmos
\end{proof}

\begin{proof}{Proof of Theorem \ref{thm:tlc}}
    Based on the Proposition 3.4 in \citet{shapiro2001duality}, we can exchange the supremum and minimum operators, that is 
    \begin{equation}
     \min_{\eta_{stl} \in \mathbb{R}}\sup_{{\mathbb P_{stl}(\bm{x})} \in \mathcal{P}_{stl}(\bm{x})}\mathbb{E}_{\mathbb P_{stl}(\bm{x})}\left\{\eta_{stl}+\frac1{\epsilon}[\max_{k\in\{1,2\}}l_{stlk}(\bm{\zeta}_{st}(\bm{x}))+l^0_{stlk}-\eta_{stl}]_+\right\}\leq0 \quad\forall s,t,l\text{,}
    \end{equation}
    Then, (\ref{cvar3}) follows from Corollary 5.1 in \citet{mohajerin2018data} with ambiguity set (\ref{ddas}) and the random variable $L_2(\bm{\zeta}_{st}(\bm{x}))=\sum_{w}\pi_{wl}x_w\xi_{stw}$.
    \hfill \Halmos
\end{proof}

\section{L-shaped Algorithm}\label{sec:la}

This is simply a standard benders decomposition routine.

Problem $RP$ can be decomposed into a master problem, which consists of decisions on capacity planning, generation, reserve and the regularization term, and several subproblems, which represent the empirical expectation. The master problem is formulated as
\begingroup
\setlength{\abovedisplayskip}{2pt}
\setlength{\belowdisplayskip}{0pt}
\begin{align}
\begin{split}
    MRP:\quad \min \quad &\sum_w{c_wx_w} +\sum_s\Big\{\sum_{t}\sum_g\left[F_g(P_{stg}) + UR_{g}\bar r_{stg} + DR_{g}\underline r_{stg}\right] +  \\[-2mm]
    &\hspace{3cm}\sum_{t}\left[\phi_{st}\kappa_{st}\vartheta_{st}\right]\Big\}\Delta_s 
\end{split} \\
\text{s.t.} \quad  &(\ref{cos:power balance})-(\ref{cos:ramp down limit}), (\ref{cos:non-negative})-(\ref{cos:integer}), (\ref{reformulation:radius}), (\ref{cvar3})
\end{align}
\endgroup
Given a feasible solution to the master problem $MRP$, the subproblem is a sample average approximation for the single-period scheduling problem as follows:
\begingroup
\setlength{\abovedisplayskip}{2pt}
\setlength{\belowdisplayskip}{1pt}
\begin{align}
\begin{split}
    (SRP_{st}):\quad \min \quad &\frac1N\sum_i\Big(\sum_gz_{stgi}+WCw_{sti} + LSl_{sti}\Big) 
\end{split} \\
\text{s.t.}\quad
&DA_g\alpha_{stgi} \leq z_{stgi} \quad &\forall g,i \label{SRP:cos:power balance}\\
&-UA_g\alpha_{stgi} \leq z_{stgi} \quad &\forall g,i \label{SRP:cos:generation limit}\\
&\sum_g\alpha_{stgi} + w_{sti}-l_{sti} = \sum_wx_w(\hat\xi_{stwi}-\bar \xi_{stwi}) \quad &\forall i \label{SRP:cos:ramp up limit}\\[-2mm]
&-\bar r_{stg} \leq \alpha_{stgi} \leq \underline r_{stg} \quad & \forall g,i\label{SRP:cos:ramp down limit}
\end{align}
\endgroup

For each $s$ and $t$, let $\bm{\tau}^1_{stg}, \bm{\tau}^2_{stg}, \bm{\gamma}_{st} \in \mathbb{R}^N$ be the dual variables associated with constraints (\ref{SRP:cos:power balance}), (\ref{SRP:cos:generation limit}), (\ref{SRP:cos:ramp up limit}), respectively; $\bm{\mu}^1_{stg}, \bm{\mu}^2_{stg} \in \mathbb{R}^N$ be the dual variables associated with constraints (\ref{SRP:cos:ramp down limit}); $\bm{\Lambda}_{st}$ be the set of extreme points of the dual problem $SRP_{st}$. Given that the subproblem $SRP_{st}$ is manifestly feasible and has optimal solutions, for each scenario $s$ and period $t$, we only need to add optimality cuts into the master problem and derive a relaxation of the problem $MRP$ as
\begingroup
\setlength{\abovedisplayskip}{2pt}
\setlength{\belowdisplayskip}{0pt}
\begin{align}
\begin{split}
    MRP^{\prime}:\quad \min \quad &\sum_w{c_wx_w} +\sum_s\Big\{\sum_{t}\sum_g\left[F_g(P_{stg}) + UR_{g}\bar r_{stg} + DR_{g}\underline r_{stg}\right] +  \\[-1mm]
    &\hspace{3cm}
    \sum_{t}\left[\phi_{st}\kappa_{st}\vartheta_{st} + q_{st}\right]\Big\}\Delta_s 
\end{split} \notag\\[-3mm]
\text{s.t.} \quad
\begin{split}
    &-\frac{1}{N}\sum_i\Big[\sum_wx_w(\hat\xi_{stwi}-\bar \xi_{stwi})\gamma_{sti}+\sum_g\bar r_{stg}\mu^1_{stgi} + \underline r_{stg}\mu^2_{stgi} \Big] \leq q_{st} , \\ 
    &\hspace{5cm}\forall (\cdot,\cdot,\bm{\gamma}_{st}, \bm{\mu}^1_{stg},\bm{\mu}^2_{stg}) \in \bm{\Lambda}_{st}, \forall s,t,
\end{split} \label{optimality cuts}\\
&(\ref{cos:power balance})-(\ref{cos:ramp down limit}), (\ref{cos:non-negative})-(\ref{cos:integer}), (\ref{reformulation:radius}), (\ref{cvar3}) \notag
\end{align}
\endgroup

Since the set of extreme points $\bm{\Lambda}_{st}$ grows exponentially, the optimality cuts cannot be pre-enumerated. Instead, they are generated dynamically during the iterative process of solving the subproblem $SRP_{st}$ and the relaxed master problem $MRP^{\prime}$. The steps of the L-shaped algorithm are summarized in Algorithm \ref{L-shaped algorithm}.
\vspace{-0.5cm}
\begin{algorithm}[H] 
\caption{L-shaped Algorithm}\label{L-shaped algorithm}
\begin{algorithmic}[1]
\State $lb \leftarrow 0, ub \leftarrow \infty$;
\While{$(ub - lb)/lb > \nu$}
    \State Solve the master problem $MRP^\prime$ and obtain the optimal solution $(x, \bar{P}_{stg}, \bar r_{stg}, \underline r_{stg} , q_{st})$;
    \State $lb \leftarrow \max \{lb,OV($$MRP^\prime$$)\}$, where $OV(\cdot)$ denotes the objective value;
    \For{$s\in [S],t \in [T]$}
    \State Construct the subproblem $SRP_{st}$ according to $(x, \bar{P}_{stg}, \bar r_{stg}, \underline r_{stg})$;
    \State Solve the subproblem $SRP_{st}$ and obtain the dual optimal solution $(\tau^1_{st}, \tau^2_{st}, \bm{\gamma}_{st}, \bm{\mu}^1_{stg}, \bm{\mu}^2_{stg})$;
    \State Add the optimality cut (\ref{optimality cuts}) to the master problem $MRP^\prime$;
    \EndFor
    \State $ub \leftarrow \min \{ ub, OV($$MRP^\prime$$)$ $+$ $\sum_s\sum_tOV(SRP_{st})-q_{st}\}$;
\EndWhile
\end{algorithmic}
\end{algorithm}

\section{Additional Results for Section \ref{sec:ne}}\label{sec:nr}

For each training set, we solve the DDRO and NDRO models across a sequence of radius parameters $\kappa_{st}/\kappa'_{st}$, where each parameter corresponds to a specific total installed capacity. As shown in Figure~\ref{fig:obj}, the minimum total cost achieved by DDRO is lower than that of the benchmark methods, illustrating its potential advantages. 
However, since investment and generation costs dominate the total cost, the benefit from improved risk management may be obscured. This, in turn, complicates the tuning of the radius parameter ($\kappa_{st}/\kappa'_{st}$) and may lead to numerical challenges in identifying optimal values through cross-validation.
Meanwhile, Figure~\ref{fig:rmc} shows that the risk management cost under DDRO and NDRO is consistently lower than that of EO, clearly highlighting DRO’s superior capability in managing uncertainty. Motivated by this observation, we further focus on the risk management performance of different models.

\begin{figure}[htbp]
    \FIGURE
    {
    \subcaptionbox{Total Cost \label{fig:obj}}
    {\includegraphics[width=0.4\textwidth]{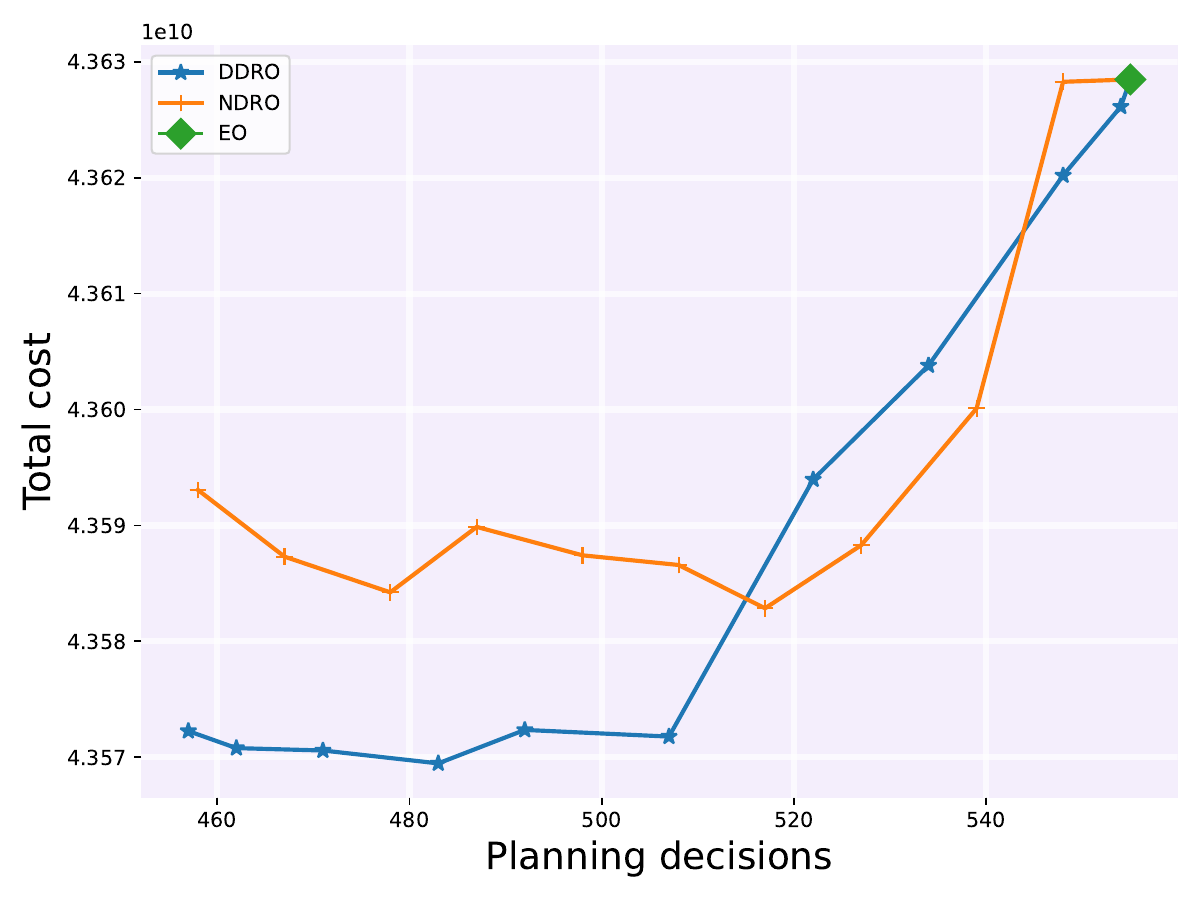}}
    \subcaptionbox{Risk Management Cost \label{fig:rmc}}
    {\includegraphics[width=0.4\textwidth]{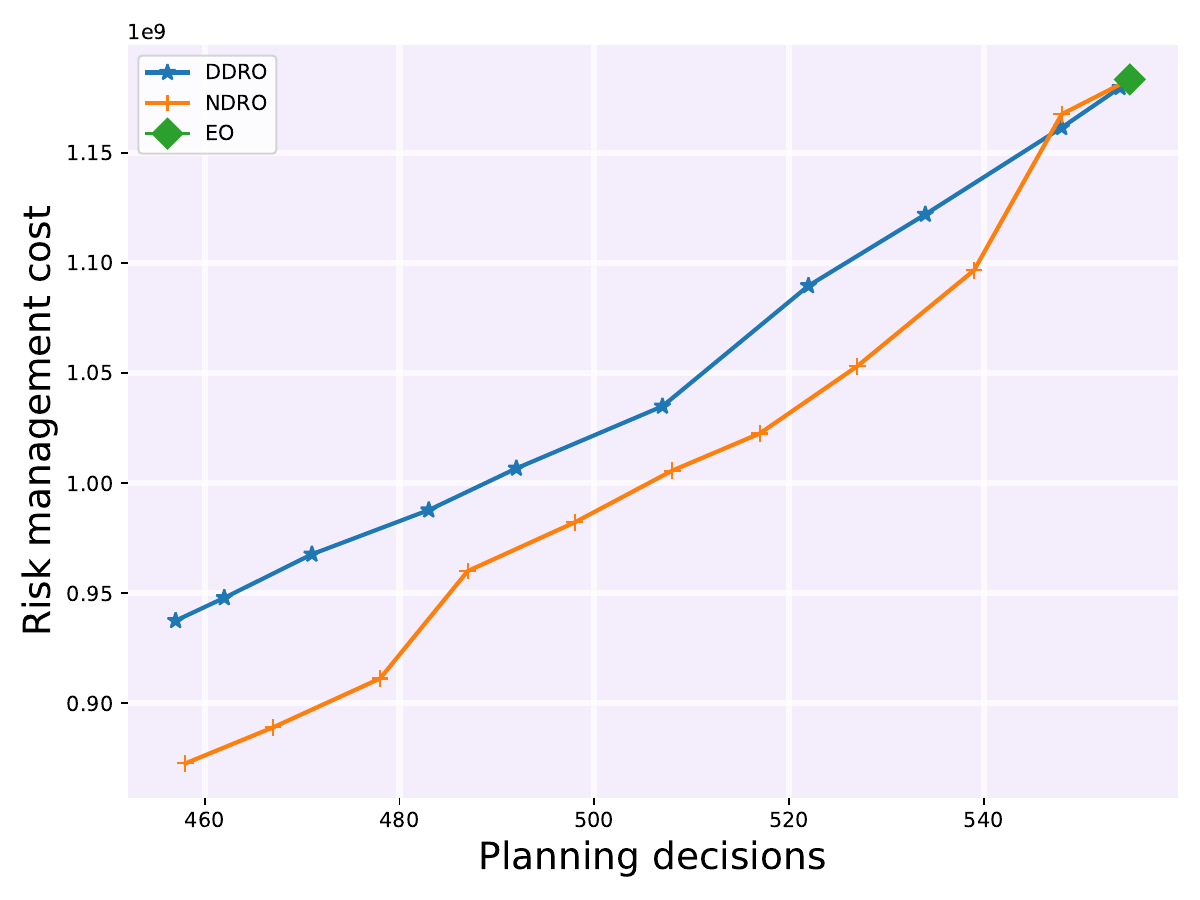}}
    }
    {
    Out-of-sample Metrics With Respect to Model Parameters}
    {}
\end{figure}

\begin{figure}[htbp]
    \FIGURE
    {
    \subcaptionbox{$X=400$}
    {\includegraphics[width=0.3\textwidth]{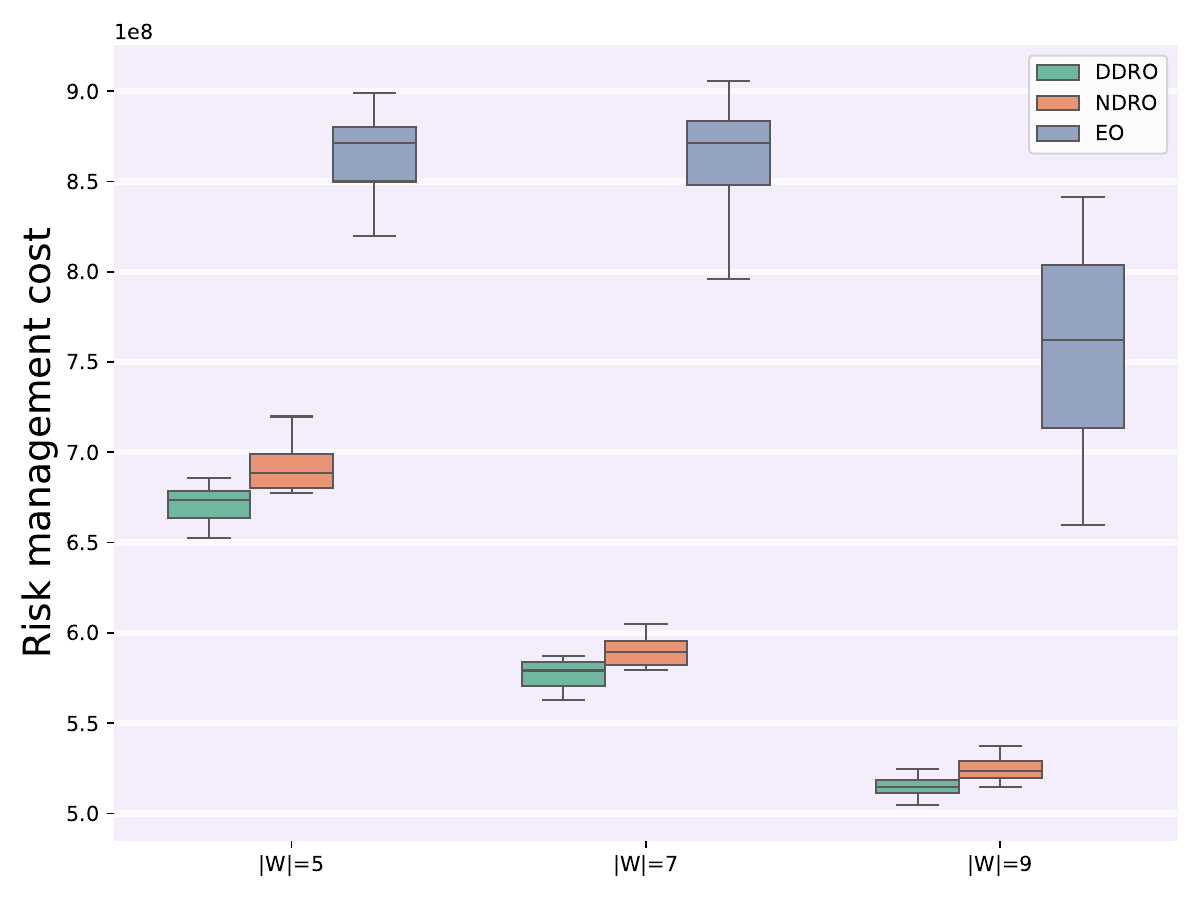}}
    \subcaptionbox{$X=600$}
    {\includegraphics[width=0.3\textwidth]{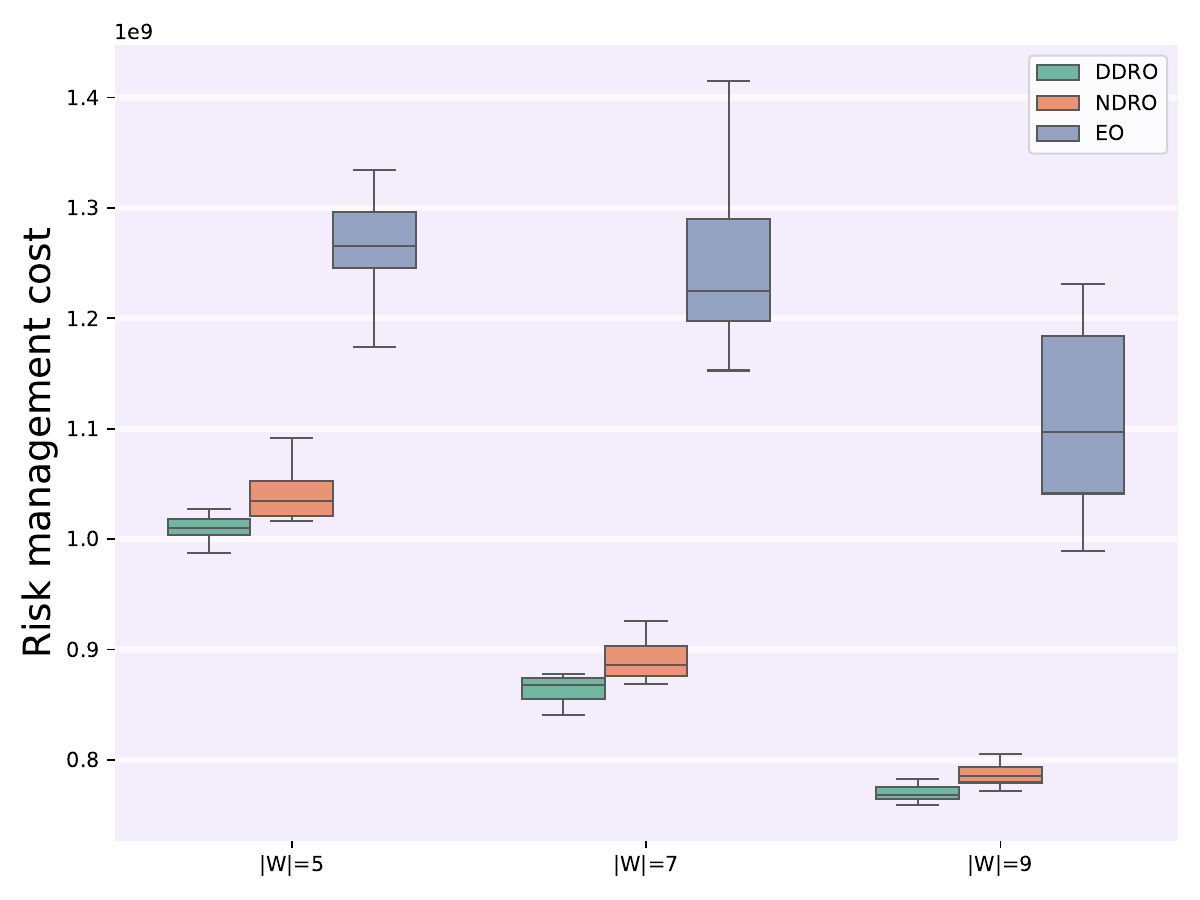}}
    \subcaptionbox{$X=900$}
    {\includegraphics[width=0.3\textwidth]{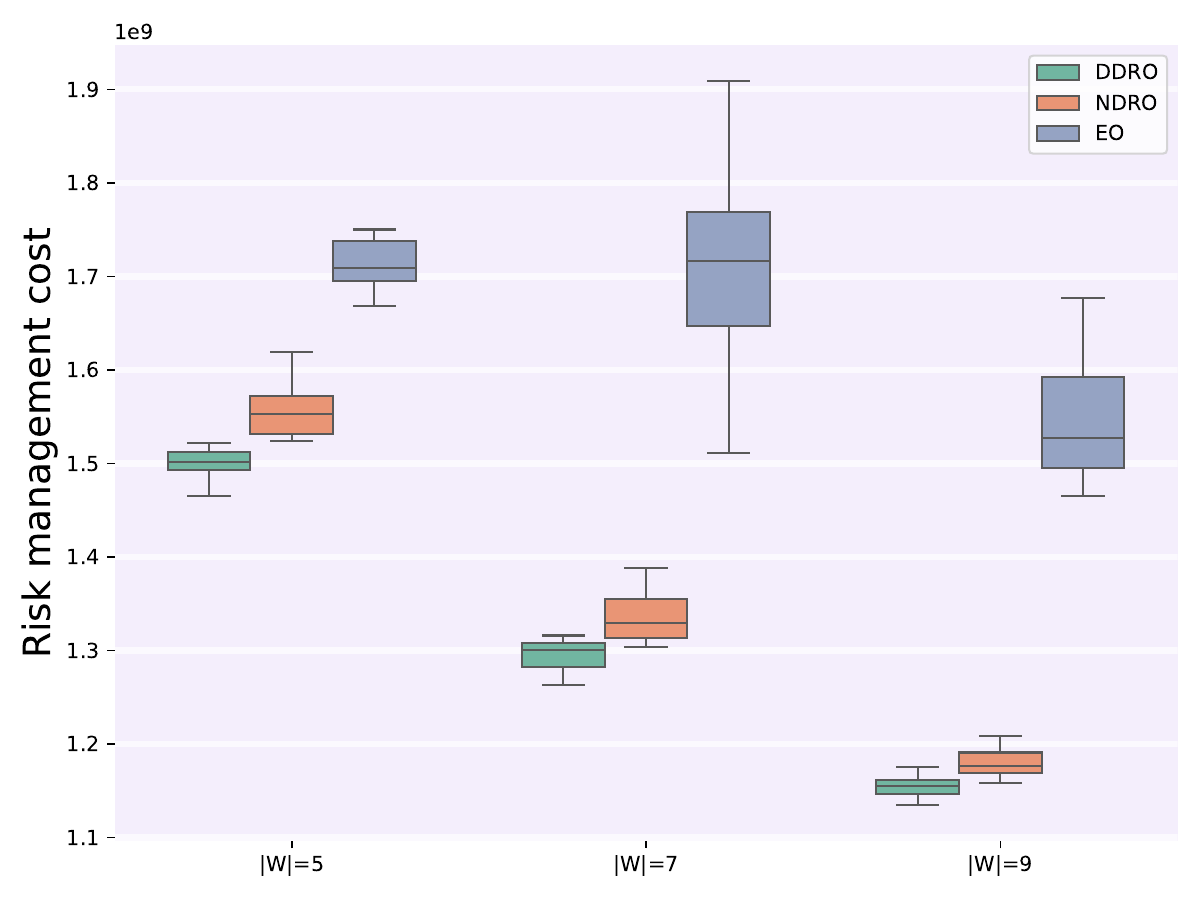}}
    }
    {
    Box Plots of Out-of-sample Risk Management Cost Using Traditional Parameters \label{fig:1}}
    {}
\end{figure}
    
\begin{figure}[htbp]
    \FIGURE
    {
    \subcaptionbox{$X=400$}
    {\includegraphics[width=0.3\textwidth]{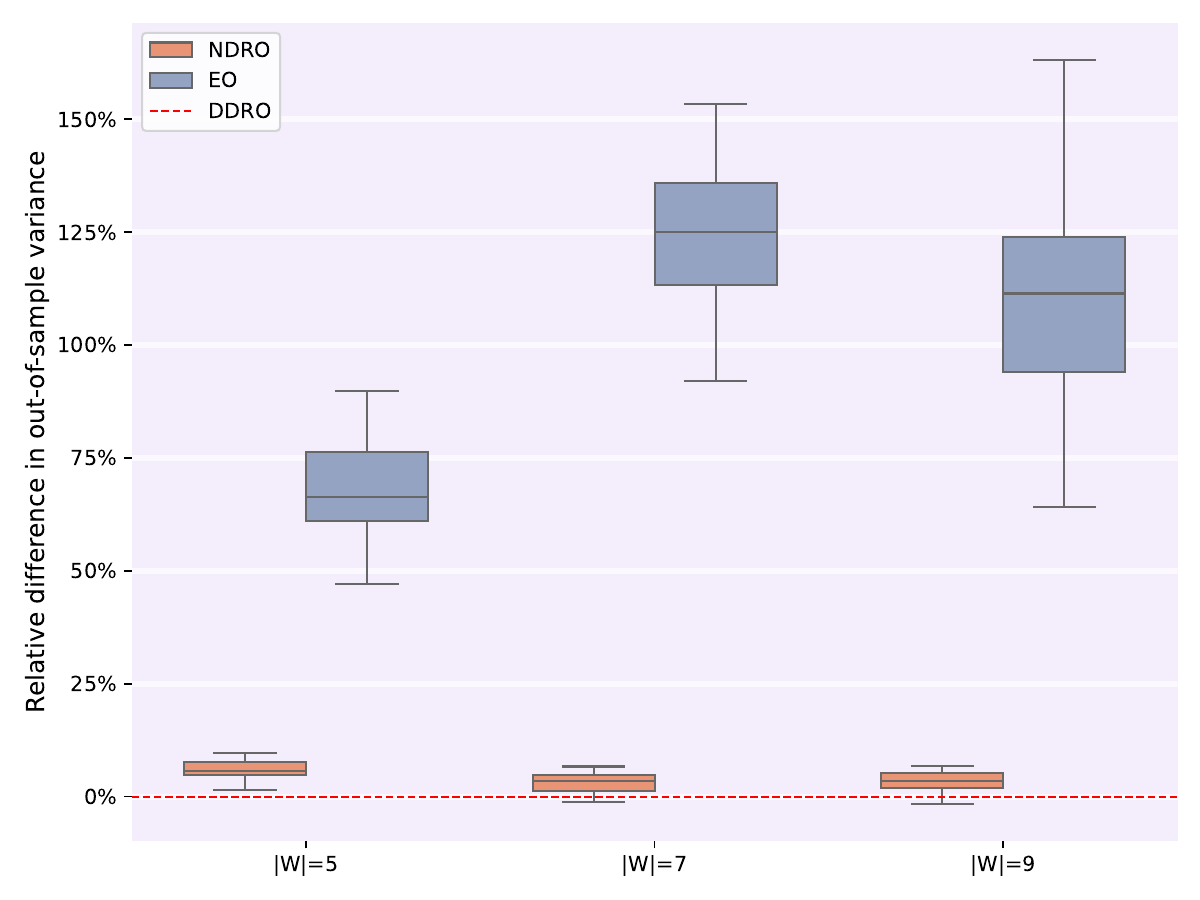}}
    \subcaptionbox{$X=600$}
    {\includegraphics[width=0.3\textwidth]{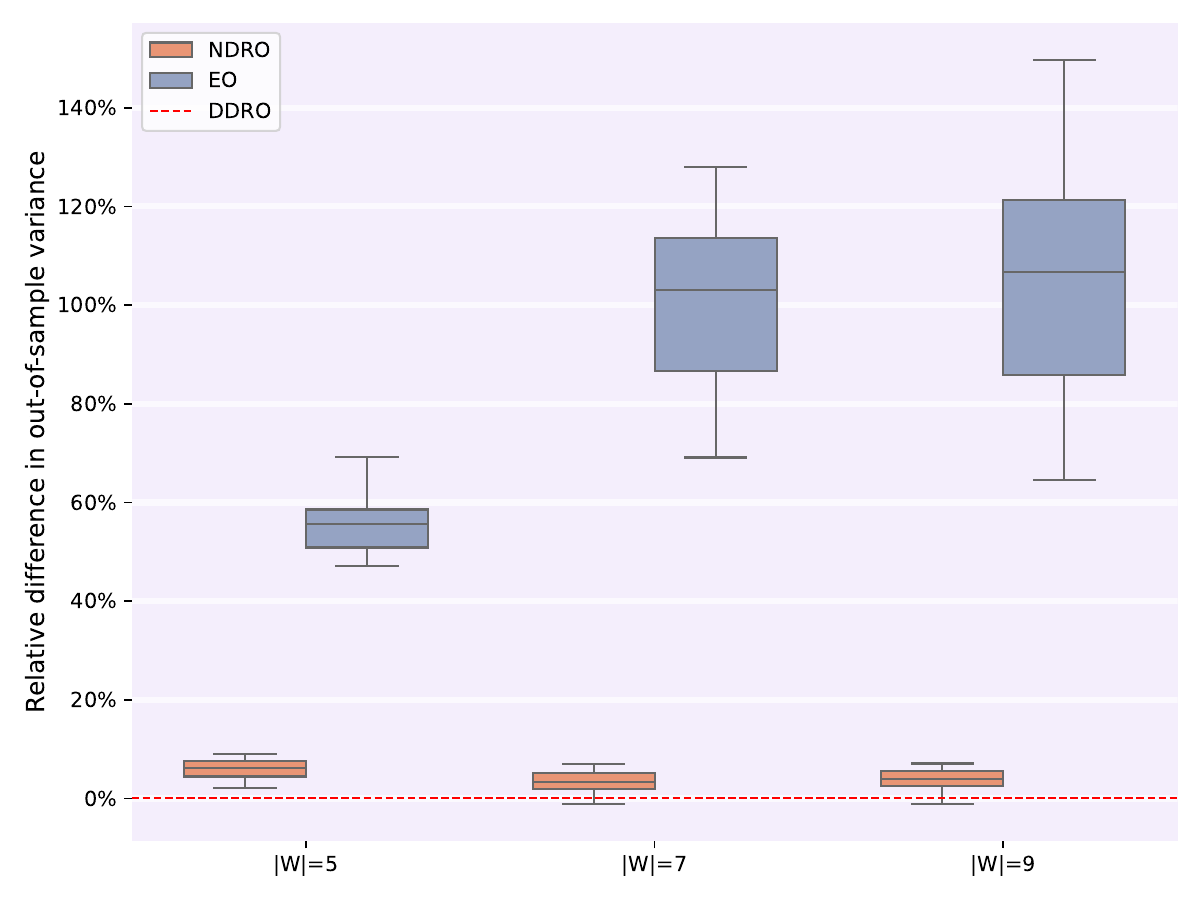}}
    \subcaptionbox{$X=900$}
    {\includegraphics[width=0.3\textwidth]{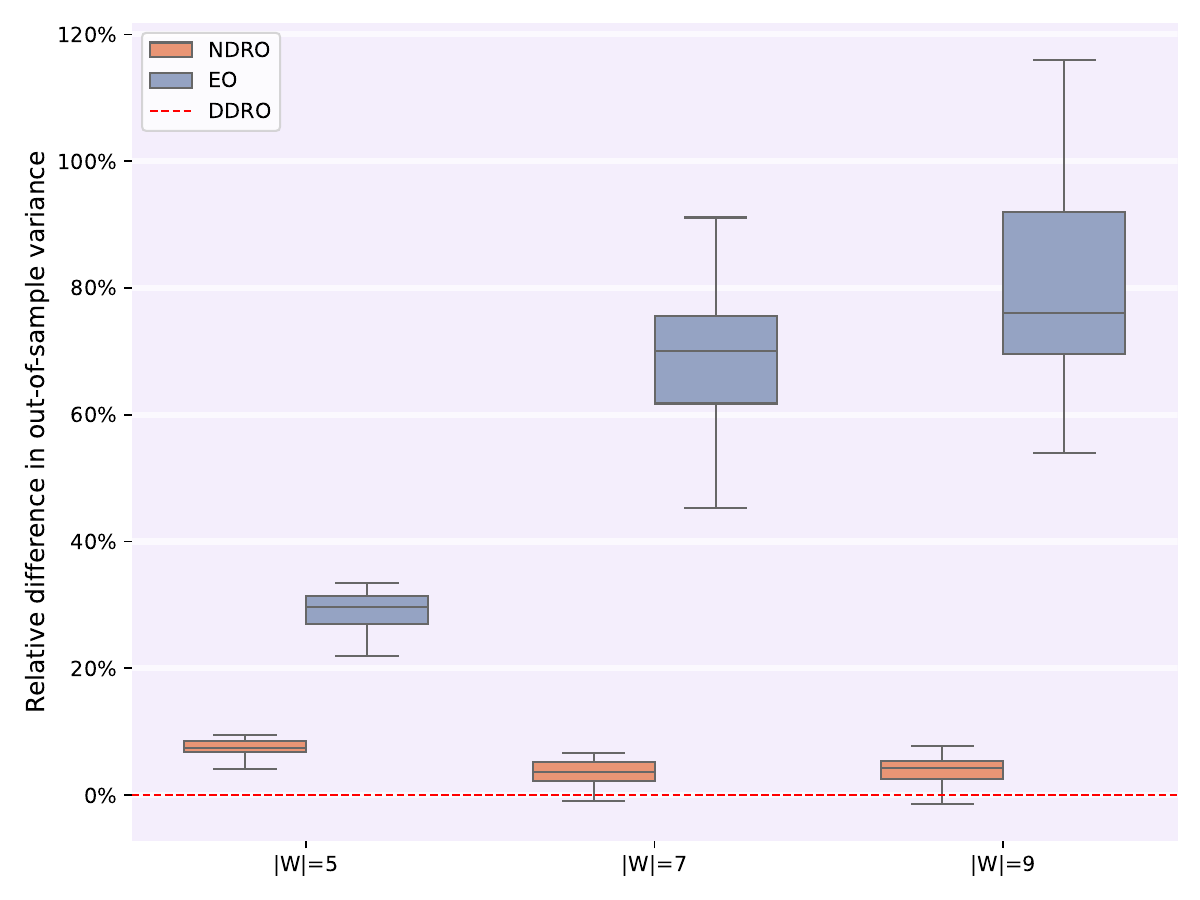}}
    }
    {
    Box Plots of Out-of-sample Smoothing Performance Using Traditional Parameters \label{fig:2}}
    {}
\end{figure}

To enable a fair comparison, we fix the total installed capacity $\sum_w x_w = X$. This setting is practical, as total capacity is typically constrained by investment budgets or minimum renewable penetration requirements. As shown in Figures~\ref{fig:1}, under traditional power system parameters, both DDRO and NDRO effectively reduce risk management costs while optimizing total cost.
Additionally, we use the out-of-sample variance of the aggregated wind power output to quantify the smoothing performance. The relative differences in out-of-sample smoothing performance are illustrated in Figure~\ref{fig:2}, where positive values indicate that DDRO achieves a certain percentage reduction in variance compared to benchmark methods. The results highlight that DDRO’s capacity allocation decisions consistently result in smaller variances and more stable aggregated wind power generation.

Notably, in this setting—where the objective remains minimizing total cost—the EO model shows weaker performance in risk control compared to its behavior in Section~\ref{sec:ne}. These results suggest that the DRO-based models, when minimizing total cost, tend to do so by prioritizing the reduction of risk management costs. Nevertheless, directly comparing risk management costs in this case is not entirely fair, since EO may provide benefits in investment and generation costs. 

To isolate and fairly assess the models’ risk management capabilities, we modify the parameter settings in Section~\ref{sec:ne} so that the optimization objective more directly emphasizes risk-related performance. Specifically, we assume uniform unit investment costs and wind power forecasts across all wind farms, identical reserve and adjustment costs for all thermal units, and sufficiently large transmission line capacities. Under these settings, fixing the total installed capacity ensures that investment and generation costs—typically the dominant components of total cost—remain constant across all models. This design allows for a clearer comparison of risk management performance and highlights the advantages of the DDRO and NDRO models in managing uncertainty.

\section{Data Generation}\label{sec:dg}

The data generation process is designed to produce non-negative samples that primarily preserve the marginal distributions of the original features, as well as their mean vector and covariance matrix. The methodology involves several key stages:
\begin{enumerate}[label=\arabic*.]
    \item Initialization \& Marginal Transformation: The mean vector and covariance matrix of the real data are computed. A QuantileTransformer is then fitted to the real data to map its marginal distributions to an approximately standard normal distribution. The correlation matrix of these Gaussianized data is determined to serve as the target covariance for subsequent sampling.

    \item Gaussian Sampling \& Inverse Transformation: New samples are drawn from a multivariate normal distribution parameterized by a zero mean and the correlation matrix derived in the previous step. These new Gaussian samples are then inverse-transformed using the fitted QuantileTransformer, mapping them back to the scale and marginal distributions characteristic of the real data.

    \item Iterative Non-negativity Enforcement \& Covariance Matching: An iterative refinement procedure is applied. In each iteration:
    \begin{enumerate}[label=\alph*.]
        \item The current generated data points are truncated at zero to enforce the non-negativity constraint.
        \item A 'covariance matching' step adjusts the mean and covariance of these non-negative samples. This involves centering the data, whitening it using the Cholesky decomposition of its current covariance matrix, and then re-correlating it using the Cholesky decomposition of the real data's covariance matrix and shifting it by the real mean. This aims to restore the target covariance structure that might have been altered by the non-negativity enforcement.
    \end{enumerate}

    \item Finalization: After a predefined number of iterations, a final truncation at zero ensures all output values are strictly non-negative. This iterative approach seeks to balance the non-negativity constraint with the preservation of the real data's correlational structure and marginal distribution characteristics, though some minor deviation in the final covariance may occur due to the non-negativity imposition.
\end{enumerate}

\begin{equation*}
    Correlation \  matrix=
    \begin{bmatrix}
        1.0000      & 0.3125     & 0.2991     & 0.4763 \\
        0.3125      & 1.0000     & 0.2688     & 0.3865 \\
        0.2991      & 0.2688     & 1.0000     & 0.4822 \\
        0.4763      & 0.3865     & 0.4822     & 1.0000
    \end{bmatrix}
\end{equation*}

\begin{figure}[htbp]
    \FIGURE
    {
    \subcaptionbox{Real data}
    {\includegraphics[width=0.4\textwidth]{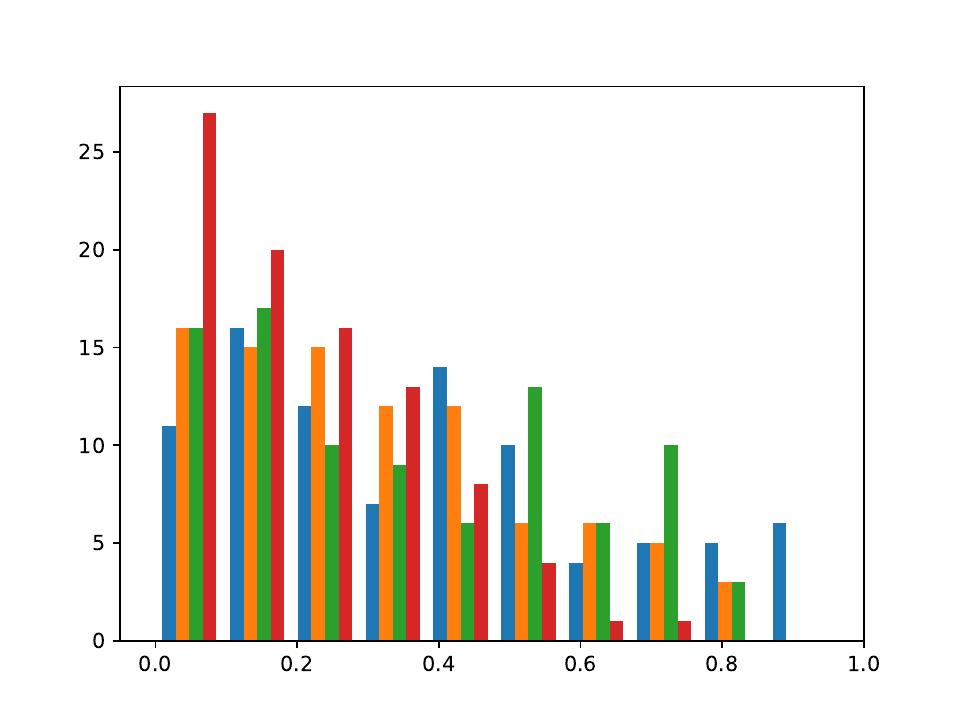}}
    \subcaptionbox{Generated data}
    {\includegraphics[width=0.4\textwidth]{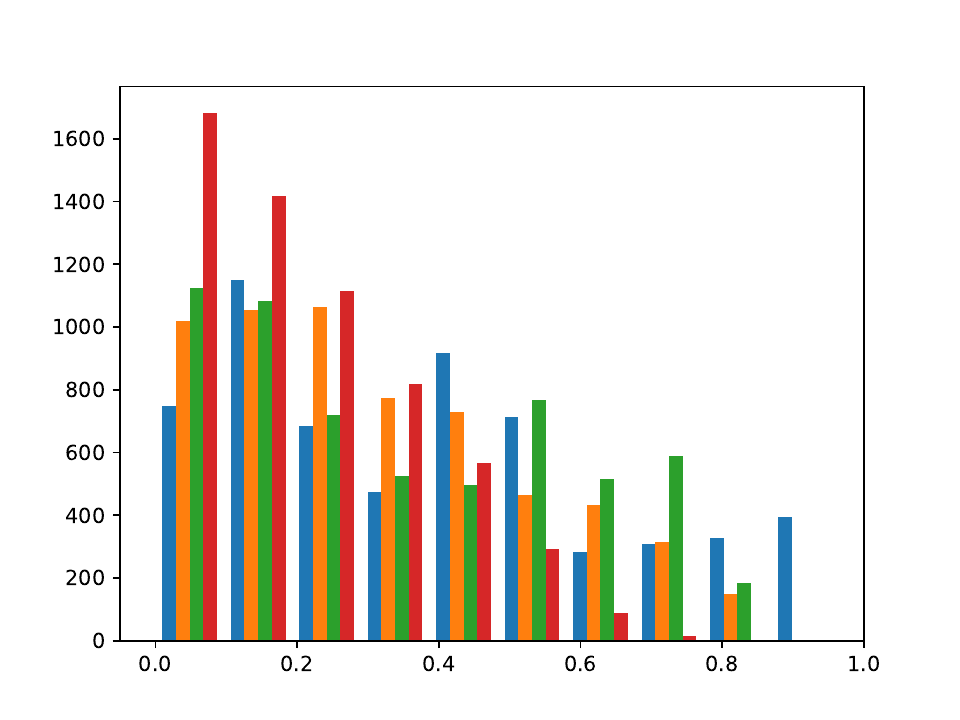}}
    }
    {
    Histograms of Real Data and Generated Data \label{fig:data generation}}
    {}
\end{figure}

\end{APPENDICES}

\end{document}